\documentclass{compositio}

\usepackage[centertags]{amsmath}
\usepackage{amsfonts}
\usepackage{amssymb}
\usepackage{amsthm}
\usepackage{amscd}
\usepackage{hyperref}
\usepackage{tocloft}    

\usepackage[centertags]{amsmath}
\usepackage{amsfonts}
\usepackage{amssymb}
\usepackage{amsthm}
\usepackage{amscd}
\usepackage{hyperref}
\usepackage{algorithmic, algorithm}
\usepackage[all]{xy}

\renewcommand{\contentsname}{Contents}

\DeclareFontFamily{OT1}{rsfs}{}
\DeclareFontShape{OT1}{rsfs}{n}{it}{<-> rsfs10}{}
\DeclareMathAlphabet{\mathscr}{OT1}{rsfs}{n}{it}

\newcommand{\tVtau}{\delta_{V}}
\newcommand{\tWtau}{\delta_{W}}
\newcommand{\qtau}{q}
\newcommand{\Stau}{S}
\newcommand{\Stauhat}{\widehat{\Stau}}
\newcommand{\Ttau}{T}
\newcommand{\Ttauhat}{\widehat{\Ttau}}
\newcommand{\Ttauc}{T_{c}}
\newcommand{\Btau}{B}
\newcommand{\Btauhat}{\widehat{\Btau}}
\newcommand{\Gtau}{G}
\newcommand{\Ktau}{K}

\newcommand{\Ftau}{k_0}

\newcommand{\rhotilde}{\widetilde{\rho}}

\newcommand{\cO}{\mathcal{O}}

\newcommand{\p}{\mathfrak{p}}

\newcommand{\cH}{\mathcal{H}}
\newcommand{\cA}{\mathcal{A}}

\newcommand{\cL}{\mathcal{L}}

\newcommand{\ds}{\displaystyle}

\newcommand{\cB}{\mathcal{B}}
\newcommand{\cE}{\mathcal{E}}

\newcommand{\et}{\text{\rm\'et}}

\newcommand{\ra}{\rightarrow}
\newcommand{\xra}[1]{\xrightarrow{#1}}
\newcommand{\hra}{\hookrightarrow}

\newcommand{\Gm}{\mathbf{G}_m}

\newcommand{\mthree}[9]{\left [
        \begin{matrix}#1&#2&#3\\#4&#5&#6\\#7&#8&#9
        \end{matrix}\right ]}
\newcommand{\mtwo}[4]{\left(
        \begin{matrix}#1&#2\\#3&#4
        \end{matrix}\right)}

\newcommand{\vphi}{\varphi}

\newcommand{\Zhat}{\widehat{\Z}}

\newcommand{\comment}[1]{}
\newcommand{\Q}{\mathbf{Q}}

\newcommand{\R}{\mathbf{R}}

\newcommand{\cC}{\mathcal{C}}

\newcommand{\cS}{\mathcal{S}}
\newcommand{\cZ}{\mathcal{Z}}

\newcommand{\C}{\mathbf{C}}

\newcommand{\bc}{\mathbf{c}}

\newcommand{\Qbar}{\overline{\Q}}

\newcommand{\cT}{\mathcal{T}}

\newcommand{\Z}{\mathbf{Z}}

\newcommand{\A}{\mathcal{A}}

\newcommand{\isom}{\cong}

\newcommand{\Af}{\mathbf{A}_f}

\DeclareFontEncoding{OT2}{}{} 

\renewcommand{\H}{\HH}

\newcommand{\Xtilde}{\widetilde{X}}
\newcommand{\Ktilde}{\widetilde{K}}

\newcommand{\Glochat}{\widehat{G}}
\newcommand{\Tlochat}{\widehat{T}}

\DeclareMathOperator{\cl}{cl}

\DeclareMathOperator{\proj}{proj}
\DeclareMathOperator{\dist}{dist}

\DeclareMathOperator{\Spec}{Spec}

\DeclareMathOperator{\Tr}{Tr}

\DeclareMathOperator{\ab}{ab}
\DeclareMathOperator{\Aut}{Aut}

\DeclareMathOperator{\Fr}{Fr}
\DeclareMathOperator{\Ver}{Ver}

\DeclareMathOperator{\GL}{\mathbf{GL}}

\DeclareMathOperator{\Sp}{\mathbf{Sp}}

\DeclareMathOperator{\U}{\mathbf{U}}
\DeclareMathOperator{\SU}{\mathbf{SU}}
\DeclareMathOperator{\Gal}{Gal}
\DeclareMathOperator{\SL}{\mathbf{SL}}
\DeclareMathOperator{\SO}{\mathbf{SO}}

\DeclareMathOperator{\HH}{H}
\DeclareMathOperator{\G}{\mathbf{G}}
\DeclareMathOperator{\Gtilde}{\widetilde{\mathbf{G}}}
\DeclareMathOperator{\Hbf}{\mathbf{H}}
\DeclareMathOperator{\Bbf}{\mathbf{B}}
\DeclareMathOperator{\Tbf}{\mathbf{T}}

\DeclareMathOperator{\Nbf}{\mathbf{N}}

\DeclareMathOperator{\Zbf}{\mathbf{Z}}

\DeclareMathOperator{\Sbf}{\mathbf{S}}

\DeclareMathOperator{\Ghat}{\widehat{\mathbf{G}}}

\DeclareMathOperator{\diag}{diag}
\DeclareMathOperator{\Hom}{Hom}

\DeclareMathOperator{\Res}{Res}

\DeclareMathOperator{\CH}{CH}
\DeclareMathOperator{\Sh}{Sh}
\DeclareMathOperator{\End}{End}

\DeclareMathOperator{\inv}{inv}
\DeclareMathOperator{\GU}{\mathbf{GU}}
\DeclareMathOperator{\AJ}{AJ}
\DeclareMathOperator{\ur}{ur}
\DeclareMathOperator{\CM}{CM}

\DeclareMathOperator{\Stab}{Stab}
\DeclareMathOperator{\pr}{pr}

\DeclareMathOperator{\cyc}{cyc}

\DeclareMathOperator{\rec}{rec}
\DeclareMathOperator{\der}{der}
\DeclareMathOperator{\ad}{ad}
\DeclareMathOperator{\Art}{Art}

\DeclareMathOperator{\Hyp}{\mathbf{Hyp}}
\DeclareMathOperator{\Inv}{\mathbf{Inv}}

\DeclareMathOperator{\Supp}{Supp}

\DeclareMathOperator{\geom}{geom}

\DeclareMathOperator{\IH}{\mathbf{IH}}

\theoremstyle{plain} 
\newtheorem{thm}{Theorem}[section] 
\newtheorem{conj}{Conjecture}[section]
\newtheorem{prop}[thm]{Proposition}
\newtheorem{cor}[thm]{Corollary}
\newtheorem{lem}[thm]{Lemma}

\theoremstyle{definition} 
\newtheorem{defn}[thm]{Definition} 
\theoremstyle{remark} 
\newtheorem{rem}{Remark}

\newcommand{\leftexp}[2]{{\vphantom{#2}}^{#1}{#2}}

\newcounter{tasknumber}
\newcommand{\task}[2][]{%
  \addtocounter{tasknumber}{1}%
  \begin{center}%
  \framebox[1.1\width]{\begin{minipage}{0.9\textwidth}%
  \textbf{Task \arabic{tasknumber}} \textit{\if!#1(unassigned)!\else (#1)\fi}: {#2}%
  \end{minipage}}%
  \end{center}%
}

\newcounter{assumptionnumber}
\newcommand{\assumption}[2][]{%
  \addtocounter{assumptionnumber}{1}%
  \begin{center}%
  \framebox[1.1\width]{\begin{minipage}{0.9\textwidth}%
  \textbf{Assumption \arabic{assumptionnumber}} \textit{\if!#1!\else (#1)\fi}: {#2}%
  \end{minipage}}%
  \end{center}%
}

\newcommand{\authnote}[2][]{\noindent {\if!#1!  {\bf TODO} \else {\small \bf #1} \fi: #2} \vspace{0.1in}}
\newcommand{\dimnote}[1]{{\authnote[Dimitar]{\textbf{\small #1}}}}

\makeatletter
\def\Ddots{\mathinner{\mkern1mu\raise\p@
\vbox{\kern7\p@\hbox{.}}\mkern2mu
\raise4\p@\hbox{.}\mkern2mu\raise7\p@\hbox{.}\mkern1mu}}
\makeatother

\title[]{Hecke and Galois Properties of Special Cycles on Unitary Shimura Varieties}
 
\author{Dimitar Jetchev}
\email{dimitar.jetchev@epfl.ch}
\address{Ecole Polytechnique F\'ed\'erale de Lausanne, Switzerland}

\thanks{}

\begin{document}

\begin{abstract}
We define and study a collection of special cycles on certain non-PEL Shimura varieties for $\U(2,1) \times \U(1,1)$ that appear naturally in the context of the conjectures of Gan, Gross and Prasad on restrictions of automorphic forms for unitary groups and conjectural generalizations of the Gross--Zagier formula. We express the Galois action in terms of the distance function on the Bruhat--Tits buildings for these groups. In addition, we calculate explicitly the Hecke polynomial appearing in the congruence relation conjectured by Blasius and Rogawski. 
Using the action of the local Hecke algebra on the Bruhat--Tits building, we establish explicit relations (distribution relations) between the Hecke action and the Galois action on the special cycles. These relations yield a new Euler system that can be used to study Selmer groups for certain Galois representations associated to automorphic forms on unitary groups and prove new instances of the Bloch--Kato--Beilinson conjecture.  
\end{abstract}
 
\maketitle


\section{Introduction}
\subsection{Motivation}
In \cite{gross:heegner-representation}, Gross outlines a program to link automorphic $L$-functions with special cycles on Shimura varieties via the Gross--Prasad restriction problems for automorphic representations. A basic case is the case of classical Heegner points on modular curves and restrictions of automorphic representations on $\GL_2$ to a non-split torus associated to an imaginary quadratic field. Subsequently, the work of Gan, Gross and Prasad 
provides Gross--Zagier type conjectures for classical groups \cite[\S 26--27]{gan-gross-prasad} relating two major open questions in number theory: the Birch and Swinnerton-Dyer conjecture (and its generalizations to higher dimensions via the Bloch--Kato--Beilinson conjectures) and the Langlands reciprocity conjectures.

It is thus of interest to study whether the Gross--Zagier type conjectures from \cite{gan-gross-prasad}
imply new results towards the Bloch--Kato--Beilinson conjectures. Such a program will aim at generalizing Kolyvagin's proof of the Birch and Swinnerton-Dyer conjecture for the case when the analytic rank of the elliptic curve is at most one \cite{kolyvagin:euler_systems, gross:kolyvagin} and provide a more conceptual representation-theoretic understanding of the latter.  
To achieve that, one needs an Euler system similar to Kolyvagin's Euler system of Heegner points. Since the Heegner point analogue of \cite[Conj.27.1]{gan-gross-prasad} is a higher-dimensional cycle on a Shimura variety, one could hope for an Euler system constructed from similar special cycles, but defined over increasing abelian extensions of the reflex field. Using $p$-adic Abel--Jacobi maps, one can obtain cohomology classes in the appropriate Selmer groups of geometric $p$-adic Galois representations appearing in the cohomology of Shimura varieties associated to unitary groups and subsequently, apply Kolyvagin's method to these classes. 

\subsection{Main results}
This article carries out the construction of an Euler system for higher rank unitary groups. More precisely, we define a collection of special cycles and study their Hecke and Galois properties. Establishing the Euler system relations is achieved via a comparison of the two actions (distribution relations). 
We note that lots of recent progress has been made towards the Gross--Zagier type conjecture 
\cite[Conj.27.1]{gan-gross-prasad} starting with the work of W. Zhang on the relative trace formula and a conjectural arithmetic fundamental lemma \cite{zhang:afl}. The latter has been proven in \cite{zhang:afl} in the case $\U(2,1) \times \U(1,1)$. Partial progress has been made by W. Zhang, Rapoport and Terstiege in the general case $\U(n, 1) \times \U(n-1, 1)$ \cite{rapoport-terstiege-zhang}. Although the main results of this paper are for $n=3$, we do some of the computations for arbitrary $n$ as those will be used in forthcoming work. 

\paragraph{Hermitian spaces, unitary groups, Shimura varieties and special cycles.} Let $F$ be a totally real number field with $[F : \Q] = d$ and let $E/F$ be a totally imaginary quadratic extension with non-trivial automorphism denoted by $x \mapsto \overline{x}$ for $x \in E$. Let $\rho_1, \dots, \rho_d$ be the real 
places of $F$. Choose an embedding $\rhotilde_1 \colon E \hra \C$ that extends the place 
$\rho_1 \colon F \hra \R$.  Moreover, fix embeddings 
$\iota_\tau \colon \overline{E} \hra \overline{E}_\tau$ for every finite place $\tau$ of $E$.

Let $n \geq 3$ be an odd integer and let $(V, \langle\,,\rangle)$ be a non-degenerate Hermitian space of dimension $n$ over $E$. 
Suppose that $V$ has signature $(n-1, 1)$ at $\rho_1$ and signatures $(n, 0)$ at each of the places 
$\rho_2, \dots, \rho_d$. Let $W \subset V$ be a Hermitian subspace of dimension 
$n-1$ that has signature $(n-2, 1)$ at $\rho_1$ and signatures $(n-1, 0)$ at 
$\rho_2, \dots, \rho_d$. Let $D \subset V$ be the $E$-line that is the orthogonal complement of $W$ with respect to the Hermitian form, i.e., for which $V = W \perp D$. 

Associated to $V$ and $W$ are the groups of unitary isometries $\U(V)$ and $\U(W)$, respectively, defined over $F$. We view $\Hbf = \Res_{F/\Q} \U(W)$ as an algebraic subgroup of $\G = \Res_{F/\Q} ( \U(V) \times \U(W))$ via the diagonal embedding (that is, the natural embedding $\U(W) \hra \U(V)$ on the first factor\footnote{Here, each unitary isometry for $W$ is extended to a unitary isometry for $V$ via the identity on $D$.} and the identity map on the second factor). 

We assume (without losing generality) that $D$ contains a vector $e_D \in D$ with $\langle e_D, e_D\rangle = 1$ (this means that $L_D = \cO_E e_D$ is a global self-dual $\cO_E$-lattice in $D$). Indeed, if not, take 
any $e_D \in D$ and let $\langle e_D, e_D \rangle = \lambda \in F^\times$ (such an $e_D$ exists since 
$\langle \,,\rangle$ is non-degenerate). By rescaling the hermitian pairing with $\lambda^{-1}$, the unitary groups are unchanged\footnote{From the point of view of Shimura varieties and special cycles which is the goal of this paper, rescaling the hermitian form by a scalar in $F^\times$ will not affect any of the geometric objects of study.}; yet, $\langle e_D, e_D \rangle = 1$. Fix now any $\cO_E$-lattice 
$L_W \subset W$ of full rank for which $L_W \subset L_W^{\vee}$ and consider the $\cO_E$-lattice 
$L_V = L_W \oplus L_D \subset V$ (an integral structure or a lattice of full rank for $V$). 

Associated to the $\Q$-algebraic groups $\Hbf$ and $\G$ are Shimura data 
$(\Hbf, Y)$ and $(\G, X)$ introduced in Section~\ref{subsec:shimvar}. We also introduce (again in Section~\ref{subsec:shimvar}) some compact open subgroups $K_{\Hbf} \subset \Hbf(\Af)$ and $K \subset \G(\Af)$ (obtained from the integral structures $L_V$ and $L_W$) where $\Af$ denotes the finite ad\`eles of $\Q$. 
These data give rise to Shimura varieties $\Sh_{K_{\Hbf}}(\Hbf, Y)$ and $\Sh_{K}(\G, X)$ with reflex fields $E$ and a natural diagonal cycle $\Sh_{K_{\Hbf}}(\Hbf, Y) \hra \Sh_{K}(\G, X)$. 
Considering $\G(\Af)$-translates of a connected component of the small Shimura variety $\Sh_{K_{\Hbf}}(\Hbf, Y)$ yields a collection of special cycles $\cZ_K(g) \subset \Sh_{K}(\G, X)$ for $g \in \G(\Af)$ defined (by Shimura reciprocity laws) over abelian extensions of $E$ (see Section~\ref{subsec:speccyc}). These cycles are higher-dimensional analogues of higher Heegner points (see \cite{gross:heegner_points} and \cite{gross:kolyvagin}). 

\paragraph{Galois properties of CM cycles.}\label{subsubsec:introgalprop} We first compute the field of definition of 
each cycle in terms of the distance function on the corresponding Bruhat--Tits buildings for $\U(V)$ and $\U(W)$ by describing the set $\cZ_K(\G, \Hbf)$ of special cycles and their Galois orbits adelically using reciprocity laws for the Galois action on the connected components for certain Shimura varieties associated to $\Hbf$. The latter implies 
that the orbits of the cycles under the decomposition group at a finite place $\tau$ of $F$ are in bijection with $H_\tau \backslash G_\tau /K_\tau$ (see Section~\ref{par:concomp} and Section~\ref{sec:galois-cycles}). Here, $G_{V, \tau} = \U(V)(F_\tau)$, 
$G_{W, \tau} = \U(W)(F_\tau)$, $G_\tau = G_{V, \tau} \times G_{W, \tau}$ and $H_\tau$ is the diagonal image of $G_{W, \tau}$ in $G_{\tau}$. 

\begin{defn}[(Allowable place of $F$)]
We call a finite place $\tau$ of $F$ of odd residue characteristic \emph{allowable} for the triple $(\G, \Hbf, K)$ if 
1) both 
$G_\tau$ and $H_\tau$ are quasi-split; 2) $K_\tau = K_{V, \tau} \times K_{W, \tau}$ where 
$K_{V, \tau} \subset G_{V, \tau}$ and $K_{W, \tau} := G_{W, \tau} \cap K_{V, \tau} \subset G_{W, \tau}$ are 
both hyperspecial maximal compact subgroups. 
According to the splitting behavior of $\tau$ in $E$, we call $\tau$ an allowable inert, split, or ramified place. 
\end{defn}

In this paper, we work locally at an allowable inert place $\tau$ of $F$ only\footnote{This case is the most important one from the point of view of applying Kolyvagin's arguments.} and treat the split and ramified places in forthcoming papers. The double quotient $H_\tau \backslash G_\tau / K_\tau$ is in bijection with the set of $H_\tau$-orbits $[L_{V, \tau}, L_{W, \tau}]$ of pairs $(L_{V, \tau}, L_{W, \tau})$ of self-dual local Hermitian $\cO_{E_\tau}$-lattices $L_{V, \tau} \subset V_\tau$ and $L_{W, \tau} \subset W_\tau$ where $V_\tau = V \otimes_E E_\tau$ and $W_\tau = W \otimes_E E_\tau$. In order to compute the completion at $\tau$ of the field of definition $E(\xi)$ of the cycle $\xi = \cZ_K(g)$, it suffices (by the reciprocity laws mentioned above) to compute the stabilizer of the corresponding pair in $H_\tau$. More precisely, the image under the determinant map of that stabilizer in $\U(1)(F_\tau)$ determines completely a norm subgroup of of $E_\tau^\times$ which, by local class field theory, determines this completion as an abelian extension of $E_\tau$. 
It turns out (see Section~\ref{subsec:loccond}) that this norm subgroup is of the form $\cO_c^\times 
\subset E_\tau^\times$ where $\cO_c = \cO_{F_\tau} + \varpi^c \cO_{E_\tau} \subset \cO_{E_\tau}$ 
is the local order of $\cO_{E_\tau}$ of conductor $\varpi^c$ for $\varpi \in \cO_{F_\tau}$ a uniformizer. 
We call $\bc_\tau([L_{V, \tau}, L_{W, \tau}]) = \varpi^c$ the local conductor at $\tau$. 
The computation of the local conductor is then given by the following: 


\begin{thm}[(Local Conductor Formula)]\label{thm:galois}
Let $n = 3$ and let $\tau$ be an allowable inert place of $F$. 
(i) The map 
$$
\inv_\tau \colon (L_{V, \tau}, L_{W, \tau}) \mapsto \left ( \dist (L_{V, \tau}, \pr_{W_\tau} (L_{V, \tau})),  \dist (\pr_{W_\tau} (L_{V, \tau}), L_{W_\tau}) \right )
$$
induces a bijection between the set of $H_\tau$-orbits on $G_\tau / K_\tau$ and the set $\Inv_\tau = \{(a, b) \colon a, b \geq 0\}$. Here, $\pr_{W_\tau} (L_{V, \tau})$ denotes the convex projection of the hyperspecial point corresponding to $L_{V, \tau}$ to the Bruhat--Tits building\footnote{It is a tree in this case and as we see from Fig.~\ref{fig:buildings}, the projection of any hyperspecial (black) point is hyperspecial (black) point as well.} and $\dist$ indicates 
the distance function\footnote{In this case, $\cB(V_\tau)$ is the tree on Fig.~\ref{fig:buildings} and the distance is the usual distance in the sense of a tree where we normalize so that the distance between a black (hyperspecial) and a white (special, but not hyperspecial) vertices is 1/2.} on the building $\cB(V_\tau)$ of $G_{V, \tau}$.  

\noindent (ii) Given a $H_\tau$-orbit $[(L_{V, \tau}, L_{W, \tau})]$ of pairs of lattices $(L_{V, \tau}, L_{W, \tau}) \in \cL_\tau$ with $\inv_\tau([(L_{V, \tau}, L_{W, \tau})]) = (a, b)$, the local conductor is 
\begin{equation}\label{eq:localcond}
\bc_\tau([L_{V, \tau}, L_{W, \tau}]) = \mathfrak \varpi^{\min \left \{ a, 2b \right \}}. 
\end{equation}
\end{thm}

%
%
\comment{
\begin{rem}
We expect that a similar statement holds for general $n$; yet, the proof is technically more involved. The main 
obstruction is that, although the convex projection is still well-defined (using, e.g., $p$-adic norms similarly to \cite{goldman-iwahori}), the buildings are higher-dimensional and more difficult to visualize. For these higher-dimensional buildings, it is not even clear which distance function should be used, although a reasonable candidate is the \emph{gallery distance}.  
\end{rem}
}


\paragraph{Blasius--Rogawski congruence relation.} 
Blasius and Rogawski \cite{blasius-rogawski:zeta} formulate the \emph{congruence relation} conjecture generalizing the classical Eichler--Shimura relation for modular curves by providing an explicit polynomial, the Hecke polynomial, annihilating the geometric Frobenius $\Fr_\tau$ acting on the $\ell$-adic \'etale cohomology (we normalize so that geometric Frobenii correspond to uniformizers under the Artin map). 

In our setting, let $p$ be the rational prime below the allowable inert place 
$\tau$. For $\star \in \{\varnothing, V, W\}$, let $\cH_{\star, p} = \cH(G_{\star, p}, K_{\star, p})$ be the local Hecke algebra at $p$ where $G_{\star, p} = \G_{\star}(\Q_p)$. 
If $\star \in \{V, W\}$ and under the stronger assumption that the whole $\ds K_{\star, p} = \prod_{\tau \mid p} K_{\star, \tau}$ $\tau \mid p$ is hyperspecial,  
Blasius and Rogawski define\footnote{For $\GL_2$, if $\tau = p \ne \ell$, the Hecke polynomial is simply the polynomial $H_p(z) = X^2 - T_p X + p$ and the classical Eichler--Shimura relation $T_p = \Fr_p + \textbf{Ver}_p$ is equivalent to the statement that $H_p$ vanishes on $\Fr_p$ acting on the $\ell$-adic Tate module $T_\ell \cE$ of an elliptic curve $\cE$.} (following Langlands) a polynomial 
$H_{\star, \tau}(z) \in \cH_{\star, p}[z]$ representation-theoretically out of the the Shimura datum (see \cite[\S 6]{blasius-rogawski:zeta} for the definition) that depends on $\tau$
but whose coefficients are in $\cH_{\star, p}$ (the algebra (under convolution) of $K_{\star, p}$-bi-invariant locally constant functions on 
$G_{\star, p}$). 

In Section~\ref{subsec:comphecke} we define a polynomials $H_{\star, \tau}(z)$ for $\star \in \{\varnothing, V, W\}$ with coefficients in $\cH_{\star, \tau}$ under the assumption that $\tau$ is allowable and inert (thus, not necessarily assuming that $K_{\star, p}$ is hyperspecial, but only $K_{\star, \tau}$). Under the natural injection
$$
\ds \cH_{\star, \tau} \hra \cH_{\star, \tau} \otimes \bigotimes_{\substack{\tau' \mid p \\ \tau' \ne \tau}} \cH_{\star, \tau'} = \cH_{\star, p}, \qquad t \mapsto t \otimes 1, 
$$ 
our polynomials $\cH_{\star, \tau}$ map to the polynomials defined by Blasius--Rogawski, thus, justifying the use of the same notation and also, the fact that we can view their coefficients in $\cH_{\star, \tau}$. 

Let $\ell$ be a prime such that $\tau \nmid \ell$. Let $\overline{\Sh_K(\G, X)}$ denote the Baily--Borel (minimal) compactification. Following Blasius and Rogawski \cite[p.33]{blasius-rogawski:zeta}, we state the the congruence relation on cohomology (the weaker conjecture) to include the case of the product Shimura variety:  


\begin{conj}[(Congruence relation on cohomology)]\label{thm:congrel}
Let $\tau$ be an allowable inert place and let $\ell$ be a prime such that $\tau \nmid \ell$. 
Then $\star \in \{\varnothing, V, W\}$, $\IH^*\left (\overline{\Sh_{K_{\star}}(\G_{\star }, X_{\star})}_{\Qbar}, \Q_\ell \right )$ is unramified at $\tau$ and 
$$
H_{\star, \tau}(\Fr_\tau) = 0, 
$$
where the last equality is considered in $\ds \End_{\Q_\ell} \left (\IH^*\left (\overline{\Sh_{K_{\star}}(\G_{\star }, X_{\star})}_{\Qbar}, \Q_\ell \right ) \right)$. 
\end{conj}


\begin{rem}
Koskivirta \cite{koskivirta:journal} verifies a related conjecture in the case of PEL-type unitary Shimura varieties closely related to ours in the case $F = \Q$ and $\star \in \{V, W\}$. Instead of working on intersection cohomology, he works on a certain moduli space for $p$-isogenies. \emph{A priori}, it is not automatic how one can pass to  cohomology (although the latter is known to experts). Yet, a stronger form of the congruence relation (on cycles) is needed for the application to Euler systems.  
\end{rem}

Our first contribution is to deduce the congruence relation for the product of two Shimura varieties from the congruence relation of the two factors. 

\begin{thm}\label{thm:congrelprod}
Let $(\G_1, X_1)$ and $(\G_2, X_2)$ be two Shimura data and let $(\G, X)$ be the product datum (i.e., $\G = \G_1 \times \G_2$ and $X = X_1 \times X_2$). 
Suppose that Conjecture~\ref{thm:congrel} holds for both $(\G_1, X_1)$ and $(\G_2, X_2)$. Then Conjecture~\ref{thm:congrel} holds for $(\G, X)$. 
\end{thm}

\noindent As an immediate corollary, we obtain the following: 

\begin{cor}\label{cor:product}
Suppose that Conjecture~\ref{thm:congrel} holds for $\Sh_{K_{\star}}(\G_{\star}, X_{\star})$ for $\star \in \{V, W\}$. Then it holds for $\Sh_K(\G, X)$. 
\end{cor}

\noindent The second contribution is deducing the conjecture for the non-PEL type Shimura variety $(\G_\star, X_\star)$ from the PEL-type one $(\widetilde{\G}_{\star}, \widetilde{X}_{\star} )$. We do this via a slightly more general argument for an arbitrary Shimura variety. 

\begin{thm}\label{thm:centraltwist}
Let $(\G_0, X_0)$ be a Shimura datum and let $\omega \colon \Gm \ra \Zbf_{\G_0} \hra \G_0$ be a central co-character. 
If Conjecture~\ref{thm:congrel} holds for $(\G_0, \omega X_0)$ then it holds for $(\G_0, X_0)$. 
\end{thm}

\noindent To deduce the conjecture for our Shimura datum $(\G, X)$ in the case $F = \Q$, we use the work of Koskivirta together with Theorem~\ref{thm:centraltwist} to first deduce the conjecture for both $(\G_V, X_V)$ and $(\G_W, X_W)$ and then use Corollary~\ref{cor:product}. 

\paragraph{Distribution relations.}
Kolyvagin's Euler systems method \cite{kolyvagin:euler_systems, rubin:book} models local $L$-factors algebraically via cohomological data.  
Constructing Euler systems amounts to proving certain norm-compatibility relations (distribution relations) on a certain space of special objects such as special units in the case of Kato's Euler system or special cycles in the case of Heegner points. 
For elliptic curves, it is known that the classical Heegner points $\{x_c\}$ (see \cite{gross:kolyvagin}) on the modular curve $X_0(N)$ for the imaginary quadratic field $E$ satisfy the property that if $p \nmid c$ is inert in $E$ then 
\begin{equation}\label{eq:basic}
T_p (x_c) = \Tr_{E[cp ] / E[c]} (x_{cp}) \in \Z[\CM],  
\end{equation}
where $\CM$ indicates the set of points of $X_0(N)$ having complex multiplication by $E$.   
This equality is proved in \cite[Prop.3.7(i)]{gross:kolyvagin} and is known as a distribution relation for Heegner points. 
Together with the Eichler--Shimura relation \cite[Prop.3.7(ii)]{gross:kolyvagin}, one gets an Euler system and derived cohomology classes that, via general global duality arguments, yield upper bounds on Selmer groups. 
Although rather simple, \eqref{eq:basic} is not very convenient when generalizing to Shimura varieties for higher-rank groups. An alternative way of restating the above equation is as follows: 
\begin{equation}\label{eq:gl2distrel}
(\Fr_p^2 - T_p \Fr_p + p) \Fr_p (x_c) = \Tr_{E[cp] / E[c]} (\Fr_p^{-1} (x_c) - (x_{cp})).  
\end{equation}
This is more convenient as the left-hand side is simply the operator that is the value of the Hecke polynomial $H_p(z) = z^2 - T_p z + p$ at $\Fr_p$ acting on an unramified CM point whereas the right-hand side is an exact trace. 
In fact, the analogue of the pair $(\G, \Hbf)$ of algebraic groups in this case is $(\GL_2, E^\times)$ (see~\cite{gross:heegner-representation} for the precise analogy from the point of view of the Gross--Prasad restriction problems).  

In the case of unitary groups, both the Hecke algebra $\cH(\G, K)$ and the Galois group $\Gal(E^{\ab} / E)$ act on $\Z[\cZ_K(\G, \Hbf)]$ as explained in Section~\ref{sec:distrel}. Given an element $\xi \in \Z[\cZ_K(\G, \Hbf)]$, let $E(\xi)$ be the smallest abelian extension of $E$ such that all special cycles in $\Supp(\xi)$ are defined over $E(\xi)$. Moreover, let $\bc_\tau(\xi)$ denote the local conductor of $E(\xi)$ at $\tau$. Using Theorem~\ref{thm:galois} expressing the Galois action in terms of the distance function on the building, we prove the following relation between the two actions: 


\begin{thm}{(Horizontal Distribution Relations)}\label{thm:horizdistrel}
Let $\tau$ be an allowable inert place of $F$ and let $\xi \in \cZ_K(\G, \Hbf)$ be a special cycle with $\inv_\tau(\xi) = (0, 0)$. Let $E^\times \left ( \cO_{E_\tau}^\times \times N^{(\tau)} \right ) \subset \widehat{E}^\times$ be the norm subgroup corresponding to the abelian extension $E(\xi) / E$ via class field theory, where  
$\ds N^{(\tau)} \subset \left ( \widehat{E}^{(\tau)} \right)^\times$. There exists an element $\xi' \in \Z[\cZ_K(\G, \Hbf)]$ whose field of definition $E(\xi')$ has an associated norm subgroup 
$E^\times \left ( \cO_{2}^\times \times N^{(\tau)} \right ) \subset \widehat{E}^\times$ such that the following distribution relation holds: 
\begin{equation}\label{eq:unitarydistrel}
H_{\tau} (\Fr_\tau) \xi = \Tr_{E(\xi') / E(\xi)} (\xi'). 
\end{equation}
\end{thm}

\begin{rem}
Here, $E(\xi')_\tau$ denotes the completion of $E(\xi')$ at the unique place of $E(\xi)$ above $\tau$ (note that 
$E(\xi')_\tau$ is a totally ramified extension of $E(\xi)_\tau = E_\tau$ of degree $q(q+1)$ where $q$ is the order of the residue field of $F$ at the place $\tau$). In other words, the local conductor $\bc_\tau(\xi') = \varpi^2$. 
\end{rem}

\begin{rem}
We note the analogy of $\H_\tau(\Fr_\tau)$ with the left-hand side of \eqref{eq:gl2distrel} except that the above theorem increases the local conductor by 2 as opposed to 1 in the $\GL_2$-case. This is not a problem for the arithmetic applications as one can always define norm-compatible cycles over an extension of local conductor one by taking traces of $\xi(\tau)$.  
\end{rem}

\begin{rem}
For arithmetic application to Iwasawa theory, one also needs \emph{vertical} distribution relations where the conductors of the ring class extensions vary $p$-adically. Recently, such vertical relations have been established in \cite{boumasmoud-brooks-jetchev} using the cycles $\cZ_K(\G, \Hbf)$ and Theorem~\ref{thm:galois}. We expect that our Euler system can be used to prove new results towards one divisibility of the anticyclotomic main conjecture of Iwasawa theory for the relevant Galois representations.
\end{rem}

\begin{rem}
The construction of Euler systems for higher rank groups has been initiated by Cornut \cite{cornut:normes1, cornut:normes2} for the case of $\U(n) \subset \SO(2n+1)$. Our setting matches the setting of Gan, Gross and Prasad \cite[\S 27]{gan-gross-prasad} where there is already an explicit Gross-Zagier type conjecture.  
\end{rem}

\subsection{Outline of the article}
We introduce the setting for unitary groups, Hermitian lattices, the relevant Shimura data, Shimura varieties and special cycles in detail in Section~\ref{sec:shimura-cycles}. We prove Theorem~\ref{thm:galois} in Section~\ref{sec:galois-cycles} by studying the action 
of the small group $H_\tau$ on the product $\cB(V_\tau) \times \cB(W_\tau)$ of the buildings for the groups $G_{V, \tau}$ and $G_{W, \tau}$. In Section~\ref{sec:heckepol}, we compute the Hecke polynomial for our unitary groups using a method of Cornut and Koskivirta \cite{koskivirta:journal} reducing the computation to local combinatorics on the Bruhat--Tits buildings via \emph{canonical retraction maps} on buildings. In Section~\ref{sec:cong-blas-rog}, we prove Theorems~\ref{thm:congrelprod} and~\ref{thm:centraltwist} and combine these with the recent results of J.-S. Koskivirta in the case $F = \Q$ to get Conjecture~\ref{thm:congrel} in our case. Finally, we prove Theorem~\ref{thm:horizdistrel} in Section~\ref{sec:distrel} via a local combinatorial argument using Theorem~\ref{thm:galois}.  

\section{Shimura Varieties and Special Cycles}\label{sec:shimura-cycles}
\setcounter{paragraph}{0}

\subsection{Unitary Groups}\label{subsec:unitary}


\paragraph{Unitary groups of isometries and similitudes.}
Let $\star \in \{V, W\}$ and define the algebraic group of unitary isometries $\U(\star)$ over $F$ via
$$
\U(\star)(R) = \{g \in \GL(\star)(R \otimes_F E) \colon \forall x, y \in \star \otimes R, \langle gx, gy\rangle = \langle x, y\rangle \},  
$$
where $R$ is any $F$-algebra.  
Let $\G_{\star} = \Res_{F/\Q} \U(\star)$, let $\G = \G_V \times \G_W$ and let $\Hbf = \G_W$, the latter viewed as an algebraic subgroup of $\G$ via the diagonal embedding. We also consider the groups of unitary similitudes $\GU(\star)$ defined over $F$ defined by  
$$
\GU(\star )(R) = \{g \in \GL(\star)(R \otimes_F E) \colon \exists \nu(g) \in R^\times,\ \forall x, y \in \star \otimes R, \langle gx, gy\rangle = \nu(g) \langle x, y\rangle \},    
$$
for any $F$-algebra $R$. 
The similitude factor $\nu \colon \GU(\star)(R) \ra R^\times$ is a homomorphism and $\U(\star)(R) = \ker(\nu)$. Throughout, let $\Gtilde_{\star}$ be the $\Q$-reductive group defined by the fiber product 
\[
\xymatrix {
\Gtilde_{\star} \ar[r]^{\nu} \ar[d] & \mathbf{G}_{m, \Q} \ar[d] \\
\Res_{F/\Q} \GU(\star) \ar[r]^{\nu} & \Res_{F /\Q} \mathbf{G}_{m, F}. \\ 
}
\]
Let $\Gtilde =\Gtilde_V \times \Gtilde_W$.  

\paragraph{Lattices and self-dual lattices in hermitian spaces.} 
Let $\tau$ be an inert place of $F$. Let $\mathscr V$ be an hermitian space over $E_\tau$. An $\cO_{E_\tau}$-lattice in $\mathscr V$ will be an $\cO_{E_\tau}$-submodule of $\mathscr L \subset \mathscr V$ of full rank. Given an $\cO_{E_\tau}$-lattice 
$\mathscr L \subset \mathscr V$, define its \emph{dual lattice} $\mathscr L^\vee \subset \mathscr V$ by 
$$
\mathscr L^\vee = \{v \in \mathscr V \colon \langle v, \mathscr L \rangle \subseteq \cO_{E_\tau} \}. 
$$
A lattice $\mathscr L \subset \mathscr V$ is \emph{self-dual} if $\mathscr L^\vee = \mathscr L$. Not every local hermitian space contains self-dual lattices. Recall that a local hermitian space of even dimension $2m$ is called split if it is a sum of $m$ mutually orthogonal hyperbolic planes and that a local hermitian space of odd dimension $2m+1$ is called split if it is a sum of $m$ mutually orthogonal hyperbolic planes and an anisotropic line. 

For $\star \in \{V, W\}$, if $L_{\star} \subset \star$ is the global lattice (integral structure) introduced in 
Section~\ref{subsubsec:introgalprop}, we obtain a local lattice 
$L_{\star, \tau}  = L_\star \otimes_{\cO_E} \cO_{E_\tau} \subset \star_\tau$ for any finite place $\tau$ of $E$ and an adelic lattice $\widehat{L}_{\star} = L \otimes_{\cO_E} \widehat{\cO_E} \subset \widehat{\star}$ where $\widehat{\cO_E} = \cO_E \otimes \Zhat$ and $\widehat{\star} = \star \otimes \widehat{\Q}$. 

\begin{lem}\label{lem:locherm}
For all but finitely many places $\tau$ of $F$ that are inert in $E$ the spaces $V_\tau$ and $W_\tau$ are both split and the lattices $L_{V,\tau}$ and $L_{W, \tau}$ are both self-dual in these spaces. 
\end{lem}

\begin{proof}
Fix an $E$-basis $\{v_i\}$ for $V$.  
Since local hermitian $E_\tau$-spaces are classified by their dimension and discriminant (an element of $F_\tau^\times / N E_\tau^\times$), it follows that for all but finitely many inert places $\tau$, the local determinant $\ds \det \left ( \langle v_i, v_j \rangle_\tau \right )_{i, j  = 1}^n \in \cO_{F_\tau}^\times$. As long as $\tau$ is unramified in $E$, the local norm map $N \colon \cO_{E_\tau}^\times \ra \cO_{F_\tau}^\times$ will be surjective (see, e.g., \cite[II \S 4]{lang:ant}) and hence, local discriminant will be the trivial element of $F_\tau^\times / N E_\tau^\times$, i.e., $V_\tau$ will be isomorphic to the split space. The argument for $W_\tau$ is similar. Now, it is not hard to check (using the hypothesis that $\langle e_D, e_D \rangle = 1$), that for all, but finitely many of these places, both $L_{V, \tau}$ and $L_{W, \tau}$ will be self-dual (it is sufficient to avoid places dividing $[L_V^\vee : L_V]$ and $[L_W^\vee : L_W]$). 
\end{proof}

For $\star \in \{V, W\}$ and for each finite place $\tau$ of $F$, the group 
$G_{\star, \tau}$ acts transitively on the set $\cL(\star_\tau)$ of self-dual local $\cO_{E_\tau}$-lattices in $\star_\tau$. As we explain later, $\cL(\star_\tau)$ corresponds to the set of hyperspecial vertices for the Bruhat--Tits buildings for the local unitary group $G_{\star, \tau}$. 

\paragraph{Self-dual lattices adapted to a decomposition.} For an inert place $\tau$ of $F$, consider an orthogonal decomposition 
\begin{equation}\label{eq:dec}
\mathscr V = \mathscr V' \perp \mathscr V'', 
\end{equation}
of a non-degenerate split local Hermitian $E_\tau$-space $\mathscr V$, where $\mathscr V'$ and 
$\mathscr V''$ are non-degenerate split local Hermitian $E_\tau$-vector 
subspaces (for the restriction of the Hermitian form). If $\mathscr L \subset \mathscr V$ is a self-dual local Hermitian 
$\cO_{E_\tau}$-lattice then we say that the decomposition 
\eqref{eq:dec} is \emph{adapted} to $\mathscr L$ (or that $\mathscr L$ is \emph{adapted} to \eqref{eq:dec}) if 
$\mathscr L \cap \mathscr V'$ is a self-dual $\cO_{E_\tau}$-lattice of $\mathscr V'$. Note that 
$$
\mathscr L \cap \mathscr V' \text{ is a self-dual }\cO_{E_\tau}\text{-lattice of }\mathscr V' \Longleftrightarrow \mathscr L \cap \mathscr V'' \text{ is a self-dual } \cO_{E_\tau}\text{-lattice of }\mathscr V''. 
$$

\paragraph{Integral structures.}\label{par:intstruct}
Let $\tau$ be an inert place of $E$ that satisfies the condition of Lemma~\ref{lem:locherm} and let 
$\star \in \{V, W\}$. The global lattices $L_{\star} \subset \star$ gives rise to a compact open subgroup 
$\Stab_{\G_{\star}(\Af)} (\widehat{L}_{\star} ) \subset \G_{\star}(\Af)$. 
The compact open subgroup $K_{\star, \tau} = \Stab_{G_{\star, \tau}}(L_{\star, \tau})$ is a 
hyperspecial maximal compact subgroup of $G_{\star, \tau}$. 
Depending on the particular application, we will be considering compact open subgroups of $\G(\Af)$ of the form $K = K_{\tau} K^{(\tau)}$ where $K_{\tau} = K_{V_\tau} \times K_{W_\tau} \subset \U(V)(F_\tau) \times \U(W)(F_\tau)$ and $K^{(\tau)} = K_{V}^{(\tau)} \times K_{W}^{(\tau)} \subset \U(V)(\mathbf{A}_{F, f}^{(\tau)}) \times \U(W)(\mathbf{A}_{F, f}^{(\tau)})$ being a product of open compact subgroups (we view $K$ as a subgroup of $\G(\Af)$). For such a $K \subset \G(\Af)$ and $\star \in \{V, W\}$, we will be considering compact open subgroups $\Ktilde_{\star} \subset \Gtilde_{\star}(\Af)$ of the form $\Ktilde_{\star} = \Ktilde_{\star, \tau} \times \Ktilde_{\star}^{(\tau)}$ where $\Ktilde_{\star, \tau}$ is the hyperspecial maximal subgroups of $\widetilde{G}_{\star, \tau}$ that is the stabilizer of the self-dual lattices $L_{\star, \tau}$ and $\Ktilde_{\star}^{(\tau)}$ is such that $K_{\star}^{(\tau)} = \Ktilde_{\star}^{(\tau)} \cap \U(\star)(\mathbf{A}_{F, f}^{(\tau)})$.


Although it is not strictly necessary, for such a $\tau$, we fix a Witt basis $\{e_{1}, e_{-1}, \dots, e_{m}, e_{-m}\}$ for $W_\tau$ for the lattice 
$L_{W, \tau}$, that is, a basis which satisfies 
$$
L_{W, \tau} = \cO_{E_\tau}e_{1} \oplus \cO_{E_\tau}e_{-1} \oplus \dots \oplus \cO_{E_\tau}e_{m} \oplus 
\cO_{E_\tau}e_{-m}, 
$$
and  
$\langle e_{i}, e_{-i}\rangle = 1$ for all $i$, and $\langle e_{i}, e_{j} \rangle = 0$ for all $i \ne -j$. This basis yields a Witt decomposition 
\begin{equation}
W_\tau = (\underbrace{E_\tau e_{1} \oplus E_{\tau} e_{-1}}_{H_1}) \perp \dots \perp (\underbrace{E_\tau e_{m} \oplus E_\tau e_{-m}}_{H_m}).  
\end{equation}
Here, $H_i = E_\tau e_{i} \oplus E_\tau e_{-i}$ is a hyperbolic plane with $E_\tau e_{i}$ and $E_\tau e_{-i}$ being isotropic lines. Using the vector $e_{D}$ (that satisfies $\langle e_D, e_D \rangle = 1$), we get a Witt basis $\{e_1, e_{-1}, \dots, e_m, e_{-m}, e_D\}$ for $V_\tau$ that is an $\cO_{E_\tau}$-basis for $L_{V, \tau}$, i.e.,  
$$
L_{V, \tau} = \cO_{E_\tau}e_{1} \oplus \cO_{E_\tau}e_{-1} \oplus \dots \oplus \cO_{E_\tau}e_{m} \oplus 
\cO_{E_\tau}e_{-m} \oplus \cO_{E_\tau} e_D,
$$
This yields a Witt decomposition for $V_\tau$: 
\begin{equation}
V_\tau = H_1 \perp H_2 \perp \dots \perp H_m \perp \cO_{E_\tau} e_D. 
\end{equation}

\subsection{Shimura Varieties}\label{subsec:shimvar}
Here, we describe the Shimura varieties associated to the unitary groups $\Hbf$ and $\G$ described in the introduction.  

\paragraph{The groups $\U(V)_{E}$ and $\GU(V)_E$.}\label{par:ungps} 
For any $E$-algebra $S$ there is an $E$-algebra isomorphism 
\begin{equation}\label{eq:complexconj}
S \otimes_F E \isom S \times S, \qquad x \otimes \alpha \mapsto ( \overline{\alpha} x,   \alpha x),  
\end{equation}
where the Galois group $\Gal(E/F)$ acts on $S \otimes_F E$ via  
$x \otimes \alpha \mapsto x \otimes \overline{\alpha}$ for $\alpha \in E$ and $x \in S$ and it acts via  
$(x, y) \mapsto (y, x)$ on $S \times S$ for any $x, y \in S$. 
Similarly, there is an isomorphism of $E$-vector spaces 
\begin{equation}\label{eq:vectE}
V \otimes_F  E \isom V \oplus V, \qquad v \otimes \alpha \mapsto (\overline{\alpha} v, \alpha v), \ v \in V, \ \alpha \in E. 
\end{equation}
The latter induces an isomorphism 
\begin{equation}\label{eq:vectS}
V \otimes_F S = (V \otimes_F E) \otimes_E S \isom V_S \oplus V_S, 
\end{equation}
where $V_S = V \otimes_E S$. 
This yields a group isomorphism 
\begin{equation}\label{eq:GLV}
\GL(V \otimes_F S) \isom \GL(V_S \oplus V_S) = \GL(V_S) \times \GL(V_S).  
\end{equation}
Under this identification, $\U(V)(S)$ can be described as 
$$
\U(V)(S) = \{g \in \GL(V \otimes_F S) \colon \langle g v, g w\rangle = \langle v, w\rangle, \ \forall v, w\in V\otimes_F S\}. 
$$ 
Using \eqref{eq:GLV}, the elements of $\U(V)(S)$ are the pairs $(g_1, g_2)$ where 
$g_1, g_2 \in \GL(V_S)$ such that for all $(v_1, v_2) \in V_S \oplus V_S$ and 
$(w_1, w_2) \in V_S \oplus V_S$, we have 
$$
\langle (g_1, g_2)(v_1, v_2), (g_1, g_2) (w_1, w_2) \rangle_{V \otimes_F S}  = \langle (v_1, v_2), (w_1, w_2)\rangle_{V \otimes_F S} \in E \otimes_F S,   
$$
where $\langle \,,\rangle_{V \otimes_F S}$ is the Hermitian pairing on $V_S \oplus V_S$ induced from the natural one on $V \otimes_F S$ via \eqref{eq:vectS}. If we use the identification \eqref{eq:complexconj}, we see that  
$$
\langle (v_1, v_2), (w_1, w_2) \rangle_{V \otimes_F S} = (\langle v_2, w_1 \rangle_S, \langle v_1, w_2 \rangle_S) \in S \times S,  
$$
and for any $(g_1, g_2) \in \U(V)(S)$, 
$$
\langle (g_1, g_2)(v_1, v_2), (g_1, g_2) (w_1, w_2) \rangle_{V \otimes_F S}  = 
(\langle g_2 v_2, g_1 w_1 \rangle_S, \langle g_1 v_1, g_2 w_2 \rangle_S) \in S \times S.
$$
This means that $\langle g_2 v_2, g_1 w_1 \rangle_S = \langle v_2, w_1 \rangle_S$ for all $v_2, w_1 \in V_S$, i.e., $g_2$ is completely determined from $g_1$ and hence, the map $(g_1, g_2) \mapsto g_1$ identifies $\U(V)(S) \isom \GL(V_S) = \GL(V)(S)$ for any $E$-algebra $S$, i.e., 
gives us an isomorphism of algebraic groups $\U(V)_E \isom \GL(V)_E$. Similarly, we get an isomorphism $\GU(V)_E \isom \GL(V)_E \times \mathbf{G}_{m, E}$ (the determinant map accounting for the second factor).

\paragraph{Shimura datum $(\G, X)$.}\label{par:shimdatum}
Let $\Sbf = \Res_{\C/\R} \G_{m, \C}$ be the Deligne circle group. Choose a basis for the $\Qbar$-hermitian space $(\star \otimes_E \Qbar, \langle\,, \rangle_{\Qbar})$ such that the Hermitian form with respect to that basis is $J = \diag(1, \dots, 1, -1)$. Let $X_{\star}$ be the $\G_{\star}(\R)$-conjugacy class of homomorphisms of $\R$-algebraic groups 
\begin{equation}\label{eq:datum}
h_{\star} \colon \Sbf \ra \G_{{\star}, \R}, \qquad z \mapsto \left ( \diag(1, \dots, 1, \overline{z} / z), \mathbf{1}_2, \dots \mathbf{1}_d \right), 
\end{equation}
where we use the identification $\G_{\star, \R} \isom \U(\star)_{F_{\rho_1}} \times \U(\star)_{F_{\rho_2}} \times \dots \times \U(\star)_{F_{\rho_d}}$. 
Alternatively \cite{gross:letter}, Witt's theorem implies that $X_{\star}$ is the space of negative lines in ${\star} \otimes_{\rho_1} \R$ (i.e., $X_\star$ is a complex ball of dimension $\dim \star -1$) 
Consider the identification $\mathbf{S}_{\C} \isom \mathbf{G}_{m, \C} \times \mathbf{G}_{m, \C}$ and the corresponding embedding of $\mathbf{S}(\R)$ into $\mathbf{S}(\C)$ given by $z \mapsto (z, \overline{z})$ where $\Sbf(\R) \isom \C^\times \hra \C^\times \times \C^\times \isom \Sbf(\C)$. 
The homomorphism $h_{\C}$ is then given by 
$$
h_{\C} \colon \Sbf_{\C} \ra \G_{V, \C}, \qquad (z_1, z_2) \mapsto \diag(1, \dots, 1, z_2 / z_1) \times \mathbf{1}_2 \times \dots \times \mathbf{1}_d. 
$$ 

\paragraph{Reflex fields.} 
Following \cite[p.101]{milne:shimura}, for a subfield $k \subset \C$, let $\cC_{\star}(k)$ 
be the $\G_{\star}(k)$-conjugacy class of co-characters of $\G_{{\star}, k}$ defined over $k$. 
Consider the homomorphism $\mathbf{G}_{m, \C} \hra \mathbf{S}_{\C} \isom \mathbf{G}_{m, \C} \times \mathbf{G}_{m,\C}$ given by $z \mapsto (z, 1)$. Given $x \in X_{\star}$, let $\mu_{\star, x}$ be the co-character 
of $\G_{\star}$ obtained by precomposing $h_{x, \C} \colon \mathbf{S}_{\C} \ra \G_{{\star}, \C}$ with that homomorphism. As 
$x$ varies over $X_{\star}$, we obtain a $\G_{\star}(\R)$-conjugacy class of co-characters and hence, an element $\mu_{X_{\star}} \in \cC_{\star}(\R)$. 

One can view $\mu_{X_{\star}}$ as an element of $\cC_{\star}(\Qbar)$. Indeed, if $\Tbf_{\star, \Qbar} \subset \G_{\star, \Qbar}$ is a maximal split torus then by \cite[Lem.12.1]{milne:shimura}, 
$\G(\C) \backslash \Hom(\G_{m, \C}, \G_{\C})$ is in bijection with 
$\mathscr W_{\star} \backslash \Hom(\G_{m, \C}, \Tbf_{\star, \C})$ where $\mathscr W_{\star} = N_{\G_{\star}(\C)}(\Tbf_{\star, \C}) / C_{\G_{\star}(\C)}(\Tbf_{{\star}, \C})$. Since neither $\mathscr W_{\star}$, nor $\Hom(\G_{m, \C}, \Tbf_{\star, \C})$ changes when we replace $\C$ by $\Qbar$, one view $\mu_{X_{\star}}$ as an element of $\cC_{\star}(\Qbar)$. 
The reflex field $E(\G_{\star}, X_{\star})$ is then the fixed field of the subgroup of $\Gal(\Qbar / \Q)$ fixing 
$\mu_{X_{\star}}$ as an element of $\cC_{\star}(\Qbar)$ (i.e., stabilizing $\mu_{X_{\star}}$ as a subset of 
$\Hom(\G_{m, \Qbar}, \G_{\star, \Qbar})$). We can now take $\Tbf_{{\star}, \Qbar}$ to be a diagonal torus with respect to the basis chosen in 
Section~\ref{par:shimdatum}. Then $\mu_{X, {\star}}$ is the $\G_{\star}(\Qbar)$-conjugacy class of the co-character  
$$
\mu_{\star} \colon \G_{m, \Qbar} \ra \G_{\star, \Qbar}, \qquad \lambda \mapsto \left ( \diag(1, \dots, 1, \lambda^{-1}), \mathbf{1}_2,  \dots,  \mathbf{1}_d \right ). 
$$
The action of $\sigma \in \Gal(\Qbar / \Q)$ on $\mu_{\star}$ is then given by 
\begin{equation}\label{eq:musigma}
\leftexp{\sigma}\mu_{\star}(\lambda) := \sigma \left ( \mu_{\star}(\sigma^{-1}(\lambda))\right ) = 
\left ( \sigma (\diag (1, \dots, 1, \sigma^{-1}(\lambda^{-1}))), \mathbf{1}_2, \dots, \mathbf{1}_d\right ). 
\end{equation}
Recall that under the isomorphism $\GL(\star)(\Qbar \otimes_F E) \isom \GL(\star )(\Qbar) \times \GL(\star)(\Qbar)$ the natural action of $\sigma \in \Gal(\Qbar / \Q)$ on $\GL(\star)(\Qbar \otimes_F E)$ corresponds to the action  
$$
\leftexp{\sigma}( g_1, g_2) = 
\begin{cases}
(\leftexp{\sigma}g_1, \leftexp{\sigma}g_2) & \text{if } \sigma \in \Gal(\Qbar / E), \\
(\leftexp{\sigma}g_2, \leftexp{\sigma}g_1) & \text{otherwise,}
\end{cases}
$$
where $\leftexp{\sigma}{g}$ denotes the usual action of $\sigma$ on $g \in \GL(\star)(\Qbar)$. As $\U(\star)(\Qbar)$ is a subgroup of $\GL(\star)(\Qbar \otimes_F E)$, this also gives us the action on $\U(\star)(\Qbar)$. 
Using that and \eqref{eq:musigma}, we obtain  
$$
\leftexp{\sigma}{\mu_{\star}} = 
\begin{cases}
\mu_{\star} & \text{if } \sigma \in \Gal(\Qbar / E), \\
\mu_{\star}^{-1} & \text{otherwise}. 
\end{cases}
$$
Since $\mu_{\star}$ and $\mu_{\star}^{-1}$ are not in the same $\G_{\star}(\Qbar)$-conjugacy class, the subgroup of $\Gal(\Qbar / \Q)$ stabilizing the $\G_{\star}(\Qbar)$-conjugacy class of $\mu_{\star}$ is exactly $\Gal(\Qbar/E)$, i.e., $E(\G, X) = E$. 


Finally, let $\Delta \colon W \hra V \times W$ be the diagonal embedding (that is, the natural inclusion on the first factor and the identity on the second factor). There is an induced diagonal embedding $\Delta \colon X_W \hra X_V \times X_W = X$. 
Consider the symmetric space $X = X_V \times X_W$ for the product group $\G = \G_V \times \G_W$ and let $Y = \Delta(X_W) \subset X$. Finally, let $K_{\Hbf} = K \cap \Hbf(\Af)$. 

\paragraph{Shimura data $(\Gtilde, X')$ and $(\Gtilde, \Xtilde)$.}
Besides the Shimura datum $(\G, X)$, we consider two other data for the group of unitary similitudes $\Gtilde_{\star}$. Let $X'_{\star}$ be the $\Gtilde_{\star}(\R)$-conjugacy class of 
$$
h \colon \mathbf{S}_{\R} \ra \Gtilde_{\star, \R}, \qquad z \mapsto \diag(1, \dots, 1, \overline{z} / z)
\times \mathbf{1}_2 \times \dots \times \mathbf{1}_d, 
$$
where we have identified $\Gtilde_{\star}(\R)$ with a subgroup of $\GU(\star)(F_{\rho_1}) \times \dots \times \GU(\star)(F_{\rho_d})$. 

\begin{rem}
The domains $X'_{\star}$ and $X_{\star}$ for the groups $\Gtilde_{\star, \R}$ and $\G_{\star, \R}$ are closely related. Yet, as we will see below, $X'_{\star}$ might be a disjoint union of two conjugates of $X_{\star}$ (e.g., in the case when the dimension of the space ${\star}$ is even). 
\end{rem}

\noindent The other hermitian symmetric domain $\Xtilde_{\star}$ is defined as the $\Gtilde_{\star}(\R)$-conjugacy class of 
$$
\widetilde{h} \colon \Sbf_{\R} \ra \Gtilde_{\star, \R}, \qquad z \mapsto \diag(z, \dots, z, \overline{z})
\times z \cdot \mathbf{1}_2 \times \dots \times z \cdot \mathbf{1}_d. 
$$

\paragraph{Shimura varieties.} 
Consider Shimura varieties $\Sh_{K_{\star}}(\G_{\star}, X_{\star})$  
whose complex points are given by  
$$
\Sh_{K_\star}(\G_{\star}, X_{\star})(\C) := \G_{\star}(\Q) \backslash (\G_{\star}(\Af) \times X_{\star}) / K_{\star}, 
$$
and let $\Sh_K(\G, X) = \Sh_{K_V}(\G_V, X_V) \times \Sh_{K_W}(\G_W, X_W)$. The complex points of $\Sh_K(\G, X)$ are given by 
$$
\Sh_K(\G, X)(\C) := \G(\Q) \backslash (\G(\Af) \times X) / K.  
$$

\noindent We also consider the Shimura varieties $\Sh_{\Ktilde_{\star}}(\Gtilde_{\star}, X_{\star}')$, 
$\Sh_{\Ktilde_{\star}}(\Gtilde_{\star}, \Xtilde_{\star})$ 
as well as $\Sh_{\Ktilde}(\Gtilde, X')$ and $\Sh_{\Ktilde}(\Gtilde, \Xtilde)$. 
Here, $\Ktilde_{\star}$ is as in \ref{par:intstruct}. 

\paragraph{Connected components.} By \cite[Lem.5.13]{milne:shimura}, the connected 
components of $\Sh_K(\G, X)$ are indexed by the double cosets $\G(\Q) \backslash \G(\Af) / K$. 
Similarly, the connected components of $\Sh_{\Ktilde}(\Gtilde, X')$ are indexed by $\Gtilde(\Q)^{\dagger} \backslash \Gtilde(\Af) / \Ktilde$, where 
$
\Gtilde(\Q)^{\dagger} = \Gtilde_V(\Q)^{\dagger} \times \Gtilde_W(\Q)^{\dagger}
$
for 
$$
 \Gtilde_{\star}(\Q)^{\dagger} = \{g_{\star} \in \Gtilde_{\star}(\Q) \colon \nu(g_{\star}) > 0\}, \qquad \star \in \{V, W\}. 
$$
The latter definition makes sense since we know that for $\widetilde{G}_{\star}$, the similitude factor $\nu$ takes values in $\Q^\times$. 
More precisely, if $g_1, \dots, g_r$ (resp., $\widetilde{g}_1, \dots, \widetilde{g}_s$) are double coset representatives for $$\G(\Q) \backslash \G(\Af) / K$$ (resp.,  $\Gtilde(\Q)^{\dagger} \backslash \Gtilde(\Af) / \Ktilde$) then 
$$
\Sh_K(\G, X) = \bigsqcup_{i = 1}^r \Gamma_{g_i} \backslash X, \qquad 
\Gamma_{g_i} = \G(\Q) \cap g_i K g_i^{-1}. 
$$
and 
$$
\Sh_{\Ktilde}(\Gtilde, X') = \bigsqcup_{i = 1}^s \widetilde{\Gamma}_{\widetilde{g}_i} \backslash X', \qquad 
\widetilde{\Gamma}_{\widetilde{g}_i} = \Gtilde(\Q)^{\dagger} \cap \widetilde{g}_i \Ktilde \widetilde{g}_i^{-1}. 
$$

\paragraph{Relation between $(\G, X)$ and $(\Gtilde, X')$.} Note that the inclusion $\G_V(\Af) \hra \Gtilde_V(\Af)$ induces a natural map 
$e_V \colon \Sh_{K_V}(\G_V, X_V) \ra \Sh_{\Ktilde_V}(\Gtilde_V, X_V')$. Similarly, 
we have a natural map $e_W \colon \Sh_{K_W}(\G_W, X_W) \ra \Sh_{\Ktilde_W}(\Gtilde_W, X_W')$. We now show that the maps $e_V$ and $e_W$ are closed embeddings. 

\begin{lem}\label{lem:injodd}
For $\star \in \{V, W\}$ the map $e_{\star}$ is injective and identifies $\Sh_{K_{\star}}(\G_{\star}, X_{\star})$ with an open and closed subset of $\Sh_{\Ktilde_{\star}}(\Gtilde_{\star}, X_{\star}')$. 
\end{lem}

\begin{proof}
The inclusion $\G_V(\Af) \hra \Gtilde_V(\Af)$ induces a map 
$$
\G_V(\Q) \backslash \G_V(\Af) / K_V \hra \Gtilde_V(\Q)^{\dagger} \backslash \Gtilde_V(\Af) / \Ktilde_V. 
$$
To check that this map is injective, it suffices to prove that for any 
$g', g'' \in \G_V(\Af)$ for which $\Gtilde_V(\Q)^{\dagger} g' \Ktilde_V = \Gtilde_V(\Q)^{\dagger} g'' \Ktilde_V$, we have $\G_V(\Q) g' K_V = \G_V(\Q) g'' K_V$. Suppose that $g'' = g_{\Q} g' k$ where $g_{\Q} \in \Gtilde_V(\Q)^{\dagger}$ and $k \in \Ktilde_V$. It follows that $\nu(g_{\Q}) \nu(k) = 1$. Since $\nu(g_{\Q}) \in \Q^\times$  and $\nu(k) \in \Zhat^\times$ and since $\Q^\times \cap \Zhat^\times = \{\pm 1\}$, we get that $\nu(g_{\Q}) = 1$ (this uses that $\nu(g_{\Q}) \in \Q_{> 0}$). 

We have thus checked injectivity on the set of connected components. We still need to check injectivity on each connected component. For $\star \in \{V, W\}$ take any connected component $\Gamma_g \backslash X_{\star}$ for $g \in \G_{\star}(\Q) \backslash \G_{\star}(\Af) / K_{\star}$. The corresponding component of $\Sh_{\Ktilde_\star}(\G_\star, X_\star)$ is 
$\widetilde{\Gamma}_g \backslash X'_{\star}$ for 
$\widetilde{\Gamma}_g = \Gtilde^{\dagger}_{\star}(\Q) \cap g \Ktilde_{\star} g^{-1}$ 
and we have to show that the map 
$\Gamma_g \backslash X_{\star} \ra \widetilde{\Gamma}_g \backslash X'_{\star}$
is injective. In the case when $\dim \star$ is odd, it suffices to check that $\Gamma_g = \widetilde{\Gamma}_g$. Clearly, $\Gamma_g \subseteq \widetilde{\Gamma}_g$. To show that $\Gamma_g = \widetilde{\Gamma}_g$, take any element 
$g_{\Q} \in \widetilde{\G}(\Q)^{\dagger} \cap g \widetilde{K} g^{-1}$ and note that $\nu(g_{\Q}) = \nu(k)$ for $k \in \widetilde{K}$. The same argument as above shows that $\nu(g_{\Q}) = 1$, i.e., $g_{\Q} \in \G(\Q) \cap g \widetilde{K} g^{-1} = \Gamma_g$.  
\end{proof}

\noindent Applying the above lemma to both $e_V$ and $e_W$, 
we get an embedding $e \colon \Sh_{K}(\G, X) \hra \Sh_{\Ktilde}(\Gtilde, X')$.

\paragraph{Reciprocity law on special points.}\label{par:specialpts} Let $(\Tbf, x)$ be a special pair in the sense of \cite[p.103]{milne:shimura} where $\Tbf \subset \Gtilde_V$ is a torus. Then $x$ corresponds to a homomorphism $h_x \colon \Sbf_{\R} \ra \Gtilde_{V, \R}$ that factors through $\Tbf_{\R}$. Composing $h_{x, \C}$ with the map $\mathbf{G}_{m, \C} \ra \Sbf_{\C} \isom \mathbf{G}_{m, \C}  \times \mathbf{G}_{m, \C} $ given by $z \mapsto (z, 1)$ yields a co-character 
$\mu_x \colon\mathbf{G}_{m, \C} \ra \Tbf_{\C}$. This is the co-character giving the reciprocity law for the action of $\sigma \in \Gal(E^{\ab} / E)$ on $[g, x] \in \Sh_{\Ktilde}(\Gtilde, X')$. More precisely, $\mu_x$ gives rise to a homomorphism $r(\mu_x, \Tbf) \colon \Res_{E/\Q} \mathbf{G}_{m, E} \ra \Tbf$ of algebraic groups over $\Q$ defined by 
\begin{equation}\label{eq:rmux}
r(\mu_x, T)(t) := \prod_{\rho \colon E \hra \Qbar} \rho(\mu_x(t)), \qquad t \in \mathbf{A}^\times_E.  
\end{equation}
Note that the sum on the right-hand side is defined over $\Q$. This gives us a homomorphism $r_x \colon \mathbf{A}_{E}^\times \ra \Tbf(\Af)$. Now, if $s \in \mathbf{A}_{E}^\times$ is an id\`ele whose image in $\Gal(E^{\ab} / E)$ under the Artin map is exactly $\sigma$ then 
$\sigma [g, x] = [r_x(s) g, x]$.

\paragraph{Galois action on connected components.} The derived subgroup $\Hbf^{\der}$ of $\Hbf$ is simply connected since it is isomorphic (over $\overline{\Q}$) to $\SL_{n-1}^d$. Let $\Tbf^1 = \Hbf / \Hbf^{\der}$ (also isomorphic to $\Res_{F/\Q} \U(1)_{F}$) and let $\det \colon \Hbf \ra \Tbf^1$ be the determinant map. Let $X(\Hbf)$ be the Hermitian subdomain
$$
X(\Hbf) = \{ x \in X \colon h_x \colon \mathbf{S}_{\R} \ra \G_{\R} \text{ factors through } \Hbf_{\R}\}.
$$
Note that $X(\Hbf)$ is a connected Hermitian symmetric domain for the group $\Hbf(\R)$. 
One can apply \cite[Thm.5.17]{milne:shimura} (a simplified version of Deligne's results on the structure of the set of connected components \cite[\S 2.1.16]{deligne:shimura}) for the Shimura variety $\Sh_{K_H}(\Hbf, X(\Hbf))$ to get that 
\begin{equation}\label{eq:comp}
\pi_0 ( \Sh_{K_{\Hbf}}(\Hbf, X(\Hbf))) = \Hbf(\Q) \backslash \Hbf(\Af) / K_{\Hbf} = \Hbf(\Q) \Hbf^1(\Af)\backslash \Hbf(\Af) / K_{\Hbf}. 
\end{equation}
The latter is isomorphic to $\Tbf^1(\Af) / \Tbf^1(\Q) \det(K_{\Hbf})$ via the determinant map 
$\det \colon \Hbf(\Af) \ra \Tbf^1(\Af)$. 
Using the canonical model of $\Sh_{K_{\Hbf}}(\Hbf, X(\Hbf))$, we get an action of $\Aut(\C/\iota(E))$ on 
$\pi_0(\Sh_{K_{\Hbf}}(\Hbf, X(\Hbf)))$. Following \cite[p.109]{milne:shimura}, for any $\sigma \in \Aut(\C / \iota(E))$, let 
$s \in \mathbf{A}_E^\times$ be such that $\Art_E(s) = \sigma |_{E^{\ab}}$. Consider the 
homomorphism 
\begin{equation}\label{eq:rf}
r_f \colon \mathbf{A}_E^\times \ra \U(F)(\mathbf{A}_F) \isom \Tbf^1(\mathbf{A}_\Q) \xra{\proj} \Tbf^1(\Af), 
\end{equation}
defined by 
$r_f(s) = \overline{s}_f / s_f$ where $s_f$ denotes the finite part of $s \in \mathbf{A}_E^\times$. 
Note that this is precisely the homomorphism $r$ defined on \cite[p.109]{milne:shimura} for the Shimura datum  
$(\Hbf, X(\Hbf))$ followed by the projection map $\proj$. Indeed, if $h := \det \circ h_W \colon \mathbf{S}_{\R} \ra \Tbf^1_{\R}$ then the co-character $\mu_{h} \colon \mathbf{G}_{m, \C} \ra \Tbf^1_{\C}$ associated to $h_{\C} \colon  \Sbf_{\C} \ra \Tbf^1_{\C}$ by precomposing with $\mathbf{G}_{m, \C} \hra \mathbf{G}_{m, \C} \times \mathbf{G}_{m, \C} \isom \Sbf_{\C}$, $z \mapsto (z, 1)$ is defined over the reflex field $E$ and hence, can be evaluated on $\mathbf{A}_E^\times$. This gives rise to the homomorphism $r(\Tbf^1, \mu_h) \colon \mathbf{A}_E^\times \ra \Tbf^1(\mathbf{A}_{\Q}) \isom \U(1)(\mathbf{A}_F)$ from \cite[p.109]{milne:shimura} (the analogue of \eqref{eq:rmux} for $\mu_h$) defined by 
$$
r(\Tbf^1, \mu_h)(s) := \mu_h(s) \overline{\mu_h(s)},  \qquad s \in \mathbf{A}_E^\times.  
$$
Since $\mu_h(s) = s^{-1} \in \U(1)(\mathbf{A}_E)$ where we have used the identification $\U(1)(\mathbf{A}_E) \isom \mathbf{A}_E^\times$. Under this identification, the conjugation action on the left-hand side corresponds to the action $s \mapsto \overline{s}^{-1}$ on the right-hand side and hence, $r(\Tbf^1, \mu_h)(s) = \overline{s} / s \in \U(1)(\mathbf{A}_F) \isom \Tbf^1(\mathbf{A}_\Q)$ and hence, the projection of $r(\Tbf^1, \mu_h)$ to $\Tbf^1(\Af)$ coincides with $r_f$ defined in \eqref{eq:rf}. It now follows from \cite[p.109]{milne:shimura} and \eqref{eq:comp} that if $\cC \in \pi_0(\Sh_{K_{\Hbf}}(\Hbf, X(\Hbf)))$ is represented by 
$t \in \Tbf^1(\mathbf{A}_{f})$ then 
\begin{equation}\label{eq:galoisconn}
\cC^\sigma  = [r_f(s) t] \in \Tbf^1(\mathbf{A}_{f}) / \Tbf^1(\Q) \det(K_{\Hbf}). 
\end{equation}

\subsection{Special cycles on $\Sh_K(\G, X)$}\label{subsec:speccyc}
\paragraph{The cycles $\cZ_K(g)$.} Given $g \in \G(\Af)$, consider the cycle $\cZ_K(g)$ that is the image of $gK \times Y$ in $\Sh_K(\G, X)(\C) = \G(\Q) \backslash (\G(\Af)/K \times X)$. 
Let $\cZ_K(\G, \Hbf) = \{\cZ_K(g) \colon g \in \G(\Af)\}$ be the space of all such cycles. We have a map 
$$
\cZ_K(\bullet) \colon \Hbf(\Q) \backslash \G(\Af) / K \ra \cZ_K(\G, \Hbf). 
$$ 
The map is certainly surjective by definition of $\cZ_K(\G, \Hbf)$.  

\begin{lem}\label{lem:ident}
(i) We have 
$\Nbf_{\G}(\Hbf) = \Hbf \cdot \Zbf_{\G} \subset \G$. 

\noindent (ii) The map $\cZ_K(\bullet) \colon \Hbf(\Q) \backslash \G(\Af) / K \ra \cZ_K(\G, \Hbf)$ is surjective and induces a bijection 
$$
\cZ_K(\bullet) \colon \Nbf_{\G(\Q)}(\Hbf(\Q)) \backslash \G(\Af) / K \ra \cZ_K(\G, \Hbf).
$$ 
\end{lem}

\begin{proof}
For (i), let $R$ be any $\Q$-algebra. Clearly, $\Hbf(R) \Zbf_{\G}(R) \subset \Nbf_{\G}(\Hbf)(R)$. Let $g \in \Nbf_{\G}(\Hbf)(R)$. For any 
$h \in \Hbf(R)$, $g h g^{-1}$ fixes $D_R = D \otimes_{\Q} R$ pointwise, i.e., $\Hbf(R)$ fixes $g^{-1}D_R$ 
pointwise. But the only line in $V_R = V \otimes_{\Q} R$ fixed pointwise by $\Hbf(R)$ is $D_R$ itself. Hence, $g^{-1}D_R = D_R$. But it is not hard to check that the subgroup of $\G(R)$ that fixes $D_R$ (not pointwise, but as a set) is precisely $\Hbf(R) \Zbf_{\G}(R)$, so $g \in \Hbf(R) \Zbf_{\G}(R)$.  

For (ii), the condition $\cZ_{K}(g') = \cZ_{K}(g'')$ is equivalent to $\G(\Q) (g'K, Y) = \G(\Q) (g'' K, Y)$. This is equivalent to the following statement: 
\begin{equation}\label{eq:cond}
\forall y \in Y, \qquad \exists g_{\Q} \in g' K {g''}^{-1} \cap \G(\Q) \text{ such that } y \in g_{\Q} Y. 
\end{equation}
The latter means that $\ds Y = \bigcup_{g_\Q \in \G(\Q) \cap g' K {g''}^{-1}} (Y \cap g_\Q Y)$, i.e., $Y$ is a countable union of sets $Y \cap g_\Q Y$.

\vspace{0.1in}

\noindent {\bf Claim:} There exists $g_{\Q} \in g' K {g''}^{-1}$ such that $g_{\Q} Y = Y$. 

\vspace{0.1in}



We prove the claim using the fact that the Riemann manifolds $Y$ and $g_{\Q} Y$ are totally geodesic. 
Indeed, Baire's category theorem implies that there exists $g_\Q \in \G(\Q) \cap g' K {g''}^{-1}$ such that 
$Y \cap g_{\Q} Y$ contains an open set $U$ of $Y$. We claim that $Y = g_{\Q} Y$. 
Observe that $Y \cap g_{\Q} Y \subset Y$ is a totally geodesic submanifold. 
Take any point $y \in U$
and consider any geodesics $\gamma$ through $y$ in $Y$. Since the germ $[\gamma]$ of that geodesics is contained in $U$, by the extension property of geodesics, the entire geodesics is contained in $Y \cap g_{\Q} Y$. Now, using that $Y$ is connected, it follows that $Y$ is contained in 
$Y \cap g_{\Q} Y$ (indeed, any point $x \in Y$ can be connected by a geodesic to $y$ which, by the above argument, is necessarily in $Y \cap g_{\Q} Y$, i.e., $x \in Y \cap g_{\Q} Y$). 
Thus, there exists $g_\Q \in \G(\Q) \cap g' K {g''}^{-1} $ such that $g_\Q Y = Y$. 
\vspace{0.1in}


\vspace{0.1in}

\noindent We next compute the stabilizer $S(Y) = \Stab_{\G(\Q)}(Y)$ by first computing  $\Stab_{\G(\Q)}(\Delta(W))$. 

\vspace{0.1in}


\noindent {\bf Computing $\Stab_{\G(\Q)}(\Delta(W))$:} The condition 
$(g_{V, \Q}, g_{W, \Q}) \in \Stab_{\G(\Q)}(\Delta(W))$ means that for any $w \in W$, $(g_{V, \Q}, g_{W, \Q}) (\Delta_V(w), w) = (\Delta_V(w'), w')$ for some $w' \in W$, i.e., $g_{V, \Q} \Delta_V(g_{W, \Q})^{-1}$ fixes 
$\Delta_V(w)$ for every $w \in W$. This is equivalent to $g_{V, \Q} = (g_{W, \Q}, u) \in \U(W) \times \U(D) \subset \G(V)$ for some $u \in \U(D)$, i.e., 
$$
\Stab_{\G(\Q)}(\iota(W)) = \Hbf(\Q) (\underbrace{1 \times \Res_{F/\Q} \U(D)(\Q)}_{\subset \G_V(\Q)} \times \underbrace{1}_{\subset \G_W(\Q)}) \subset \G(\Q).
$$  

\vspace{0.1in}

\noindent {\bf Computing $\Stab_{\G(\Q)}(Y)$:} Let $(g_{V, \Q}, g_{W, \Q}) \in \Stab_{\G(\Q)}(\Delta(X_W))$. Equivalently, for any negative-definite line $\ell \in X_W$, there exists a negative definite line $\ell' \in X_W$ such that 
$$
(g_{V, \Q}, g_{W, \Q}) (\Delta_V(\ell), \ell) = (\Delta_V(\ell'), \ell') \Longleftrightarrow \ell' = g_{W, \Q} \ell \text{ and }  
\Delta_V(\ell') = g_{V, \Q} \Delta_V(\ell). 
$$
The latter is equivalent to $g_{V, \Q} \Delta_V(g_{W, \Q})^{-1}$ fixing the negative-definite line $\Delta_V(\ell) \in X_V$. This means that 
$$
\Stab_{\G(\Q)}(\Delta(X_W)) = \Hbf(\Q)(\underbrace{1 \times \Res_{F/\Q} \U(D)(\Q)}_{\subset \G_V(\Q)} \times \underbrace{\Zbf_{\Hbf}(\Q)}_{\subset \G_W(\Q)}) = \Hbf(\Q) \Zbf_{\G}(\Q) \subset \G(\Q). 
$$

To conclude (ii), observe that $\cZ_K(g') = \cZ_K(g'')$ if and only if $g'' \in \Stab_{\G(\Q)}(Y)g'K$, i.e., 
the map 
$$
\cZ_K(\bullet) \colon \Stab_{\G(\Q)}(Y) \backslash G(\Af) /  K \ra \cZ_K(\G, \Hbf) 
$$
is a bijection. 
\end{proof}

\paragraph{Description of $\cZ_K(g)$ in terms of connected components and Galois action.}\label{par:concomp} Alternatively, the cycle $\cZ_K(g)$ can be described as follows: 
\begin{lem}\label{lem:altern}
If $K_{g, \Hbf} = gKg^{-1} \cap \Hbf(\Af)$ for $g \in \G(\Af)$ then $\cZ_K(g)$ is the image of 
the connected component $\Hbf(\Q) (K_{g, \Hbf} \times Y)$ of $\Sh_{K_{g, \Hbf}}(\Hbf, Y)$ under the maps\footnote{Recall that all of these maps are defined over $E$.} 
\begin{equation}\label{eq:embed}
\Sh_{K_{g, \Hbf}}(\Hbf, Y) \hra \Sh_{gKg^{-1}}(\G, X) \xra{[\cdot g]} \Sh_K(\G, X). 
\end{equation}
\end{lem}

\begin{proof}
This follows immediately by chasing through the definitions of the maps: indeed, for any element $(h_{\Q} gkg^{-1}, h_{\Q} y) \in \Hbf(\Q)(K_{g, \Hbf} \times Y)$ (here, $h_{\Q} \in \Hbf(\Q), g \in \G(\Af), k \in K$ and $y \in Y$) maps to the element $[gk, y]_K \in \Sh_K(\G, X)$. Conversely, any $[gk, y]_K \in [gK \times Y]$ is the image of $[gk' g^{-1}, y]$ for any $k' \in K_{\Hbf}$. 
\end{proof}

We use that description to provide the Galois action on $\cZ_K(\G, \Hbf)$. Let $\Ver_{E / F} \colon \Gal(F^{\ab} / F) \ra \Gal(E^{\ab} / E)$ be the transfer map (the \emph{Verlagerung} map) and let $E[\infty]$ be the abelian extension of $E$ determined by the image of $\Ver_{E /F}$. The Artin map $\Art_E \colon \Tbf(\Af) \ra \Gal(E^{\ab} / E)$ induces an isomorphism 
$\Art^1_E \colon \Tbf^1(\Af) / \Tbf^1(\Q) \xra{\sim} \Gal(E[\infty] / E)$. 
Consider the group $\Nbf_{\G}(\Hbf)(\Q) \Hbf(\Af)$ and the (normal) subgroup $\Nbf_{\G}(\Hbf)(\Q)\Hbf^1(\Af)$. The quotient is isomorphic to 
$$
\frac{\Nbf_{\G}(\Hbf)(\Q) \Hbf(\Af)}{\Nbf_{\G}(\Hbf)(\Q) \Hbf^1(\Af)} \isom \frac{\Tbf^1(\Af)}{\Tbf^1(\Q)} \xra{\sim} \Gal(E[\infty] / E), 
$$
where the last map is $\ds \Art_E^1 \colon \frac{\Tbf^1(\Af)}{\Tbf^1(\Q)} \xra{\sim} \Gal(E[\infty] / E)$. 
We have a map 
$$
{\det}^* \colon \Nbf_{\G}(\Hbf)(\Q) \Hbf(\Af) \twoheadrightarrow \Nbf_{\G}(\Hbf)(\Q) \Hbf(\Af)  /\Nbf_{\G}(\Hbf)(\Q) \Hbf^1(\Af)  \ra \Tbf^1(\Af) / \Tbf^1(\Q)
$$ 
induced by the determinant map. 

\begin{lem}\label{lem:galdesc}
For any $\sigma \in \Gal(E^{\ab} / E)$ and any element $h_\sigma \in \Nbf_{\G}(\Hbf)(\Q) \Hbf(\Af)$ that satisfies   
$\Art^1_E(\det^*(h_\sigma)) = \sigma |_{E[\infty]}$, we have
$$
\cZ_K(g)^{\sigma} = \cZ_K(h_\sigma g). 
$$
\end{lem}

\begin{proof}
We use Lemma~\ref{lem:altern} and equations \eqref{eq:comp} and \eqref{eq:galoisconn}. 
For $\sigma \in \Gal(E^{\ab} / E)$ and $s \in \mathbf{A}_{E}^\times$ such that $\Art_E(s) = \sigma$, consider $r_f(s) \in \Tbf^1(\Af)$. It follows from \cite[Lem.5.21]{milne:shimura} that $\det \colon \Hbf(\Af) \ra \Tbf^1(\Af)$ is surjective, so there exists $h_s$ 
such that $\det(h_s) = r_f(s)$ and hence, 
\begin{equation}\label{eq:shimrec}
\cZ_K(g)^\sigma = \cZ_K(h_s g).
\end{equation}
We thus have to check that $\cZ_K(h_\sigma g) = \cZ_K(h_s g)$. The latter is a consequence of 
the fact that $\det^* (h_s)  = \det^* (h_\sigma)$ and hence, $h_s^{-1} h_\sigma \in \Nbf_{\G}(\Hbf)(\Q) \Hbf^1(\Af)$. To complete the proof, we use 
$$
\cZ_K(\G, \Hbf) = \Nbf_{\G}(\Hbf)(\Q) \backslash \G(\Af) / K = \Nbf_{\G}(\Hbf)(\Q) \Hbf^1(\Af) \backslash \G(\Af) / K, 
$$
the latter being a consequence of strong approximation. 
\end{proof}

\paragraph{Galois orbits and local conductors.}\label{par:galorbits-localcond}
Lemma~\ref{lem:galdesc} shows that $\Gal(E^{\ab} / E)$ acts on $\cZ_K(\G, \Hbf)$ via the left action of $\Nbf_{\G}(\Hbf)(\Q) \Hbf(\Af)$ on $\Nbf_{\G}(\Hbf)(\Q) \backslash \G(\Af) / K$,  
so Lemma~\ref{lem:ident} yields
\begin{equation}\label{eq:orbits}
\Gal(E^{\ab} / E) \backslash \cZ_K(\G, \Hbf) \isom \Nbf_{\G}(\Hbf)(\Q) \Hbf(\Af) \backslash \G(\Af) / K = 
\Hbf(\Af) \backslash \G(\Af) / K. 
\end{equation}
For the last equality, we have used that $\Nbf_{\G}(\Hbf)(\Q) \Hbf(\Af) K = \Hbf(\Af) K$ which can be seen as follows: pick any inert place  
$\tau$ of $F$ for which $K_p \subset \G(\Q_p)$ is hyperspecial where $p$ is the prime below $\tau$ and write
$$
\Nbf_{\G}(\Hbf)(\Q) \Hbf(\Af) K \subseteq \Nbf_{\G}(\Hbf)(\Q_p) \Hbf(\Af) K = \Zbf_{\G}(\Q_p) \Hbf(\Q_p) \Hbf(\Af) K = \Hbf(\Af) K, 
$$ 
where the second equality is a consequence of Lemma~\ref{lem:ident}(i) and the latter uses $\Zbf_{\G}(\Q_p) \subset K_p$.  


For a special cycle $\xi \in \cZ_K(\G, \Hbf)$, the field of definition $E(\xi)$ can be calculated using 
Lemma~\ref{lem:galdesc}: $E(\xi)$ is the subfield of $E[\infty]$ that satisfies 
$$
\Gal(E[\infty] / E(\xi)) = \Art^1_E \left ( {\det}^{*} \left ( \Nbf_{\G}(\Hbf)(\Q) gK g^{-1} \cap \Nbf_{\G}(\Hbf)(\Q) \Hbf(\Af)\right )\right ). 
$$
The reciprocity law described in Section~\ref{par:concomp} allows us to compute the completion $E(\xi)_\tau$ of $E(\xi)$ at the place of $E(\xi)$ above $\tau$ (determined by the fixed embedding $\iota_\tau \colon \overline{E} \hra \overline{E}_\tau$) as follows: let $x_\tau \in G_\tau / K_\tau$ 
be the $\tau$-component of an element of the right-hand side of \eqref{eq:orbits} that maps to the orbit 
$\Gal(E^{\ab} / E) \xi$. Since both $K_{V, \tau}$ and $K_{W, \tau}$ are hyperspecial maximal compact subgroups, 
$G_\tau / K_\tau$ is in bijection with the pairs $(L_{V, \tau}, L_{W, \tau})$ of self-dual hermitian lattices in 
$V_\tau$ and $W_\tau$, respectively. If $\cL_\tau$ denotes the set of all such pairs of lattices, then $H_\tau$ acts on $\cL_\tau$ and the element $x_\tau$ corresponds to an $H_\tau$-orbit denoted by $[L_{V, \tau}, L_{W, \tau}]$. In this case, by the reciprocity law \eqref{eq:shimrec}, the local extension $E(\xi)_\tau  / E_\tau$ corresponds (by local class field theory) to the local norm subgroup that is the preimage (in $\cO_{E_\tau}^\times$) under the map $r_\tau \colon E_\tau^\times \ra \U(1)(F_\tau), \ \ s_\tau \mapsto \overline{s}_\tau / s_\tau$ of the image of $\Stab_{H_\tau}(L_{V, \tau}, L_{W, \tau})$ under the determinant map $\det \colon H_\tau \ra \U(1)(F_\tau)$. This local norm subgroup turns out to be determined precisely by a local order of $\cO_{E, \tau}$. If $\varpi \in \cO_{E_\tau}$ is a uniformizer, recall that for $c \geq 0$, the local order of conductor $\varpi^c$ is $\cO_{c, \tau} = \cO_{F_\tau} + \varpi^n \cO_{E_{\tau}}$. In Section~\ref{sec:galois-cycles}, we will determine the conductor of this local order (we define $\bc_\tau(\xi) := \varpi^c$ and refer to $\bc_\tau(\xi)$ as the local conductor of $\xi$ at $\tau$) in terms of the distance functions on the corresponding Bruhat--Tits buildings at $\tau$.

\begin{lem}\label{lem:locgal}
If $\det \colon H_\tau \ra \U(1)(F_\tau)$ is the local determinant map at $\tau$ then the completion $E(\xi)_\tau$ is the subfield of $E[\infty]_\tau$ whose local Galois group 
$\Gal(E[\infty]_\tau / E(\xi)_\tau)$ is 
$$
\Art_\tau^1 \left ( \det \left ( g_\tau K_\tau g_{\tau}^{-1} \cap H_\tau \right ) \right ) = 
\Art_\tau^1 \left ( \det \left ( \Stab_{H_\tau} (L_{V, \tau}, L_{W, \tau}) \right ) \right ). 
$$
\end{lem}

\section{Local Galois Action}\label{sec:galois-cycles}
\setcounter{paragraph}{0}

We now prove Theorem~\ref{thm:galois}. 
For readability, we adopt local notation for this and next section. Assume that $n = 3$ (we expect that a similar local conductor formula should hold for any $n$, the latter being a work in progress. Let $\tau$ be an allowable inert finite place of $F$, let $k_0 = F_{\tau}$, let $k = E_{\tau}$ and let $\varpi$ be a uniformizer of $k_0$ (since $k / k_0$ is unramified and quadratic $\varpi$ is a uniformizer of $k_0$ as well). Let $q$ be the size of the residue field of $k_0$. We simplify the notation by letting $G_V =\U(V)(k_0)$, $G_W = \U(W)(k_0)$, $G = \U(V)(k_0) \times \U(W)(k_0)$. The hyperspecial maximal compact subgroups $K_{V, \tau}$ and $K_{W, \tau}$ are denoted by $K_V \subset G_V$ and $K_W \subset G_W$; let $K = K_V \times K_W$. We will also use $V$ and $W$ for the local $k$-Hermitian spaces $V_\tau$ and $W_\tau$, respectively. Let $\delta_V = \diag(\varpi, 1, \varpi^{-1})$ and $\delta_W = \diag(\varpi, \varpi^{-1})$. 

Note that \eqref{eq:orbits} and Lemma~\ref{lem:locgal} allow us to reduce the problem of computing the local Galois action at 
$\tau$ to computing stabilizers (in $H$) of elements of $G / K$. 
The quotient 
$G / K$ is in bijection with the pairs $(L_{V}, L_{W})$ of self-dual Hermitian lattices in 
$V$ and $W$, respectively, and hence, the quotient $H \backslash G / K$ is in bijection with the set of $H$-orbits of hyperspecial points on the product of the Bruhat--Tits buildings for $G_V$ and $G_W$. 

\subsection{Bruhat--Tits buildings for unitary groups}

\paragraph{Buildings, apartments, hyperspecial and special vertices.}
We describe the Bruhat--Tits buildings for unitary groups is via the theory of $p$-adic self-dual norms, an approach initiated by Goldman and Iwahori \cite{goldman-iwahori} and rdeveloped further by Bruhat and Tits \cite{bruhat-tits:4} (see also \cite{cornut:normes1} and \cite[\S 4.1]{koskivirta:thesis}). For $\star \in \{V, W\}$, let 
$\cB(\star)$ be the set of self-dual ultrametric norms in $\star$ in the sense of \cite[p.28]{koskivirta:thesis}. Given a Witt basis $\mathscr B$ for $V$, one defines the apartment $\cA_{\mathscr B}$ corresponding to $\mathscr B$ as the set of all $\alpha \in \cB(V)$ adapted to $\mathscr B$ in the sense of \cite[Defn.47]{koskivirta:thesis}. Assuming that $\langle e_0, e_0\rangle$ is a unit, we parametrize $\cA_{\mathscr B}$ by the real line $\R$ as follows: for any $\lambda \in \R$, define a self-dual norm $\alpha_\lambda$ by
$$
\alpha_{\lambda}(v) = q^{\inf \{\theta \in \R \colon v \in \varpi^{-[\theta+\lambda]} \cO_{k} e_+ \oplus \varpi^{-[\theta]} \cO_{k}e_0 \oplus \varpi^{-[\theta - \lambda]} \cO_{k} e_-\}}, 
$$
where $[r]$ is denotes the integer part of $r$.  
Associated to a self-dual norm $\alpha \in \cB(V)$ is the chain of balls $B^*(\alpha) = \{B(\alpha, \theta) \colon \theta \in \R\}$ where $B(\alpha, \theta) = \{v \in V \colon \alpha(v) \leq q^\theta\}$. We say that two norms $\alpha'$ and $\alpha''$ are equivalent (and denote it by $\alpha' \sim \alpha''$) if $B^*(\alpha') = B^*(\alpha'')$. A self-dual norm $\alpha \in \cB(V)$ is a \emph{vertex} of the building $\cB(V)$ if it is the only self-dual norm in its equivalence class $\cl(\alpha)$, i.e., if $\cl(\alpha) = \{\alpha\}$. The other equivalence classes of self-dual norms are called \emph{facets}. Given a facet $X$ with a chain of lattices $B^*(X)$, we say that a self-dual norm $\alpha$ \emph{belongs to} $X$ if $B^*(\alpha) \subset B^*(X)$. Two vertices $\alpha'$ and $\alpha''$ are called \emph{neighbors} if they are vertices of the same facet. Moreover \cite[p.30]{koskivirta:thesis}, for every $\alpha \in \cB(V)$ the chain $B^*(\alpha)$ is a union of homothety classes. A vertex $\alpha$ is \emph{hyperspecial} if $B^*(\alpha)$ is a single homothety class. The other vertices are called \emph{special}. Hyperspecial vertices correspond to $\lambda \in \Z$ whereas special, but not hyperspecial vertices correspond to $\ds \lambda \in \frac{1}{2} + \Z$. Vertices are connected by edges (facets) $X_\lambda = \ds \left ] \lambda, \lambda + \frac{1}{2}\right [$ for $\ds \lambda \in \frac{1}{2} \Z$. 

There is a rather explicit description of the hyperspecial and special vertices of the buildings $\cB(\star)$ in terms of lattices that, in the case of $n = 3$ explains the graphs in Fig.~\ref{fig:buildings}. Recall that an $\cO_{k}$-lattice $L_{\star} \subset \star$ is \emph{self-dual} if $L_{\star}^\vee = L_{\star}$. Each self dual lattice in $\star$ yields a maximal compact subgroup of $G_{\star}$ and hence, a hyperspecial point $x_{\star}$ of $\cB(\star)$. Let $\Hyp_{\star}$ be the set of hyperspecial vertices of $\cB(\star)$. To understand the incidences, consider lattices that are not self-dual, but \emph{almost self-dual}. A lattice $L$ in $\star$ is \emph{almost self-dual} if 
$
\varpi L^\vee \subsetneq L \subsetneq L^\vee.  
$
Almost self-dual lattices correspond to special, but not hyperspecial points (the white points in Fig.~\ref{fig:buildings}). 
Let $\Sp_{\star}$ be the set of special, but not hyperspecial vertices in $\star$. To explain the incidences, recall that a choice of a Witt basis $\langle e_+, e_0, e_- \rangle$ of $V$ (resp., $\langle e_+, e_-\rangle$ of $W$) fixes apartments $\cA_{V}$ (resp., $\cA_{W}$) of $\cB(V)$ (resp., $\cB(W)$). The intersection $\Hyp_{V} \cap \cA_{V}$ consists of all lattices of the form $\langle \varpi^n e_+, e_0, \varpi^{-n}e_-\rangle$ for $n \in \Z$. Similarly, $\Hyp_{W} \cap \cA_{W}$ consists of all lattices of the form $\langle \varpi^n e_+, \varpi^{-n}e_-\rangle$ for $n \in \Z$. The intersection $\Sp_{V} \cap \cA_{V}$ consists of all lattices of the form $\langle \varpi^{n+1} e_+, e_0, \varpi^{-n}e_-\rangle$ for $n \in \Z$. Finally, $\Sp_{W} \cap \cA_{W}$ consists of all lattices of the form $\langle \varpi^{n+1} e_+, \varpi^{-n}e_-\rangle$ for $n \in \Z$. If 
$$
L = \langle \varpi^a e_+, e_0, \varpi^b e_- \rangle, L' = \langle \varpi^{a'} e_+, e_0, \varpi^{b'} e_- \rangle \in \cA_{V} \cap (\Hyp_{V} \cup \Sp_{V})
$$
are two lattices then their distance is defined as 
$$
\dist(L, L') = \frac{1}{2} \left ( |a - a'| + |b - b'| \right )
$$ 
(and similarly for $\cB(W)$). We connect two vertices in $\Hyp_{\star} \cup \Sp_{\star}$ by an edge if $\dist(L, L') = 1/2$. If we color the hyperspecial vertices with black and the special ones with white then the vertices of each edge have different colors. The resulting graph $\cB(\star)$ is a tree (the dimension is equal to the rank of the maximal split torus) and we can count the number of neighbors of each black and white vertex by counting the number of isotropic lines in Hermitian spaces over finite fields. For $\cB(V)$, each hyperspecial vertex has $q^3 +1$ white neighbors and each white neighbor has $q+1$ black neighbors.  For 
$\cB(W)$ each black vertex has $q+1$ white neighbors and each white neighbor has $q+1$ black ones. One can draw the tree $\cB(V)$ and the subtree $\cB(W)$ as  shown in Fig.~\ref{fig:buildings}.  


\begin{figure}
\begin{center}
\begin{picture}(300,200)

\put(0, 0){\line(250, 0){250}}
\put(0, 0){\line(1,1){100}}
\put(250, 0){\line(1,1){100}}
\put(100,100){\line(1,0){250}}

\put(175,50){\circle*{6}}
\put(135,50){\circle{6}}
\put(215,50){\circle{6}}
\put(195,70){\circle{6}}
\put(155,30){\circle{6}}

\put(175,50){\line(1,0){37}}
\put(175,50){\line(-1,0){37}}
\put(175,50){\line(1,1){17}}
\put(175,50){\line(-1,-1){17}}

\put(120, 50){\circle*{6}}
\put(125, 40){\circle*{6}}
\put(145, 60){\circle*{6}}
\put(120,50){\line(1,0){12}}
\put(123,38){\line(1,1){10}}
\put(147,62){\line(-1,-1){10}}

\put(205,40){\circle*{6}}
\put(225,60){\circle*{6}}
\put(230,50){\circle*{6}}
\put(230,50){\line(-1,0){12}}
\put(217,52){\line(1,1){10}}
\put(213,48){\line(-1,-1){10}}

\put(210,70){\circle*{6}}
\put(180,70){\circle*{6}}
\put(205,80){\circle*{6}}
\put(198, 70){\line(1,0){15}}
\put(192, 70){\line(-1,0){15}}
\put(197, 72){\line(1,1){10}}

\put(140,30){\circle*{6}}
\put(170,30){\circle*{6}}
\put(145,20){\circle*{6}}
\put(158, 30){\line(1,0){15}}
\put(152, 30){\line(-1,0){15}}
\put(153, 28){\line(-1,-1){10}}

\put(25, 10){$\cB(W)$}

\put(155,114){\circle{6}}
\put(165,118){\circle{6}}
\put(175,120){\circle{6}}
\put(185,118){\circle{6}}
\put(195,114){\circle{6}}
\put(175,50){\line(0,1){67}}
\put(175,49){\line(1,6){11}}
\put(175,49){\line(-1,6){11}}
\put(175,50){\line(1,3){20}}
\put(175,50){\line(-1,3){20}}

\put(210,114){\circle{6}}
\put(220,118){\circle{6}}
\put(230,120){\circle{6}}
\put(240,118){\circle{6}}
\put(250,114){\circle{6}}
\put(230,50){\line(0,1){67}}
\put(230,49){\line(1,6){11}}
\put(230,49){\line(-1,6){11}}
\put(230,50){\line(1,3){20}}
\put(230,50){\line(-1,3){20}}

\put(100,114){\circle{6}}
\put(110,118){\circle{6}}
\put(120,120){\circle{6}}
\put(130,118){\circle{6}}
\put(140,114){\circle{6}}
\put(120,50){\line(0,1){67}}
\put(120,49){\line(1,6){11}}
\put(120,49){\line(-1,6){11}}
\put(120,50){\line(1,3){20}}
\put(120,50){\line(-1,3){20}}

\put(105,188){\circle*{6}}
\put(120,190){\circle*{6}}
\put(135,188){\circle*{6}}
\put(120,123){\line(0,1){67}}
\put(121,123){\line(1,5){13}}
\put(119,123){\line(-1,5){13}}

\put(160,188){\circle*{6}}
\put(175,190){\circle*{6}}
\put(190,188){\circle*{6}}
\put(175,123){\line(0,1){67}}
\put(176,123){\line(1,5){13}}
\put(174,123){\line(-1,5){13}}

\put(215,188){\circle*{6}}
\put(230,190){\circle*{6}}
\put(245,188){\circle*{6}}
\put(230,123){\line(0,1){67}}
\put(231,123){\line(1,5){13}}
\put(229,123){\line(-1,5){13}}

\put(25, 110){$\cB(V)$}
\end{picture}
\end{center}
\caption{The Bruhat--Tits building $\cB(W)$ viewed as a sub-building of $\cB(V)$.}
\label{fig:buildings}
\end{figure}
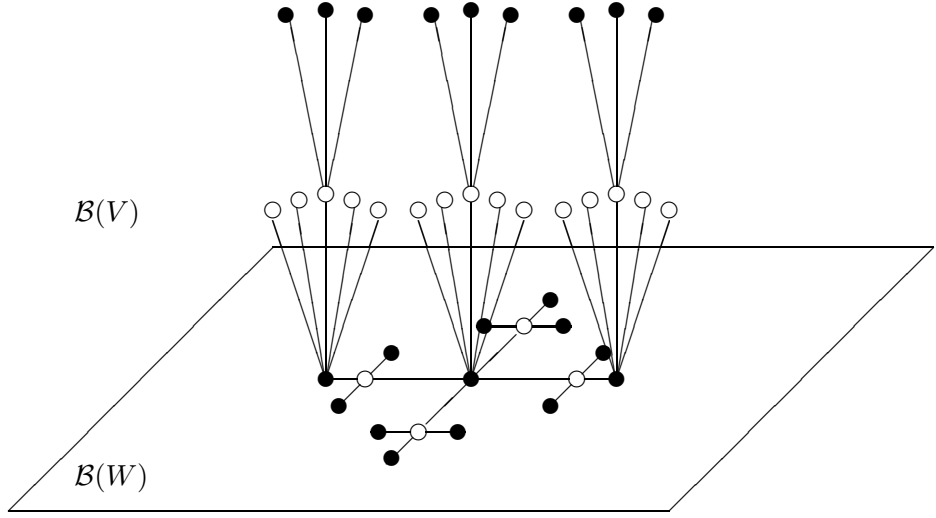

\subsection{Local conductor and the distance function - proof of Theorem~\ref{thm:galois}}\label{subsec:loccond}
To prove Theorem~\ref{thm:galois}, we compute the local conductor in terms of the distance function on $\cB(V)$. 

\paragraph{Local conductors and stabilizers.}\label{par:localcondstab}
Let $\cL$ be the set of pairs $(L_{V}, L_{W})$ of self-dual hermitian lattices in 
$V$ and $W$, respectively. Given a pair $(L_V, L_W) \in \cL$, let $[L_{V}, L_{W}] := H \cdot (L_V, L_W)$ denote its $H$-orbit. 
For $n \geq 0$, if $\varpi^n$ is $\cO_{n} = \cO_{k_0} + \varpi^n \cO_{k}$ is the local order of conductor then there is a filtration 
\begin{equation}\label{eq:filt1}
\cO_{0}^\times \supset \cO_1^\times \supset \dots \supset \cO_{n}^\times \supset \dots
\end{equation}
whose image under the map $r \colon k^\times \ra \U(1)(k_0)$ given 
by $r(s) = \overline{s} / s$ for $s \in k^\times$ yields a filtration: 
\begin{equation}\label{eq:filt2}
\cO_0^1 := r (\cO_0) \supset \cO_1^1 := r(\cO_1) \supset \dots \supset \cO^1_n := r(\cO_n) \supset \dots.
\end{equation}

\begin{lem}\label{lem:filt}
If $s \in \cO_0^1$ satisfies $s \equiv 1 \bmod \varpi^c$ for some $c > 0$ then $r^{-1}(s) \subset \varpi^{\Z} \cO_c^\times$.  
\end{lem}

\begin{proof}
Let $\eta \in \cO_{k}$ be such that $\cO_{k} = \cO_{k_0}[\eta]$ and $\eta^2 \in \cO_{k_0}^\times$ (i.e., $\overline{\eta} = - \eta$). That $\eta \in \cO_{k_0}^\times$ is a consequence of the fact that $k / k_0$ is unramified. Let $s = \overline{\lambda} / \lambda$ for $\lambda = x + \eta y$ with $x, y \in \cO_{k_0}$. We can assume that $v(\lambda) = 0$ (otherwise, we consider $\varpi^{-v(\lambda)} \lambda$ whose image under $r$ is still $s$). Since, $x - \eta y = (x + \eta y)(1 + \varpi^c z)$ where $s = 1+ \varpi^c z$, we get $2 \eta y + \varpi^c (x + \eta y) = 0$. Since $\tau$ has odd residue characteristic, it follows that $v(y) \geq c$. Hence, $\lambda = x + \eta y \in \cO_{c}^\times$.  
\end{proof}

\noindent Calculating the local conductor amounts to detecting the position of 
$$
\ds \det(\Stab_{H}(L_{V}, L_{W})) \subset \U(1)(k_0)
$$ 
with respect to the filtration \eqref{eq:filt2}. We will show that 
$\det(\Stab_{H}(L_{V}, L_{W})) = \cO^1_c$ for a unique $c$ that is calculated purely in terms of the distance function on the building $\cB(V)$. 

\paragraph{Lines in $\cB(V)$ are apartments.} The following lemma is a basic property of buildings over complete 
local fields proved in great generality in \cite[Cor.2.8.4]{bruhat-tits:1}:  

\begin{lem}\label{lem:maxflats}
All geodesic lines of $\cB(\star)$ are precisely the apartments of $\cB(\star)$.  
\end{lem}

\noindent The statement is known as an \emph{extension of geodesics} property (also discussed in \cite{parreau:immeuble} and \cite{kleiner-leeb}).

\paragraph{Two relevant apartments in $\cB(V)$.}
\noindent There is a unique self-dual lattice $L_{D} = \cO_{k} e_0 \subset D$. Let  
$$
\dist(L_{V}, L_{W}) = n \qquad \text{and} \qquad \dist(L_{V}, \pr_{W}(L_{V})) =d. 
$$
Since $\cB(V)$ is a tree, $\dist(\pr_{W}(L_{V}), L_{W}) = n-d$. Here, $d$ is the distance from $L_{V}$ to the sub-building $\cB(W)$ and $n$ is the distance between the two hyperspecial vertices corresponding to $L_{V}$ and $L_{W} \oplus L_{D}$. Consider two apartments $\cA$ and $\widetilde{\cA}$ that will be used in the computation: 

\begin{itemize}
\item $\cA$: an apartment containing the two hyperspecial vertices $\pr_{W}(L_{V}) \oplus L_{D}$ and $L_{W} \oplus L_{D}$ and contained entirely in $\cB(W)$ is shown on Fig.~\ref{fig:apts}. In addition, we choose a Witt basis $\mathscr B = \{e_+, e_0, e_-\}$ for $\cA$ such that 
\begin{equation}\label{eq:A}
\pr_{W}(L_{V}) \oplus L_{D} = \langle e_+, e_0, e_- \rangle \qquad \text{and} \qquad 
L_{W} \oplus L_{D} = \langle \varpi^{-(n-d)} e_+ , e_0, \varpi^{n-d} e_- \rangle. 
\end{equation}

\item $\widetilde{\cA}$: an apartment containing the three hyperspecial vertices of $\cB(V)$ corresponding to the self-dual lattices $\{L_{V}, \pr_{W}(L_{V}), L_{W}\}$ and intersecting $\cA$ in a half-line contained in $\cB(W)$ whose end point is $\pr_{W}(L_{V})$. Such an apartment exists: choose a geodesic line $\ell$ of $\cB(V)$ containing the three points $\{L_{V}, \pr_{W}(L_{V}), L_{W}\}$ such that $\ell \cap \cB(W)$ is the half-line starting at $\pr_W(L_V)$ and use Lemma~\ref{lem:maxflats} to deduce that such a line corresponds to an apartment $\widetilde{\cA}$ of $\cB(V)$. 
In our case, $\widetilde{\cA}$ can be visualized as in Fig.~\ref{fig:apts}. 
\end{itemize} 
Note that the common half-apartment $\cA \cap \widetilde{\cA}$ is determined by the isotropic line $ke_+$. 

\begin{figure}
\begin{center}
\begin{picture}(300,200)

\put(0, 0){\line(250, 0){250}}
\put(0, 0){\line(1,1){100}}
\put(250, 0){\line(1,1){100}}
\put(100,100){\line(1,0){250}}
\put(175,50){\circle*{6}}
\put(90, 35){$\cA$}
\put(153,75){$d$}
\put(145, 35){$\pr_{W}(L_{V})$}
\put(200,55){$n-d$}
\put(250,50){\circle*{6}}
\put(235, 35){$L_{W}$}
\put(155,112){\circle*{6}}
\put(163,110){$L_{V}$}
\put(125, 145){$\widetilde{\cA}$}

\put(175,50){\line(1,0){100}}
\multiput(175,50)(-20,0){5}{\line(-1,0){10}}
\put(25, 10){$\cB(W)$}
\put(175,50){\line(-1,3){40}}

\put(25, 110){$\cB(V)$}
\end{picture}
\end{center}
\caption{The choice of the two apartments $\cA$ and $\widetilde{\cA}$.}
\label{fig:apts}
\end{figure}

\paragraph{A Witt basis $\widetilde{\mathscr B}$ for $\widetilde{\cA}$.}
Call a Witt basis $\widetilde{\mathscr B} = \{\widetilde{e}_+, \widetilde{e}_0, \widetilde{e}_-\}$ for $\widetilde{\cA}$ \emph{suitable} if the following conditions are satisfied 
\begin{enumerate}
\item $\langle \widetilde{e}_0, \widetilde{e}_0 \rangle = 1$, 
\item 
$
\pr_{W}(L_{V}) \oplus L_{D} = \langle \widetilde{e}_+, \widetilde{e}_0, \widetilde{e}_- \rangle \qquad \text{and} \qquad 
L_{W} \oplus L_{D} = \langle \varpi^d \widetilde{e}_+, \widetilde{e}_0, \varpi^{-d}\widetilde{e}_- \rangle.
$
\end{enumerate}
\noindent Given a suitable Witt basis $\widetilde{\mathscr B}$, let $S_{\widetilde{\mathscr B}} \colon V \ra V$ be the linear transformation for which $S_{\widetilde{\mathscr B}}(e_{\pm}) = \widetilde{e}_{\pm}$ and 
$S_{\widetilde{\mathscr B}}(e_0) = \widetilde{e}_0$. Since $\langle \widetilde{e}_0, \widetilde{e}_0 \rangle  = \langle e_0, e_0\rangle  = 1$ and both $\mathscr B$ and $\widetilde{\mathscr B}$ are Witt bases, $S_{\widetilde{\mathscr B}} \in G_V$. 
By abuse of notation, let $S_{\widetilde{\mathscr B}}$ be the matrix representing the unitary isometry $S_{\widetilde{\mathscr B}}$ with respect to the Witt basis $\mathscr B$ (taken in the order $\{e_+, e_0, e_-\}$).
  
\begin{lem}\label{lem:Scong}
There exists a suitable Witt basis $\widetilde{\mathscr B}$ for $\widetilde{\cA}$ such that the matrix $S_{\widetilde{\mathscr B}}$ is of the form 
$$
S_{\widetilde{\mathscr B}} = {\mthree {1} {\beta} {\gamma} {0} {1} {-\overline{\beta}} {0} {0} {1}}, \qquad \beta, \gamma \in \cO_{k}^\times, \ \beta \overline{\beta} + \gamma + \overline{\gamma} = 0.  
$$
\end{lem}

\begin{proof}
Choose any suitable Witt basis $\mathscr B'$ for $\widetilde{A}$ and observe that since $S_{\mathscr B'}$ is the change-of-basis matrix for two bases of the same lattice $\pr_{W}(L_{V}) \oplus L_{D}$, then 
$S_{\mathscr B'} \in \GL_3(\cO_{k})$. Using that  
$$
\langle \varpi^{-m} \widetilde{e}_+, \widetilde{e}_0, \varpi^{m} \widetilde{e}_- \rangle = 
\langle \varpi^{-m} e_+, e_0, \varpi^m e_- \rangle, \qquad \text{for each } m\geq 0, 
$$
but 
$\langle \varpi  \widetilde{e}_+, \widetilde{e}_0, \widetilde{e}_- \rangle \ne \langle \varpi e_+, e_0, e_- \rangle$ (the latter is a consequence of the fact that the apartments $\cA$ and $\widetilde{\cA}$ diverge at the almost self-dual lattices $L = \langle \varpi e_+, e_0, e_- \rangle$ and $\widetilde{L} = \langle \varpi  \widetilde{e}_+, \widetilde{e}_0, \widetilde{e}_- \rangle$), 
we get $\delta_V^m S_{\mathscr B'} \delta_V^{-m} \in \GL_3(\cO_{k})$ for all $m \geq 0$ 
(recall that $\delta_V = \diag(\varpi, 1, \varpi^{-1})$), but 
$$
\diag(\varpi^{-1}, 1, 1) S_{\mathscr B'} \diag(\varpi, 1, 1) \notin \GL_3(\cO_{k})
$$ 
Here, we are using that $\delta_V^m S_{\mathscr B'} \delta_V^{-m}$ is change-of-basis matrix from the basis 
$\{\varpi^{-m} e_+, e_0, \varpi^m e_-\}$ to the basis $\{ \varpi^{-m} \widetilde{e}_+, \widetilde{e}_0, \varpi^{m} \widetilde{e}_- \}$). The first condition implies that $S_{\mathscr B'}$ is upper-triangular. 
In particular, $\widetilde{e}_+ = u_+ e_+$ and $\widetilde{e}_0 = u_0 e_0 + v_+ e_+$ for some 
$u_+, u_0 \in \cO_k^\times$ with $\overline{u_0} u_0 = 1$. Replacing the Witt basis 
$\mathscr B' = \{\widetilde{e}_+, \widetilde{e}_0, \widetilde{e}_-\}$ with the (suitable) Witt basis 
$\widetilde{\mathscr B} = \{u_+^{-1} \widetilde{e}_+, u_0^{-1} \widetilde{e}_0, \overline{u}_+ \widetilde{e}_-\}$, 
we obtain that  
$$
S_{\widetilde{\mathscr B}} = {\mthree {1} {\beta} {\gamma} {0} {1} {-\overline{\beta}} {0} {0} {1}}, \qquad \beta, \gamma \in \cO_{k},\ \ \ \beta \overline{\beta} + \gamma + \overline{\gamma} = 0.  
$$  
The fact that $\langle \varpi  \widetilde{e}_+, \widetilde{e}_0, \widetilde{e}_- \rangle \ne \langle \varpi e_+, e_0, e_- \rangle$ yields 
\begin{equation}\label{eq:notingl3}
{\mthree {1} {\varpi^{-1} \beta} {\varpi^{-1}\gamma} {0} {1} {-\overline{\beta}} {0} {0} {1}} \notin \GL_3(\cO_{k}),
\end{equation}
which implies that $\gamma \in \cO_{k}^\times$. Indeed, if $v(\gamma) > 0$ then $\beta \overline{\beta} + \gamma + \overline{\gamma} = 0$ implies that $v(\beta) > 0$ which is a contradiction with \eqref{eq:notingl3}. It remains to show that 
$\beta \in \cO_{k}^\times$. To prove this, we use the fact that the self-dual lattice $\widetilde{L}_{-1} = \langle \varpi  \widetilde{e}_+, \widetilde{e}_0, \varpi^{-1} \widetilde{e}_- \rangle \notin \cB(W)$ (since the intersection $\cB(W) \cap \widetilde{\cA}$ is a half-line whose end-point is $\langle \widetilde{e}_+, \widetilde{e}_0, \widetilde{e}_- \rangle = \langle e_+, e_0 ,e_-\rangle$). Note that if 
$\varpi^{-1} \beta \in \cO_{k}$ then $\widetilde{L}_{-1} \cap k e_0 = \cO_{k} e_0$ which is equivalent to $\widetilde{L}_{-1} \in \cB(W)$ (we have used that $\widetilde{L}_{-1} \cap k e_0 \subseteq (\widetilde{L}_{-1} \cap k e_0)^\vee$ since $\widetilde{L}_{-1}$ is self-dual). Hence, $\beta \in \cO_{k}^\times$ which proves the lemma. 
\end{proof}

\paragraph{Computing $\Stab_{H}(L_{V}, L_{W})$.}
Since $\dist (L_{W}, \pr_{W} (L_{V})) = n-d$, 
$\Stab_{H}(L_{W}, \pr_{W} (L_{V, {\tau}}))$ computed 
with respect to the basis $\mathscr B$ and viewed as a subgroup of $G_V$ is 
\begin{equation}\label{eq:stab1}
\Stab_{H}(\pr_{W} (L_{V}), L_{W}) =  G_V \cap  
{\mthree {\cO_{k}} {} {\cO_{k}} {} {1} {} {\varpi^{2(n-d)} \cO_{k}} {} {\cO_{k}}} \subset G_V. 
\end{equation} 
Here, $G_V$ is viewed as a subgroup of $\GL_3(\cO_{k})$ with respect to the Witt basis $\mathscr B$ and we have used the fact that the stabilizer belongs to $\delta_V^{-(n-d)} \GL_3(\cO_{k}) \delta_V^{n-d} \cap \GL_3(\cO_{k})$. 
With respect to the basis $\widetilde{\mathscr B}$, $L_{V} = \langle \varpi^d \widetilde{e}_+, \widetilde{e}_0, \varpi^{-d}\widetilde{e}_- \rangle$ and 
$L_{W} = \langle \varpi^{-(n-d)} \widetilde{e}_+, \widetilde{e}_0, \varpi^{n-d}\widetilde{e}_- \rangle$, i.e., 
\begin{equation}\label{eq:stab2}
\Stab_{G_V}(L_{V}, L_{W}) = G_V \cap {\mthree {\cO_{k}} {\varpi^d \cO_{k}} {\varpi^{2d} \cO_{k}} {\varpi^{n-d} \cO_{k}} {\cO_{k} } {\varpi^{d}\cO_{k}} {\varpi^{2(n-d)} \cO_{k}} {\varpi^{n-d} \cO_{k}} {\cO_{k}}}. 
\end{equation}
Hence, 
\begin{equation}\label{eq:stabLVLW}
\Stab_{H_{\tau}}(L_{V}, L_{W}) = G_V \cap S_{\widetilde{\mathscr B}}^{-1}{\mthree {\cO_{k}} {\varpi^d \cO_{k}} {\varpi^{2d} \cO_{k}} {\varpi^{n-d} \cO_{k}} {\cO_{k} } {\varpi^{d}\cO_{k}} {\varpi^{2(n-d)} \cO_{k}} {\varpi^{n-d} \cO_{k}} {\cO_{k}}}S_{\widetilde{\mathscr B}} \cap   
{\mthree {\cO_{k}} {} {\cO_{k}} {} {1} {} {\varpi^{2(n-d)} \cO_{k}} {} {\cO_{k}}}. 
\end{equation}

\paragraph{Computing $\det (\Stab_{H}(L_{V}, L_{W}))$.}
Let $A = {\mthree {x_{11}} {} {x_{13}} {} {1} {} {x_{31}} {} {x_{33}}}$ in the above intersection (here, $v(x_{11}), v(x_{13}), v(x_{33}) \geq 0$ and $v(x_{31}) \geq 2(n-d)$. The intersection condition then means that there exists a matrix $B = {\mthree {y_{11}} {y_{12}} {y_{13}} {y_{21}} {y_{22}} {y_{23}} {y_{31}} {y_{32}} {y_{33}}}$ with $v(y_{12}), v(y_{23}) \geq n-d$, $v(y_{13}) \geq 2d$, $v(y_{21}), v(y_{32}) \geq n-d$, $v(y_{31}) \geq 2(n-d)$ such that $SA = BS$, i.e.,  
\begin{equation}\label{eq:stab}
{\mthree {1} {\beta} {\gamma} {0} {1} {-\overline{\beta}} {0} {0} {1}} \cdot {\mthree {x_{11}} {} {x_{13}} {} {1} {} {x_{31}} {} {x_{33}}} =  {\mthree {y_{11}} {y_{12}} {y_{13}} {y_{21}} {y_{22}} {y_{23}} {y_{31}} {y_{32}} {y_{33}}} \cdot 
{\mthree {1} {\beta} {\gamma} {0} {1} {-\overline{\beta}} {0} {0} {1}}.  
\end{equation}

\begin{lem}
If $c = \min(d, 2(n-d))$ then 
$
\det ( \Stab_{H}(L_{V}, L_{W})) =  \cO^1_c. 
$
\end{lem}

\begin{proof}
By comparing the entries on the left and the right-hand sides of \eqref{eq:stab}, we obtain: 
\begin{itemize}
\item {\bf $(1,1)$:} $v(x_{11} - y_{11}) \geq 2(n-d)$, 
\item {\bf $(1,2)$:} $v(x_{11} - 1) \geq \min(d, 2(n-d))$,  
\item {\bf $(2,1)$:} $v(y_{21}) \geq 2(n-d)$ since $y_{21} = \overline{\beta} x_{31}$, 
\item {\bf $(2,2)$:} $v(y_{22} - 1) \geq 2(n-d)$, 
\item {\bf $(2,3)$:} $v(x_{33} - 1) \geq \min (d, 2(n-d))$. 
\end{itemize}
This means that if $c = \min(d, 2(n-d))$ then $\det(A) = x_{11} x_{33} - x_{13} x_{31}$ satisfies $v(\det(A) - 1) \geq c$ so $\det(A) \in \cO^1_c$ by Lemma~\ref{lem:filt} (one needs $c > 0$ to apply the lemma, but note that if $c = 0$ there is nothing to prove). Conversely, take any element $a + \eta b \in \cO_{c}^\times$ (here, $a \in \cO_{k_0}^\times$ and $v(b) = c$) and let $\lambda = 1 + \eta a^{-1} b \in \cO_{c}^\times$ where $\eta \in \cO_{k}$ is as in the proof of Lemma~\ref{lem:filt}. Consider the element 
$\ds \lambda / \overline{\lambda} \in \cO^1_c$. It remains to show that there exists $A \in \Stab_{H}(L_{V}, L_{W})$ such that $\det(A) = \lambda / \overline{\lambda}$.  
There are two cases we will consider: 

\vspace{0.1in}

\noindent {\bf Case 1: $d \leq 2(n-d)$.} In this case $c = d$ and we will write a upper-triangular matrix with respect to the basis $\widetilde{\mathscr B}$ stabilizing both $L_{V}$ and $L_{W} \oplus L_{D}$ whose determinant is $\lambda/\overline{\lambda}$ and whose conjugate under $S$ will be exactly of the form 
${\mthree {\star} {} {\star} {} {1} {} {\star} {} {\star}}$ (i.e., it will give us an element $h \in H$ whose determinant is $\lambda / \overline{\lambda}$). We will be looking for a matrix of the form 
$$
B =  
{\mthree {1} {x} {y} {} {1} {-\overline{x}} {} {} {1}} \cdot {\mthree {\lambda} {} {} {} {1} {} {} {} {\overline{\lambda}^{-1}}} = {\mthree {\lambda} {x} {\overline{\lambda}^{-1} y} {} {1} { -\overline{\lambda}^{-1}\overline{x}} {} {} {\overline{\lambda}^{-1}}},  
$$
where $x \overline{x} + y + \overline{y} = 0$ (i.e., a unitary matrix). Moreover, the matrix leaves $L_{V}$ stable if and only if $v(x) \geq c$ and $v(y) \geq 2c$ (note that as it is lower-triangular, it always leaves $L_{W} \oplus L_{D}$ stable). We now calculate 
$$
S^{-1}_{\mathscr B} B S_{\mathscr B} = {\mthree {1} {-\beta} {\overline{\gamma}} {} {1} {\overline{\beta}} {} {} {1}} \cdot  {\mthree {\lambda} {x} {\overline{\lambda}^{-1} y} {} {1} { -\overline{\lambda}^{-1}\overline{x}} {} {} {\overline{\lambda}^{-1}}} \cdot 
{\mthree {1} {\beta} {\gamma} {} {1} {-\overline{\beta}} {} {} {1}} = 
{\mthree {\lambda} {x + (\lambda-1)\beta} {C}
{} {1} {-\overline{\lambda}^{-1} \overline{x} + \overline{\beta}(\overline{\lambda}^{-1}-1)} 
{} {} {\overline{\lambda}^{-1}}}, 
$$
where 
$$
C = -\overline{\lambda}^{-1} y + \beta \overline{\lambda}^{-1} \overline{x} + \overline{\lambda}^{-1} \overline{\gamma} + \lambda \gamma - (x - \beta) \overline{\beta}.
$$ 
We thus want to make $x + (\lambda-1)\beta = 0$, i.e., $x = (1-\lambda)\beta$. For this particular $x$, we check immediately that the entry 
$-\overline{\lambda}^{-1} \overline{x} + \overline{\beta}(\overline{\lambda}^{-1}-1) = 0$ as well. In addition, since $v(1-\lambda) \geq c = d$ then $v(x) \geq c$. We only need to choose $y$ so that $v(y) \geq 2c$. But the only constraint on $y$ is that $x \overline{x} + y + \overline{y} = 0$ and hence, we can choose $y = s + \eta t$ where $s= \overline{x} x / 2$ (we are using that $\tau$ is of odd residue characteristic) and 
$t \in \p_\tau^{2c}$ is arbitrary (the latter will guarantee that $v(y) \geq 2c$). 
\vspace{0.1in}

\noindent {\bf Case 2: $d > 2(n-d)$.} In this case $c = 2(n-d)$. Consider the following matrix (in $H$ with respect to the basis $\cB$): 
$$
A = {\mthree {1 - \gamma x} {} {\gamma \overline{\gamma} x} {} {1} {} {x} {} {1 - \overline{\gamma} x}}, \qquad x = \frac{1 - \lambda / \overline{\lambda}}{ \gamma + \overline{\gamma}}.  
$$ 
Note that $x \in \cO_{k}$ as $\gamma + \overline{\gamma} = - \beta \overline{\beta} \in \cO_{k}^{\times}$. We check that $\det(A) = \lambda / \overline{\lambda}$. Moreover, using $x + \overline{x} = (\gamma + \overline{\gamma})x \overline{x}$, we obtain that $\leftexp{t}{\overline{A}} J_3 A = J_3$ where $\ds J_3 = {\mthree 0 0 1 0 1 0 1 0 0}$. Moreover, $v(x) \geq c = 2(n-d)$ as $\lambda / \overline{\lambda} \in U^1(c)$. We only need to check that $B = S_{\mathscr B}AS_{\mathscr B}^{-1}$ is of the form given by \eqref{eq:stab2}. But one computes (using $\beta \overline{\beta} = - \gamma - \overline{\gamma}$) 
$$
S_{\mathscr B}AS_{\mathscr B}^{-1} = {\mthree {1} {0} {0} {-\overline{\beta}x} {1 + \beta \overline{\beta} x} {0} {x} {-\beta x} {1}}, 
$$
which proves that $A$ stabilizes the pair $(L_{V}, L_{W})$. This proves the lemma. 
\end{proof}

\subsection{Local invariants of Galois orbits}
Given $x = (x_{V}, x_{W}) \in \Hyp = \Hyp_{V} \times \Hyp_{W}$, define $\inv(x)$ to be the pair $(a, b)$ where $a = \dist(x_{V}, \pr_{W}(x_{V}))$ and $b = \dist(\pr_{W}(x_{V}), x_{W})$. 

\begin{prop}\label{prop:orbits}
Two points $x, y \in \Hyp$ lie on the same $H$-orbit if and only if $\inv(x) = \inv(y)$. 
\end{prop}

\noindent To prove the proposition, we introduce the notion of a \emph{special apartment} for the building $\cB(V)$. An apartment $\cA$ determined by a Witt basis $\{\overline{e}_+, \overline{e}_-, \overline{e}_0\}$ is called \emph{special} if the intersection $\cA \cap \cB(W)$ is a half-line. Let $\cS$ be the set of special apartments.

\begin{lem}\label{lem:transitive}
The group $H$ acts transitively on $\cS$.  
\end{lem}

\begin{proof}
Any special apartment in $\cS$ is determined by a pair of isotropic lines of $V$. One of these isotropic lines corresponds to a half-apartment of $\cB(W)$ whereas the other line is not a subspace of $W$. Since $H$ acts transitively on the set of isotropic lines of $W$, we can assume that the first of these two lines is exactly $e_+$. This arranges any two special apartments $\cA', \cA'' \in \cS$ to share a common half-line (determined by $e_+$) of $W$ as shown in Fig.~\ref{fig:apts2}. 

The stabilizer of $e_+$ is then the Borel subgroup $B_V \subset G_V$. We will be done if we show that $B_V$ acts transitively on the set of isotropic lines of $V$ that do not belong to $W$. These lines are all of the form $k (s_+ e_+ + e_0 + s_- e_-)$ with $s_+ \overline{s_-} + \overline{s_+} s_- + \langle e_0, e_0 \rangle = 0$; in particular $s_- \ne 0$. Using an element of $T_V \subset B_V$, we can assume that $s_- = 1$, so $s_+ + \overline{s_+} + 1 = 0$. Thus, the proof reduces to showing that $B_V$ acts transitively on the isotropic lines of the form $k(s_+ e_+ + e_0 + e_-)$. But any other such line is of the form $k((s_+ + t)e_+ + e_0 + e_-)$ for some $t \in k$ for which 
$t + \overline{t} = 0$. It then follows that the image in $G_V$ of the unitary matrix $\ds {\mtwo {1} {t} {} {1}} \in U_W \subset B_W$ transforms the isotropic line $k(s_+ e_+ + e_0 + e_-)$ into $k((s_+ + t)e_+ + e_0 + e_-)$. It follows that the an element of $H$ acts transitively on $\cS$.  
\end{proof}

\begin{rem}
Note that one can conclude from the proof that the action of $H$ on $\cS$ is not only transitive, but also faithful (i.e., $\cS$ is an $H$-torsor). In what follows, we are going to use only the transitivity.  
\end{rem}

\begin{figure}
\begin{center}
\begin{picture}(300,200)

\put(0, 0){\line(250, 0){250}}
\put(0, 0){\line(1,1){100}}
\put(250, 0){\line(1,1){100}}
\put(100,100){\line(1,0){250}}
\put(85, 115){$\cA'$}
\put(125, 145){$\cA''$}
\put(175,50){\line(1,0){100}}
\put(175,50){\line(-1,1){90}}
\multiput(175,50)(-20,0){5}{\line(-1,0){10}}
\put(25, 10){$\cB(W)$}
\put(175,50){\line(-1,3){40}}

\put(25, 110){$\cB(V)$}
\end{picture}
\end{center}
\caption{The transitive action of $H$ reduces to the case where $\cA'$ and $\cA''$ share the same intersection with $\cB(W)$. We extend this half-line to an apartment $\cA$ of 
the sub-building $\cB(W)$.}
\label{fig:apts2}
\end{figure}
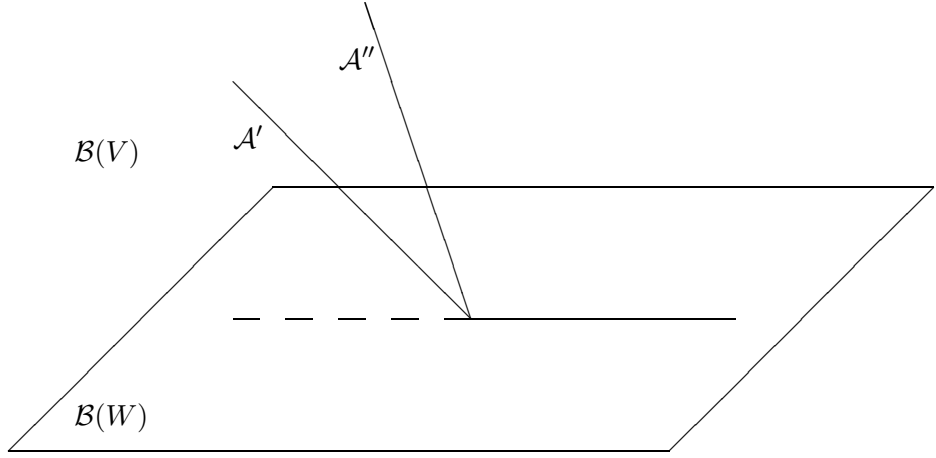

\begin{proof}[Proof of Proposition~\ref{prop:orbits}]
Choose a special apartment $\cA_x$ containing $x_{V}, \pr_{W}(x_{V})$ and $x_{W}$.  Such an apartment exists thanks to Lemma~\ref{lem:maxflats} and the relative position of the buildings $\cB(V)$ and $\cB(W)$ (just choose a line that goes through the three points and intersects $\cB(W)$ in a half-line). Similarly, choose $\cA_y$ containing  $y_{V}, \pr_{W}(y_{V})$ and $y_{W}$. By Lemma~\ref{lem:transitive}, there exists an element $h \in H$ such that 
$h \cA_x = \cA_y$. Since $h$ preserves distances, it follows that $h$ transforms $x \in \Inv$ to $y \in \Inv$.  
\end{proof}

\section{Computing the Hecke Polynomial}\label{sec:heckepol}
\setcounter{paragraph}{0}

We retain the notation from Section~\ref{sec:galois-cycles}. In addition, choose a Witt $k$-basis $\{e_+, e_0, e_-\}$ of $V$ adapted to $K_V$ such that $\{e_+, e_-\}$ is a Witt $k$-basis of $W$ adapted to $K_W$. Let $B_V \subset G_V$ be the Borel subgroup that is the stabilizer of the maximal isotropic flag $k e_+$ of $V$ and let $B_W \subset G_W$ be the Borel subgroup that is the stabilizer of the maximal isotropic flag $ke_+$ of $W$. Let $U_V$ (resp., $U_W$) be the unipotent radical of $B_V$ (resp. $B_W$) and let $T_V$ (resp., $T_W$) be the corresponding Levi subgroups. 

Let $\cH = \cH(G, K)$ be the local Hecke algebra. Let $\delta_V = \diag(\varpi, 1, \varpi^{-1})$ and $\delta_W = \diag(\varpi, \varpi^{-1})$. Let $\mu$ be the conjugacy class of co-characters of $\Glochat$ 
determined by the Shimura datum. Blasius and Rogawski \cite[\S 6]{blasius-rogawski:zeta} associate to $\mu$ a polynomial $H_\tau(z)$ with coefficients in $\cH$ and conjecture that it vanishes on the geometric Frobenius acting on the $\ell$-adic \'etale cohomology of the corresponding Shimura variety (providing an analogue of the classical Eichler--Shimura relation). We now compute 
the polynomial for the Shimura variety $\Sh_K(\G, X)$. Given an element $g \in G$, let $\mathbf{1}_{KgK}$ be the characteristic function of the double coset $K g K$ viewed as an element of the local Hecke algebra $\cH$.

\begin{thm}\label{thm:hecke}
The Hecke polynomial $H_\tau(z) \in \cH[z]$ at the place $\tau$ for the Shimura datum $(\G, X)$ defined in Section~\ref{subsec:shimvar} is given by 
\begin{equation}
H_\tau(z) = H^{(2)}(z) H^{(4)}(z), 
\end{equation}
where 
$$
H^{(2)}(z) = z^2 - \qtau^2(\mathbf{1}_{\Ktau(1, \delta_W)\Ktau} - (\qtau-1) \mathbf{1}_{\Ktau}) z + \qtau^6 \in \cH[z],  
$$ 
and 
$$
H^{(4)}(z) = z^4 + d_1 z^3 + d_2 z^2 + d_3 z + d_4 \in \cH[z]. 
$$
Here, 
\begin{eqnarray*}
d_1 &=& -\mathbf{1}_{\Ktau (\tVtau, \tWtau) \Ktau} + (\qtau-1)( \mathbf{1}_{\Ktau (\tVtau, 1) \Ktau} + (q-1)\mathbf{1}_{\Ktau (1, \tWtau) \Ktau}) - (q-1)^2, \\ 
d_2 &=& \qtau^2 \mathbf{1}_{\Ktau (\tVtau^2, 1)\Ktau} + \qtau^4 \mathbf{1}_{\Ktau (1, \tWtau^2)\Ktau} - 2\qtau^2(\qtau-1) \mathbf{1}_{\Ktau (\tVtau, 1)\Ktau} -2 \qtau^4(\qtau -1) \mathbf{1}_{\Ktau (1, \tWtau)\Ktau} - q^2(q^2+1)(q-1)^2, \\
d_3 &=& \qtau^6 (  -\mathbf{1}_{\Ktau (\tVtau, \tWtau) \Ktau} + (\qtau-1)( \mathbf{1}_{\Ktau (\tVtau, 1) \Ktau} + \mathbf{1}_{\Ktau (1, \tWtau) \Ktau}) - (\qtau-1)^2), \\ 
d_4 &=& \qtau^{12}.  
\end{eqnarray*}

\end{thm}

\subsection{Unramified Local Langlands Correspondence}
We state the conjecture for $G_V$ (it is similar for $G_W$). 

\paragraph{Unramified local parameters.}
The action of the Weil group $W_{k_0}$ on $\widehat{G_V}$ is explained in \cite[\S 1.6]{blasius-rogawski:zeta} and in our case, factors through the projection $W_{k_0} \ra \Gal(k / k_0)$. Let $\leftexp{L}{G_V} = \widehat{G_V} \rtimes W_{k_0}$ be the $L$-group of $G_V$. Let $\Phi \in W_{\Ftau}$ be the Frobenius automorphism and let $v \colon W_{\Ftau} \ra \Z$ be the map that sends an element $w \in W_{k_0}$ to the unique exponent $n$ such that $w$ induces the automorphism $\Phi^n$ when restricted to the residue field of $k_0$. We then have an exact sequence 
$$
0 \ra I \ra W_{k_0} \xra{v} \Z \ra 0, 
$$
where $I \subset W_{k_0}$ is the inertia group. Recall \cite[\S 1.10]{blasius-rogawski:zeta} that a \emph{local parameter}  is a homomorphism  
\begin{equation}
\phi \colon W_{k_0} \times \SU_2(\R) \ra \leftexp{L}{G_V}, 
\end{equation}
such that the composition of $\phi$ with the projection to $\leftexp{L}{\Gtau} \ra W_{\Ftau}$ is the identity and 
$\phi(w)$ is semisimple for all $w \in W_{\Ftau}$. Two parameters $\phi_1$ and $\phi_2$ are equivalent if they are conjugated by an element $g \in \widehat{\Gtau}$. 

To introduce unramified local parameters, note that since $G_V$ is unramified (i.e., $G_V$ is quasi-split over $k_0$ and splits over the unramified extension $k$), the action of $W_{k_0}$ on $\widehat{G_V}$ factors through the map $W_{k_0} \xra{v} \Z$ (equivalently, the inertia group acts trivially), i.e., $\widehat{G_V} \rtimes \Z$ is defined and we have a map $\widehat{G_V} \rtimes W_{k_0} \ra \widehat{G_V} \rtimes \Z$. A local parameter $\phi$ is \emph{unramified} if the following two properties are satisfied: 
\begin{enumerate}
\item $\phi$ is trivial on $\SU_2(\R)$, 
\item The composition $W_{k_0} \xra{\phi} \widehat{G_V} \rtimes W_{k_0} \ra \widehat{G_V} \rtimes \Z$ factors through $v \colon W_{k_0} \ra \Z$ (i.e., the inertia group is in the kernel of the composition).  
\end{enumerate}

Let $\Phi_{\ur}(G_V)$ be the set of equivalence classes of unramified local $L$-parameters. Since an unramified local parameter 
$\phi$ is uniquely determined by the semi-simple element $\phi(\Phi) = g \rtimes \Phi$ then the set $\Phi_{\ur}(G_V)$ of equivalence 
classes of unramified local parameters is in bijection with $\widehat{G_V}$-orbits of semisimple elements $g \rtimes \Phi \in \leftexp{L}{G_V}$. As we will see, the latter are easier to describe for the maximal torus $T_V$. 

\paragraph{Unramified representations and unramified local parameters.}
Let $K_V \subset G_V$ be a fixed hyperspecial maximal compact subgroup. An irreducible and admissible representation $\pi$ of $G_V$ is called \emph{unramified} if $\pi^{K_V} \ne 0$. Let $\Pi_{\ur}(G_V)$ be the set 
of isomorphism classes of unramified representations of $G_V$. Following \cite[Prop.1.12.1]{blasius-rogawski:zeta}, there is a natural bijection between $\Phi_{\ur}(G_V)$ and $\Pi_{\ur}(G_V)$ that we now explain. 
First, it follows from \cite[p.535]{blasius-rogawski:zeta} that there are canonical isomorphisms 
$\Phi_{\ur}(G_V) \isom \Phi_{\ur}(T_V) / \Omega(T_V)$ and $\Pi_{\ur}(G_V) \isom \Pi_{\ur}(T_V) / \Omega(T_V)$ where 
$\Omega(T_V) = N_{G_V}(T_V) / T_V$ is the Weyl group. This reduces the problem of relating unramified representations to unramified local parameters from $G_V$ to the maximal torus $T_V$. Let $S_V \subset T_V$ be the the maximal split (over $k_0$) subtorus of $T_V$. It is proved in \cite[p.534]{blasius-rogawski:zeta} that  
\begin{equation}\label{eq:unramloclang}
\Pi_{\ur}(T_V) \simeq \widehat{S_V} \simeq \Phi_{\ur}(T_V). 
\end{equation}

\comment{
Indeed, 
\begin{enumerate}
\item A character $\chi \colon T_V \ra \C^\times$ is unramified if and only if $T_{V, c} \in \ker \chi$ where $T_{V, c} = T_V \cap K_V$. Thus, 
we get 
$$
\Pi_{\ur}(\Ttau) = \Hom(\Ttau / \Ttauc, \C^\times). 
$$

\item There is a map $\vphi \colon \Ttau \ra \Hom(X^*(\Ttau), \Z)$ given by $t \mapsto \{\chi \mapsto v(\chi(t))\}$. Clearly, 
$\ker(\vphi) = \Ttauc$, and $\text{im}(\vphi) = \Hom(X^*(\Ttau), \Z)^{\sigma}$, i.e., we have an isomorphism 
$$
\vphi \colon \Ttau / \Ttauc \ra \Hom(X^*(\Ttau), \Z)^{\sigma} = X_*(\Ttau)^{\sigma} = X_*(\Stau). 
$$
Hence, $\ds \Pi_{\ur}(\Ttau) \isom \Hom(X_*(\Stau), \C^\times) = \widehat{\Stau}$. 

\item Since $\Phi_{\ur}(T)$ is in bijection with the orbits of $\Ttauhat \rtimes \Phi$ under the adjoint action of $\Ttauhat$. 
The latter are easier to describe since $t_1 \rtimes \Phi$ and $t_2 \rtimes \Phi$ are equivalent if and only if there exists 
$t \in \Ttauhat$ such that 
$$
t_2 \rtimes \Phi = (t \rtimes 1)^{-1} (t_1 \rtimes \Phi) (t \rtimes 1) = \Phi(t) t^{-1} t_1 \rtimes \Phi.  
$$
This means that the $\Ttauhat$-orbits are in bijection with $\Ttauhat / \Ttauhat^{1 - \Phi}$. Yet, the latter is isomorphic to $\Stauhat$
since $X_*(S) \hra X_*(T)$, $X_*(S)  = X_*(T)^{\Phi = 1}$ and $\widehat{T} \isom \Hom(X_*(T), \C^\times)$. 
\end{enumerate}
}

\paragraph{Satake parameters.} In the case of $G_V \times G_W$, the maximal split torus $S = S_V \times S_W$ has dimension $2$ since 
the maximal split tori $S_V$ and $S_W$ of $G_V$ and $G_W$, respectively, are both 1-dimensional. If $\{\alpha, \beta\}$ is the basis for $X_*(S)$ consisting 
of the cocharacters $\alpha(\varpi) \ra \diag(\varpi, 1, \varpi^{-1})$ and $\beta(\varpi) = \diag(\varpi, \varpi^{-1})$ 
then we can identify $\widehat{S} \isom \Hom(X_*(S), \C^\times) \isom (\C^*)^2$. 
Indeed, let 
$$
t_{a, b} = (\diag(\varpi^{a}, 1, \varpi^{-a}), \diag(\varpi^{b}, \varpi^{-b})). 
$$ 
Let $s \colon X_*(S) \ra \C^\times$ be a homomorphism and let $(u, v)$ be 
the images of $(\alpha, \beta)$ in $(\C^\times)^2$. If $\pi(s)$ is the unramified representation corresponding to $s$ under \eqref{eq:unramloclang} (we apply for both $V$ and $W$ and write it on the product group) then $\pi(s)(t_{a, b}) = u^{a} v^b$ determines completely the representation $\pi(s)$. Here, the complex numbers $(u, v) \in (\C^\times)^2$ are known as the Satake parameters of $\pi(s)$.

\subsection{Computing the Hecke Polynomial}\label{subsec:comphecke}
We now recall the definition of the polynomial $H_\tau(z)$ that appears in the Blasius--Rogawski congruence relation 
(Theorem~\ref{thm:congrel}) and compute it in our setting. More precisely, we show the first part of 
the computation in the more general case when $\dim V = n$ and $\dim W = n-1$ and then specialize 
to the case $n = 3$ in the final part.

\paragraph{Hecke polynomials and the congruence relation.} Let $r \colon \Glochat \ra \GL(V)$ be the complex representation of $\Glochat$ of highest weight the cocharacter $\mu$ of the Shimura datum $(\G, X)$. Following Blasius and Rogawski \cite[\S 6]{blasius-rogawski:zeta} (see also the comment in the introduction), we associate to the fixed finite place $\tau$ a polynomial (Hecke polynomial) defined as follows:  
\begin{equation}\label{eq:heckepoldef}
H_{\tau}(z) = \det \left ( z - q^{\dim X} r( g^{\sigma} g) \right ) \in \cH[z].   
\end{equation}
Here, the Hecke algebra 
$\cH = \cH(G, K)$ is identified with the functions on $\Glochat$ invariant under $\sigma$-conjugation, i.e., the automorphism of $\Glochat$ given by $y \mapsto g^\sigma y g^{-1}$. When restricted to the maximal torus $\Tlochat$ of $\Glochat$, the Satake isomorphism identifies the Hecke algebra with the space of functions on $\Tlochat$ that are invariant under both $\sigma$-conjugation and the Weyl group $\Omega(T)$. The strategy to compute the polynomial is then to restrict to the maximal torus $\Tlochat$ where the above determinant can easily be evaluated and then to invert the Satake isomorphism. 



\paragraph{The representation $r \colon \Glochat \ra \GL_{n(n-1)}(\C)$.}
Here, we make explicit the computation of the Hecke polynomial $H_\tau(z)$. 
The Hermitian symmetric domain $X$ for the Shimura datum $(\G, X)$ has dimension $\dim X = 2n-3$. 
The associated co-character $\widehat{\mu_h}$ of $\widehat{G}$ can be determined as follows: the Hermitian symmetric domain $X_V$
is the conjugacy class of the embedding $h_V \colon \mathbf{S} \ra \G_{V, \R}$ given by $(z, \overline{z}) \mapsto \diag(1, \dots, 1, \overline{z} / z)$. The complexification $h_{V, \C} \colon \mathbf{S}_{\C} \ra \G_{V, \C}$ is given by  
$(z_1, z_2) \mapsto \diag(1, \dots, 1, z_2 / z_1)$, i.e., the associated co-character $\mu_V$ of the Shimura datum $(\G_V, X_V)$ is $\lambda \mapsto (1, \dots, 1, \lambda^{-1})$ which corresponds to the character $-\chi_n$ of the dual group $\GL_n(\C)$. The representation of $r_V \colon \GL_n(\C) \ra \GL_n(\C)$ of highest weight $-\chi_n$ is precisely the dual of the standard representation, namely, $A_V \mapsto \leftexp{t}{A_V}^{-1}$ for $A_V \in \GL_n(\C)$. Similarly, the associated co-character $\mu_W$ of $(\G_W, X_W)$ is the character $-\lambda_{n-1}$ of $\GL_{n-1}(\C)$, so the representation $r_W$ is the representation $A_W \mapsto \leftexp{t}{A_W}^{-1}$ for $A_W \in \GL_{n-1}(\C)$. The representation $r \colon \widehat{G} \ra \GL_{n(n-1)}(\C)$ associated to $(\G, X)$ is then an $n(n-1)$-dimensional representation 
that is the tensor product of the two representations $r_V$ and $r_W$ of $\widehat{G_V}$ and $\widehat{G_W}$, respectively, i.e., it is the representation $r \colon \widehat{G} \ra \GL(V \otimes W)$ given by $(A_V, A_W) \mapsto \leftexp{t}{A_V}^{-1} \otimes \leftexp{t}{A_W}^{-1}$. 

\paragraph{Galois action on $\widehat{G}$.} The action of $\Gal(k / k_0)$ on $\widehat{G}$ can be calculated following 
\cite[\S 1.6]{blasius-rogawski:zeta}. Indeed, let $(B, T)$ be the Borel pair and consider the standard splitting for $G$, namely: 
\begin{itemize}
\item $\Btauhat = \Btauhat_V \times \Btauhat_W$ is the product of the upper-triangular Borel subgroups, 
\item $\Ttauhat = \Ttauhat_V \times \Ttauhat_W$ is the product of the diagonal tori, 
\item $\{X_\alpha\}$ is the set of matrices $(a_{ij})_{i,j=1}^n$ where $a_{ij} = \delta_{i k} \delta_{k+1 j}$ for $k=1, \dots, n-1$, 
\item $\{Y_\beta\}$ is the set of matrices $(b_{ij})_{i,j=1}^{n-1}$ where $b_{ij} = \delta_{ik} \delta_{k+1j}$ for $k = 1, \dots, n-2$. 
\end{itemize}
According to \cite[\S 1.8(b)]{blasius-rogawski:zeta}, if 
\[
J_n = \left ( 
\begin{matrix}
{} & {} & {} & {} & {(-1)^{1-1}} \\ 
{} & {} & {} & {\Ddots} & {} \\  
{} & {} & {(-1)^{i-1}} & {} & {} \\ 
{} & {\Ddots} & {} & {} & {} \\
{(-1)^{n-1}} & {} & {} & {} & {} 
\end{matrix}
\right )
\]
then the automorphism 
$
A_V \mapsto J_n \leftexp{t}{A_V}^{-1} J_n
$
is the unique non-inner automorphism of $\widehat{G_V} = \GL_3(\C)$ that fixes the standard splitting $(B_V, T_V, \{X_\alpha\})$. Similarly, 
$
A_W \mapsto J_{n-1} \leftexp{t}{A_W}^{-1} J_{n-1}
$
is the unique non-inner automorphism of $\widehat{G_W} = \GL_2(\C)$ that fixes the standard splitting $(B_W, T_W, \{Y_\beta\})$. 

A Borel pair $(B, T)$ for $G$ gives rise to a reduced based root datum 
$$
\Psi(B, T) = (X^*(T), \Delta^*, X_*(T), \Delta_*), 
$$
where $\Delta^* \subset X^*(T)$ is the set of simple positive roots and $\Delta_* \subset X_*(T)$ is the 
set of co-roots associated to $\Delta^*$. All Borel pairs are conjugate under the action of the 
adjoint quotient $G^{\ad}$ and if $(B', T')$ and $(B, T)$ are two Borel pairs that are conjugate under $g \in G^{\ad}$ (i.e., $\ad(g)$ sends $(B', T')$ to $(B, T)$) then $\ad(g)$ is independent of $G$, i.e., there is a canonical isomorphism between $(B', T')$ and $(B, T)$. This means that all based root data $\Psi(B, T)$ are canonically identified and hence, we can write the datum obtained via these canonical identifications by $\Psi(G)$.

One calls the data $(B, T, \{(X_\alpha, Y_\beta)\})$ a splitting because it splits the exact sequence 
$$
1 \ra G^{\ad} \ra \Aut(\widehat{G}) \ra \Aut(\Psi(\widehat{G})) \ra 1.
$$

The dual group $\widehat{G}$ comes with an isomorphism $\Psi(\widehat{G}) = \Psi(G)^\vee$ where 
$\Psi(G)^\vee$ is the isomorphism class of all $(X_*(T), \Delta_*, X^*(T), \Delta^*)$ (i.e., the roots and the co-roots are switched). 

We obtain the action of $\Gal(\Qbar / \Q)$ on $\widehat{G}$ by lifting the natural action of $\Gal(\Qbar / \Q)$ on $\Psi(G)^\vee$ to the unique automorphism $\Aut(\widehat{G})$ that fixes the chosen splitting. 
More explicitly, $\sigma \in \Gal(\Qbar / \Q)$ acts on $\widehat{G}$ via
$
\leftexp{\sigma}{(A_V, A_W)} = (J_n \leftexp{t}{A_V}^{-1} J_n, J_{n-1} \leftexp{t}{A_W}^{-1} J_{n-1}). 
$ The Hecke polynomial for the Shimura datum $(\G, X)$ is 
\begin{equation}\label{eq:hecke1}
H_\tau(z) = \det(z - q^{2n-3} r(g \leftexp{\sigma}{g})) = \det(z - q^{2n-3} r((A_V, A_W) \cdot \leftexp{\sigma}(A_V, A_W))), 
\end{equation}
where $g = (A_V, A_W) \in \widehat{G}$ with $A_V \in \GL_n(\C)$ and $A_W \in \GL_{n-1}(\C)$. Let 
$$
B := r((A_V, A_W) \cdot \leftexp{\sigma}(A_V, A_W)) \in \GL_{n(n-1)}(\C).
$$ 
Then (see~\cite[p.12]{gross:satake}) 
\begin{equation}
H_\tau(z) = \sum_{i = 0}^{n(n-1)} (-1)^i \Tr \left (\bigwedge^i B \right ) z^{n(n-1) - i}. 
\end{equation}
where $A = A_V \otimes A_W$ for $(A_V, A_W) \in \widehat{G} = \widehat{G_V} \times \widehat{G_W}$. Here, the coefficients of the polynomial are viewed as functions on $\widehat{G}$. Restricted to the dual torus $\widehat{T}$, let 
$$
A_V = \diag(x_1, \dots, x_n) \qquad \text{and} \qquad A_W = \diag(y_1, \dots, y_{n-1}).
$$ 
Then 
$$
(A_V, A_W) \cdot \leftexp{\sigma}(A_V, A_W) = \left ( \diag \left ( \frac{x_1}{x_n}, \dots, \frac{x_n}{x_1}\right ), 
\diag \left ( \frac{y_1}{y_{n-1}}, \dots, \frac{y_{n-1}}{y_1}\right ) \right ). 
$$
In this case, \eqref{eq:hecke1} turns into 
\begin{equation}\label{eq:hecke2}
H_\tau(z) = \prod_{i = 1}^n \prod_{j=1}^{n-1} \left ( z - q^{2n-3} \frac{x_{n+1-i}}{x_i} \frac{y_{n-j}}{y_j}\right ).  
\end{equation}
This polynomial is invariant under $\sigma$-conjugation as well as under the Weyl group $\Omega(T)$. 
When $n = 3$, we rewrite \eqref{eq:hecke2} as 
\begin{eqnarray*}
H_\tau(z) &=& \left ( z - q^3 \frac{y_2}{y_1}\right) \left (z^2 - q^3  \left ( \frac{x_3}{x_1} + \frac{x_1}{x_3}\right ) \frac{y_2}{y_1} z + q^6 \left ( \frac{y_2}{y_1}\right )^2 \right ) \times \\ 
&&  \left ( z - q^3 \frac{y_1}{y_2}\right) \left (z^2 - q^3  \left ( \frac{x_3}{x_1} + \frac{x_1}{x_3}\right ) \frac{y_1}{y_2} z+ q^6 \left ( \frac{y_1}{y_2}\right )^2 \right ). 
\end{eqnarray*}

\paragraph{The coefficients of $H_\tau(z)$ as elements of $\cH(T, T_c)^{\Omega(T)}$.} The above representation yields a factorization into two polynomials (one of degree 2 and one of degree 4) whose coefficients are functions that are invariant under $\sigma$-conjugation and under the Weyl group $\Omega(T)$. Indeed, 
let $\delta_V = \diag(\varpi, 1, \varpi^{-1})$, $\delta_W = \diag(\varpi, \varpi^{-1})$ and let 
\begin{equation}\label{eq:soper}
s_{0, 1} = \mathbf{1}_{(1, \delta_W)T_c} + \mathbf{1}_{(1, \delta_W^{-1})T_c} \in \cH(T, T_c), \qquad 
s_{1,0} =  \mathbf{1}_{(\delta_V,1)T_c} + \mathbf{1}_{(\delta_V^{-1}, 1)T_c} \in \cH(T, T_c). 
\end{equation}
The Hecke polynomial can then be written as follows: 
\begin{equation}\label{eq:hecke-torus}
H_\tau(z) = \underbrace{(z^2  -q^3 s_{0, 1} z + q^6)}_{H^{(2)}(z)} \underbrace{(z^4 -q^3 s_{0,1}s_{1,0} z^3 + q^6 \left (  s_{0,1}^2 + s_{1, 0}^2 - 2 \right ) z^2 + -q^9 s_{0,1} s_{1, 0} z + q^{12})}_{H^{(4)}(z)},   
\end{equation}
viewed as a polynomial in $\cH(T, T_c)[z]$. We now need to obtain the polynomial with coefficients in the original Hecke algebra $\cH(G, K)$ by inverting the Satake transform. 

\subsection{The Satake Isomorphism}\label{subsec:satake}
Satake \cite{satake:spherical} showed that there is an isomorphism  $\cH \otimes \Z[q^{\pm 1/2}] \isom \cH(T, T_c)^{\Omega(T)} \otimes \Z[q^{\pm 1/2}]$, $T_c = T \cap K$ and $\Omega(T) = N_G(T) / T$ is the Weyl group. The isomorphism is defined via the following commutative diagram: 
$$
\xymatrix{
\cH(G, K) \otimes \Z[q^{\pm 1/2}] \ar@{^{(}->}[r]^{\widehat{}} \ar@{->}[d]^{|_{B}}& \cH(T, T_c) \otimes \Z[q^{\pm 1/2}] \\ 
\cH(B, L) \otimes \Z[q^{\pm 1/2}] \ar@{->}[r]^{\mathcal S}  &  \cH(T, T_c) \otimes \Z[q^{\pm 1/2}]\ar@{->}[u]_{|\delta|^{1/2}}. 
}
$$
Here, $L = B \cap K$ and $|_{B} \colon \cH(G, K) \ra \cH(B, L)$ is the restriction of functions, 
the map $\cS \colon \cH(B, L) \ra \cH(T, T_c)$ is defined by $\cS(\mathbf{1}_{L g L })  = [L \cap gL g^{-1}] \mathbf{1}_{g T_c}$ for $g \in T$, i.e., it is obtained by taking quotients by the unipotent radical, $\delta$ is the sum of the simple positive roots (in other words, the character $\delta$ of $T$ is obtained by looking at the action of $T$ on the Lie algebra of the unipotent radical $U$ of $B$), and $|\cdot|$ is normalized so that $|\varpi| = q^{-2}$. In fact, the above diagram gives an algebra homomorphism $\cS \circ |_{B} \colon \cH(G, K) \ra \cH(T, T_c)$ (the one inducing the usual Satake isomorphism), and a twisted version $|\delta|^{1/2} \circ \cS \circ |_{B} \colon \cH(G, K) \otimes \Z[q^{\pm 1/2}]\ra \cH(T, T_c) \otimes \Z[q^{\pm 1/2}]$ that we also denote by $\widehat{\cdot}$. It is explained in 
Wedhorn \cite[Prop.1.9]{wedhorn:congruence} that one has the following commutative diagram: 
$$
\xymatrix{
\cH_{\C}(G, K) \ar@{^{(}->}[r]^{|_{\Bbf}} \ar[d]^{\isom}&   \cH_{\C}(B, L)   \ar@{^{(}->}[r]^{\cS}& \cH_{\C}(T, T_c) \ar[r]^{|\delta|^{1/2}} &  \cH_{\C}(T, T_c) \ar[d]^{\isom} \\ 
\C[\Pi_{\ur}(G)] \ar[r]^{\isom} & \C[\Pi_{\ur}(T)]^{\Omega(T)} \ar@{^{(}->}[rr] && \C[\Pi_{\ur}(T)]. 
}
$$
Finally, in the case when $\delta^{1/2}$ takes values in the subgroup $q^{\Z}$, one has an isomorphism 
$$
\cH \otimes \Z[q^{-1}] \isom \cH(T, T_c)^{\Omega(T)} \otimes \Z[q^{-1}].
$$ 

\subsection{Inverting the Satake transform using buildings.}\label{subsec:satinvert}
Let $\Hyp_V$ (resp., $\Hyp_W$) denote the set of hyperspecial vertices on the building $\cB(G_V)$ (resp., $\cB(G_W)$) and let $\Hyp = \Hyp_V \times \Hyp_W$. Fix a Witt basis for $V$ and let $\A_V$ be the corresponding apartment. A choice of a fundamental chamber $\cC_V$ of $\cA_V$ gives a canonical retraction map $\rho_{\cA_V, \cC_V} \colon \cB(G_V) \ra \cA_V$ \cite[p.53]{garrett:buildings}. 
Similarly, we get a canonical retraction map $\rho_{\cA_W, \cC_W} \colon  \cB(G_W) \ra \cA_W$. 
 Let $x_0 = (x_{0, V}, x_{0, W}) \in \Hyp$ be the pair of hyperspecial vertices determined by the pair of self-dual lattices $(\langle e_+, e_0, e_-\rangle, \langle e_+, e_0\rangle)$ (the stabilizer of that pair in $G$ is $K$). If $\mu \in X_*(T)$ is a dominant co-character, the Hecke operator $\mathbf{1}_{\Ktau \mu(\varpi) \Ktau}$ acts on $x_0$ via: 
\begin{equation}
\mathbf{1}_{K \mu(\varpi) K}  \cdot x_0 = \sum_{x \in \cB(G), \mu(x, x_0) = \mu} x. 
\end{equation} 
Here, the sum is taken over all hyperspecial points $x \in \Hyp$ whose relative position (to the root $x_0$) is given by the co-character $\mu$. The notation $\mu(x, x_0)$ means the following: by the theory of elementary divisors, the position of $x_V$ relative to $x_{V, 0}$ determines a co-character $\mu_V \in X_*(T_V)$. Similarly, we get a co-character $\mu_W \in X_*(T_W)$ from $(x_W, x_{W, 0})$. 
We then define $\mu(x, x_0) = (\mu_V, \mu_W)$ and prove the following result using an idea of Cornut: 

\begin{prop}\label{prop:retract}
For a dominant co-character $\mu \in X_*(T)$, we have   
\begin{equation}\label{eq:retr2}
\cS \mathbf{1}_{\Ktau \mu(\varpi) \Ktau} \cdot x = \sum_{\substack{x = (x_V, x_W) \in \cB(G), \\ \mu(x, x_{0}) = \mu}} (\rho_{\cA_V, \cC_V}(x_V), \rho_{\cA_W, \cC_W}(x_W)). 
\end{equation} 
\end{prop}
\noindent The proof uses an auxiliary lemma: 
\begin{lem}\label{lem:canret}
Let $U_V$ be the unipotent radical of $B_V$. Each $U_{V}$-orbit of $\Hyp_{V}$ intersects the apartment $\cA_V$ at a unique point (i.e., the apartment $\cA_V$ is a fundamental domain for the action of $U_V$ on $\cB(G_V)$). Similarly, if $U_W$ is the unipotent radical of $B_W$ then each $U_W$-orbit of $\Hyp_{W}$ intersects the apartment $\cA_W$ at a unique point. 
\end{lem}

\begin{proof}
It suffices to prove the statement for $V$ as the case for $W$ is identical. 
We have  $\cA_V \cap \Hyp_V = T_{V} x_0$. Any element $x_V \in \Hyp_V$ is of the form $b_V x_{V, 0}$ (since $G_V = B_V  K_V$ and $K_V = \Stab_{G_V}(x_{V, 0})$). Since $b_V = u_V t_V$ for $u_V \in U_V$ and $t_V \in T_V$, it follows that $t_V x_0 \in U_V x_V \cap \cA_V$, i.e., the $U_V$-orbit $U_V x_V$ intersects the apartment in at least one point. Suppose now that there are two points $t_1 x_0$ and $t_2 x_0$ that are in the same $U_V$-orbit. Then there is $u \in U_V$ such that $t_2 x_0 = ut_1x_0$, i.e., $t_2^{-1} u t_1 \in \Ktau$. The latter means (looking only at the diagonal entries for the matrix representation with respect to the Witt basis for $V = W \perp D$) that $t_2^{-1} t_1$ has entries in $\cO_{E}$, i.e., $t_2^{-1} t_1 \in \Ktau$ and hence, $t_1 x_0 = t_2 x_0$, since $K = \Stab_{G}(x_0)$.  
\end{proof}

\begin{rem}\label{rem:canret}
Lemma~\ref{lem:canret} characterizes the canonical retraction completely and is often used as a definition for $\rho_{\cA_V, \cC_V}$ (resp., $\rho_{\cA_W, \cC_W}$): since $\cA_V$ is a fundamental domain for the action of $U_V$ on $\cB(G_V)$, we define $\rho_{\cA_W, \cC_W}(x)$ to be the point corresponding to $x$ on this fundamental domain. One then deduces the distance-preserving property of the canonical retraction as a consequence. Recall that the latter says that if $y_V \in \cA_V$ (resp., $y_W \in \cA_W$) is a point very far in the ray determined by the chamber $\cC_V$ then  
$\dist (y_V, x_V) = \dist(y_V, \rho_{\cA_V, \cC_V}(x_V))$ for all $x_V \in \cB(G_V)$. 
\end{rem}

\begin{rem}
The lemma can also be proved via the following more general argument: the pointwise stabilizer $U_V(0) \subset U_V$ of the chamber $\cC_V \subset \cA_V$ is a compact open subgroup. Given $t \in T_V$, the translated chamber $t \cC_V$ has a stabilizer 
$U_V(t) = t U_V(0)t^{-1}$. We also have $\bigcup_{t \in T_V} U_V(t)= U_V$. Assuming $t_2 x_{V, 0} = u t_1 x_{V, 0}$, choose $t \in T_V$ such that $u \in U(t)$. For any point 
$a \in tC_V$, one has  
$$
\dist ( a, t_2 x_{V, 0}) = \dist(ua, ut_1 x_{V, 0}) = \dist(a, t_1x_{V, 0}). 
$$
Thus, any point $a \in t \cC_V$ is equidistant from $t_1 x_{V, 0}$ and $t_2 x_{V, 0}$. But the set of points in $\cA_V$ that are equidistant from $t_1 x_{V, 0}$ and $t_2 x_{V, 0}$ is a hyperplane unless both points are equal, i.e., $t_1 x_{V, 0} = t_2 x_{V, 0}$. 
\end{rem}

\begin{proof}{(Proof of Proposition~\ref{prop:retract})}
The Hecke algebra $\cH(G, K)$ is isomorphic to $\Z[K \backslash G / K]$ and the latter is isomorphic to $\End_{\Q[G]}(\Z[G /\Ktau])$ (to give a $\Z[G]$-equivariant endomorphism $\vphi$ of $\Z[G / \Ktau]$, it suffices to specify $\vphi(K)$ that is $K$-invariant). We thus have  
$$
\cH(G, K) \isom \End_{\Z[G]}(\Z[G /  \Ktau]) \isom \End_{\Z[G]}(\Z[\Hyp]). 
$$
Moreover, the restriction map $|_{B}$ is simply 
$$
|_{B} \colon \End_{\Z[G]}(\Z [\Hyp ]) \hra \End_{\Z[B]}(\Z[\Hyp]).
$$
On the level of endomorphisms, the Satake transform $\cS$ is the composition  
$$
\End_{\Z[G]}(\Z[G/K]) \hra \End_{\Z[B]}(\Z[B/L]) \ra \End_{\Z[T]}(\Z[T/T_c]). 
$$
Since $\Hyp = G \cdot x_0 = BK \cdot x_0 = B \cdot x_0$ and since $B  = UT$ where $U$ is the unipotent radical, the last map is induced from the maps $b_V \cdot x_{V, 0} \mapsto t_V \cdot x_{V, 0}$ (here, $b_V \in B_V$, $b_V = u_Vt_V$, $u_V\in U_V$, $t_V \in T_V$) and $b_W \cdot x_{W, 0} \mapsto t_W \cdot x_{W, 0}$ where $b_W \in B_W$, $b_W = u_Wt_W$, $u_W\in U_W$, $t_W \in T_W$. By Lemma~\ref{lem:canret} and Remark~\ref{rem:canret}, this is exactly the map 
$$
\cB(G_V) \times \cB(G_W) \ra \cA_V \times \cA_W, \qquad (x_V, x_W) \mapsto (\rho_{\cA_V, \cC_V}(x_V), \rho_{\cA_W, \cC_W}(x_W)). 
$$
\end{proof}

The latter can now be computed explicitly by counting how many points $(x_V, x_W) \in \cB(G_V) \times \cB(G_W)$ retract to a given point $(y_V, y_W) \in \cA_V \times \cA_W$.  Below we have shown the building $\cB(G_V)$ together with the apartment 
$\cA_V$:

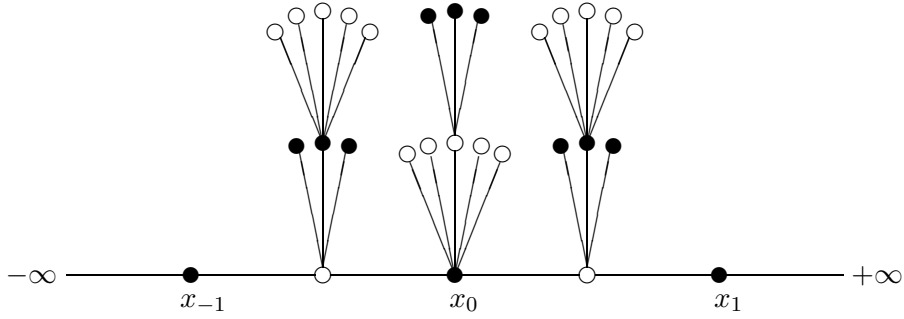
\begin{center}
\begin{picture}(300,100)
	\put(50, 0){\circle*{6}}
	\put(100, 0){\circle{6}}
	\put(150,0){\circle*{6}}
	\put(200,0){\circle{6}}
	\put(250, 0){\circle*{6}}
	
	\put(46, -12){$x_{-1}$}
	\put(148, -12){$x_0$}
	\put(248, -12){$x_{1}$}
	
	\put(90,49){\circle*{6}}
	\put(100,50){\circle*{6}}
	\put(110,49){\circle*{6}}
	
	\put(132,46){\circle{6}}
	\put(140,49){\circle{6}}
	\put(150,50){\circle{6}}
	\put(160,49){\circle{6}}
	\put(168,46){\circle{6}}	
		
	\put(190,49){\circle*{6}}
	\put(200,50){\circle*{6}}
	\put(210,49){\circle*{6}}
	
	\put(140,98){\circle*{6}}
	\put(150,100){\circle*{6}}
	\put(160,98){\circle*{6}}
	
	\put(82,92){\circle{6}}
	\put(90,98){\circle{6}}
	\put(100,100){\circle{6}}
	\put(110,98){\circle{6}}
	\put(118,92){\circle{6}}
	
	\put(182,92){\circle{6}}
	\put(190,98){\circle{6}}
	\put(200,100){\circle{6}}
	\put(210,98){\circle{6}}
	\put(218,92){\circle{6}}
	
	\put(50,0){\line(-1,0){47}}
	\put(50,0){\line(1,0){47}}
	\put(150,0){\line(-1,0){47}}
	\put(150,0){\line(1,0){47}}
	\put(250,0){\line(-1,0){47}}
	\put(250,0){\line(1,0){47}}
	
	\put(150,0){\line(2,5){17}}
	\put(150,0){\line(1,5){9}}
	\put(150,0){\line(0,1){47}}
	\put(150,0){\line(-1,5){9}}
	\put(150,0){\line(-2,5){17}}
	
	\put(100,50){\line(2,5){16}}
	\put(100,50){\line(1,5){9}}
	\put(100,50){\line(0,1){47}}
	\put(100,50){\line(-1,5){9}}
	\put(100,50){\line(-2,5){16}}
	
	\put(200,50){\line(2,5){16}}
	\put(200,50){\line(1,5){9}}
	\put(200,50){\line(0,1){47}}
	\put(200,50){\line(-1,5){9}}
	\put(200,50){\line(-2,5){16}}
	
	\put(100,3){\line(1,5){9}}
	\put(100,3){\line(0,1){47}}
	\put(100,3){\line(-1,5){9}}
	
	\put(200,3){\line(1,5){9}}
	\put(200,3){\line(0,1){47}}
	\put(200,3){\line(-1,5){9}}
	
	\put(150,53){\line(1,5){9}}
	\put(150,53){\line(0,1){47}}
	\put(150,53){\line(-1,5){9}}
	
	\put(-20, -3){$-\infty$}
	\put(300, -3){$+\infty$}
\end{picture}
\end{center}

\vspace{0.1in}

\noindent For brevity, if $a, b \in \Z$, set 
$
t_{a, b} = \mathbf{1}_{K (\diag(\varpi^a, 1, \varpi^{-a}), \diag(\varpi^b, \varpi^{-b})) K}. 
$

\vspace{0.1in}

\noindent \textbf{Computing $\widehat{t_{1,0}}$.} In this case, \eqref{eq:retr2} shows that $\cS(t_{1,0})(x_0)$ is a sum of points 
$x_{-1}, x_0$ and $x_1$. To figure out the multiplicities, we need to figure out the number of points on the sphere $S_2(x_{0, V}) = \{x \in \cB(G_V) \colon \dist(x, x_{0, V}) = 1\}$ that retract to $x_{-1}, x_0$ and $x_1$, respectively. 
\begin{enumerate}
\item $x_1$ occurs with multiplicity 1, 

\item $x_0$ occurs with multiplicity $q+1-2$ as the vertices that retract to $x_0$ are precisely the neighbors of $x_{1/2}$ that are different from $x_0$ and $x_1$, 

\item $x_{-1}$ occurs with multiplicity $1 + (q-1) + (q^3 - 1)q = q^4$. 
\end{enumerate}
Thus, 
$$
|_{B} \circ \cS(t_{1, 0}) =  \mathbf{1}_{(\delta_V, 1) T_c} + q^4 \mathbf{1}_{(\delta_V^{-1}, 1)T_c} + (q-1) \mathbf{1}_{T_c}. 
$$
The twisted Satake transform is then 
\begin{equation}\label{eq:t10}
\widehat{t_{1,0}} = q^2 (\mathbf{1}_{(\delta_V, 1) T_c} + \mathbf{1}_{(\delta_V^{-1}, 1) T_c}) + (q-1) = q^2 s_{1,0} + (q-1). 
\end{equation}

\vspace{0.1in}

\noindent \textbf{Computing $\widehat{t_{0,1}}$.} Similarly, we have that $\cS(t_{0,1})(x_0)$ is a sum of points 
$x_{-1}, x_0, x_1$ with multiplicities given by the number of points on the sphere $S_2(x_{0, W}) = \{y \in \cB(G_W) \colon \dist(y, x_{0, W}) = 1\}$ retracting to 
$x_{-1}, x_0$ and $x_1$, respectively. We have 
\begin{enumerate}
\item $x_1$ with multiplicity 1, 

\item $x_0$ with multiplicity $q+1 - 2$ (all the neighbors of $x_{1/2}$ lie on $\cB(G_W)$),  

\item $x_{-1}$ with multiplicity $1 + (q-1) + (q-1)q = q^2$. 
\end{enumerate}
Thus, 
$$
|_{B} \circ \cS(t_{0, 1}) =  \mathbf{1}_{(1, \delta_W) T_c} + q^2 \mathbf{1}_{(1, \delta_W)T_c} + (q-1) \mathbf{1}_{T_c}.
$$
and hence, 
\begin{equation}\label{eq:t01}
\widehat{t_{0,1}} = q \left (\mathbf{1}_{(1, \delta_W) T_c} + \mathbf{1}_{(1, \delta_W^{-1}) T_c} \right ) + (q-1) = qs_{0, 1} + (q-1). 
\end{equation}

\paragraph{Final form of the Hecke polynomial.} 
Finally, we compute the inverse image of the polynomial \eqref{eq:hecke-torus} under the Satake transform using 
only the identities \eqref{eq:t01} and \eqref{eq:t10} to obtain 
\begin{equation}\label{eq:h2}
H^{(2)}(z) = z^2 - q^2(t_{0,1} - (q-1))z + q^6. 
\end{equation}
and 
\begin{eqnarray}\label{eq:h4}
H^{(4)}(z) &=&  z^4 + (-t_{1,0}t_{0, 1} + (q - 1)t_{1,0} + (q - 1)t_{0,1} - (q-1)^2)z^3 + \\
&+& q^2(t_{1,0}^2 + q^2t_{0,1}^2 -2(q-1) t_{1,0} + -2q^2(q-1)t_{0,1} - 
q^4 - 2q^3 + 2q^2 - 2q + 1)z^2 + \\
&+& q^6(-t_{1,0} t_{0,1} + (q-1)t_{1,0} + (q-1)t_{0,1} - (q-1)^2)z + q^{12}. 
\end{eqnarray}
which completes the proof of Theorem~\ref{thm:hecke}.Ê




%
%
\section{The Congruence Relation of Blasius--Rogawski}\label{sec:cong-blas-rog}
\setcounter{paragraph}{0}

Here, we prove Theorems~\ref{thm:congrelprod} and ~\ref{thm:centraltwist}. 
We then remark that these results, together with recent results of J.-S. Koskivirta \cite{koskivirta:journal} for 
$F = \Q$, imply
Conjecture~\ref{thm:congrel} in our setting when $F = \Q$. 

\subsection{Proof of Theorem~\ref{thm:congrelprod}}
Let $R$ be a commutative ring and let $R_1$ and $R_2$ be two commutative rings that contain $R$. Let $x_1, \dots, x_{d_1} \in R_1$ and $y_1, \dots, y_{d_2} \in R_2$. Consider the polynomial 
$$
\prod_{i = 1}^{d_1} \prod_{j = 1}^{d_2} (z - x_i \otimes y_j) \in (R_1 \otimes R_2)[z]
$$
Since it is symmetric in both $x_i$'s and $y_j$'s, it can be written as 
$$
\sum_{i = 1}^{d_1 d_2 } P_i (u_1, \dots, u_{d_1}) \otimes Q_i(v_1, \dots, v_{d_2})  z^i, 
$$
where $u_j = (-1)^j \sigma_j(x_1, \dots, x_{d_1})$ for $j =1, \dots, d_1$ and $v_k = (-1)^k \sigma_k(y_1, \dots, y_{d_2})$ for $k = 1, \dots, d_2$ are the $j$th and $k$th elementary symmetric polynomials in the variables $x_1, \dots, x_{d_1}$ and $y_1, \dots, y_{d_2}$, respectively. 
Clearly, 
$$
H_1(z) = z^{d_1} + u_1 z^{d_1 - 1} + \dots + u_{d_1}  = \prod_{j = 1}^{d_1} (z - x_j) \in R_1[z], 
$$ 
and 
$$
H_2(z) = z^{d_2} + v_1 z^{d_2 - 1} + \dots + v_{d_2}  = \prod_{k = 1}^{d_2} (z - y_k) \in R_2[z]. 
$$ 
We define the tensor product polynomial $H = H_1 \otimes H_2$ as 
$$
H(z) := \sum_{i = 1}^{d_1 d_2} P_i(u_1, \dots, u_{d_1}) \otimes Q_i(v_1, \dots, v_{d_2}) z^i \in (R_1 \otimes R_2)[z]
$$

\begin{lem}\label{lem:tens}
The element $H(z_1 \otimes z_2)$ belongs to the ideal $\langle H_1(z_1) \otimes 1, 1 \otimes H_2(z_2)\rangle$ of 
the ring $(R_1 \otimes R_2)[z_1 \otimes 1, 1 \otimes z_2]$. 
\end{lem}

\begin{proof}
Using  
$z_1 \otimes z_2 - x_i \otimes y_j = (z_1 - x_i) \otimes z_2 + x_i \otimes (z_2 - y_j)$ and $\ds H(z_1 \otimes z_2) = \prod_{i=1}^{d_1} \prod_{j=1}^{d_2} (z_1 \otimes z_2 - x_i \otimes y_j)$, we obtain 
\begin{eqnarray*}
H(z_1 \otimes z_2) &=& \left ( (z_1 -x_1) \otimes z_2 + x_1 \otimes (z_2 - y_1) \right ) \cdot \\
&\cdot&  \left ( (z_1 - x_1) \otimes z_2 + x_1 \otimes (z_2 - y_2) \right ) \cdot \\ 
&\vdots& \\
&\cdot& \left ( (z_1 - x_1) \otimes z_2 + x_1 \otimes (z_2 - y_{d_2}) \right ) \cdot \\ 
&\cdot& \left ( (z_1 -x_2) \otimes z_2 + x_2 \otimes (z_2 - y_{1}) \right ) \cdot \\
&\vdots& \\ 
&\cdot& \left ( (z_1 -x_{d_1}) \otimes z_2 + x_{d_1} \otimes (z_2 - y_{d_2}) \right )
\end{eqnarray*}
Expanding the product and considering regrouping all the terms containing each factor $(z - x_i)$ at least one, we can write the above product as 
\begin{eqnarray*}
H(z_1 \otimes z_2) &=& (z_1 - x_1) \dots (z_1 - x_{d_1}) A_1(z_1) \otimes B_1(z_2) + A_2(z_1) \otimes (z_2 - y_1) \dots (z_2 - y_{d_2}) B_2(z_2) = \\ 
&=& H_1(z_1) A_1(z_1) \otimes B_1(z_2) + A_2(z_1) \otimes H_2(z_2) B_2(z_2), 
\end{eqnarray*}
where $A_1(z_1), A_2(z_1) \in R_1[z_1]$ and $B_1(z_2), B_2(z_2) \in R_2[z_2]$. 
\end{proof}

\noindent We will deduce Theorem~\ref{thm:congrelprod} from Lemma~\ref{lem:tens} via K\"unneth formula for intersection cohomology: 
\begin{equation}\label{eq:kunneth}
\IH^n(\overline{\Sh_{K}(\G, X)}_{\Qbar}, \Q_\ell) = \bigoplus_{p +q= n} \IH^p(\overline{\Sh_{K_1}(\G_1, X_1)}_{\Qbar}, \Q_\ell ) \otimes  \IH^q(\overline{\Sh_{K_2}(\G_2, X_2)}_{\Qbar}, \Q_\ell ). 
\end{equation} 

The latter requires some justification (especially, since the K\"unneth formula for intersection cohomology fails in general). Blasius and Rogawski state the conjecture for the intersection cohomology of ``middle perversity" for the Baily--Borel compactification \cite[\S 2]{blasius-rogawski:zeta} which by Zucker's conjecture proven by Looijenga \cite{looijenga} and Saper--Stern \cite{saper-stern} is isomorphic to the $L^2$-cohomology. The K\"unneth formula is then a consequence of the work of Cheeger \cite{cheeger} (see also \cite{cohen-goresky-ji}).  

To complete the proof, observe that the Hecke polynomial $H_{\tau}$ at $\tau$ for $(\G, X)$ is the tensor product of the 
Hecke polynomial $H_{1, \tau}$ for $(\G_1, X_1)$ and the Hecke polynomial $H_{2, \tau}$ for $(\G_2, X_2)$.  
Using \eqref{eq:kunneth} and the fact that $H_{i, \tau}(\Fr_\tau)$ vanishes on $\IH^{*}(\overline{\Sh_{K_s}(\G_s, X_s)}_{\Qbar}, \Q_\ell )$ for $s = 1,2$, Lemma~\ref{lem:tens} implies that $H(\Fr_\tau)$ vanishes on $\IH^*(\overline{\Sh_{K}(\G, X)}_{\Qbar}, \Q_\ell )$. 


\subsection{Proof of Theorem~\ref{thm:centraltwist}}
The homomorphism of algebraic groups 
$\omega_{\R} \colon \Sbf_{\R} \ra \Z_{\G_{0, \R}} \hra \G_{0, \R}$ yields a map on complex points  
$$
f \colon \Sh_{K_0}(\G_0, X_0) \ra \Sh_{K_0}(\G_0, \omega X_0), \qquad \G_0(\Q)(g, x)K_0 \ra \G_0(\Q) (g, \omega_{\R} x) K_0. 
$$ 
Here, $x$ is viewed as a homomorphism $x \colon  \Sbf_{\R} \ra \G_{0, \R}$. 
A priori, the map $f$ is a holomorphic map of complex analytic manifolds. By Borel's theorem \cite[Thm.3.14]{milne:shimura}, it is a regular map defined over $\C$. Both Shimura varieties $\Sh_{K_0}(\G_0, X_0)$ and $\Sh_{K_0}(\G_0, \omega X_0)$ have canonical models $\mathscr M$ and $\mathscr M_{\omega}$, respectively, defined over the reflex field $E$. We now show that $f$ yields a morphism defined over a finite abelian extension of $E$ on these canonical models that depends on $K_0$. Moreover, we will explicitly compute that extension. 

To do that, we will restrict $f$ to special points which are Zariski dense in $\mathscr M$ and we will 
use the reciprocity law on special points. 
More precisely, associated to $f$ is a 1-cocycle 
$$
u \colon \Aut(\C/\iota(E)) \ra \Aut \left (\mathscr M_{\C} \right ), \qquad u(\sigma) := f^{-1} (\sigma \cdot f). 
$$
Clearly, $f$ is defined over an extension $L$ of $\iota(E)$ ($L \subset \C$) if and only if 
$u|_{\Aut(\C/L)}$ is trivial. 

\begin{lem}\label{lem:cocyc}
For any special pair $(\Tbf, x)$ on $\Sh_{K_0}(\G_0, X_0)$ and any $\sigma \in \Aut(\C / \iota(E))$, we have 
\begin{equation}\label{eq:cocycle}
\sigma \cdot f([g, x]) = \langle s \rangle_{\omega} f(\sigma \cdot [g, x]),  
\end{equation}
where $s \in \mathbf{A}_E^\times$ is any element for which $\Art_E(s) = \sigma|_{\iota (E^{\ab})}$ and 
$$
\langle s \rangle_{\omega} \colon \Sh_{K_0}(\G_0, X_0) \ra \Sh_{K_0}(\G_0, \omega X_0)
$$ 
is given by $\langle s \rangle_{\omega} \colon [g', x'] \mapsto [\omega(s) g', x']$. 
\end{lem}

\begin{proof}
Let $\mu_x \colon\mathbf{G}_{m, \C} \ra \Tbf_{\C}$ be the associated co-character to $x \in X$ and 
let 
$
\mu_{\omega} \colon \mathbf{G}_{m, \C} \xra{z \mapsto (z, 1)} \mathbf{G}_{m, \C} \times \mathbf{G}_{m, \C} \isom \Sbf_{\C} \xra{\omega_{\C}}  \Tbf_{\C}
$
be the associated cocharacter to $\omega$ (recall that it descends to a cocharacter over $E$). Note that $\omega = \mu_{\omega} \overline{\mu_{\omega}}$. 
Note that $f([g, x]) = [g, y]$ where $y = x \omega$ and $\mu_y = \mu_x \mu_{\omega}$. 
Using the reciprocity law from Section~\ref{par:specialpts}, we obtain
$$
\sigma \cdot f([g, x]) = \sigma \cdot [g, y] = [r_y(s) g, y] = [\mu_{\omega}(s) \overline{\mu_{\omega}(s)} r_x(s) g, y] = [\omega(s) r_x(s)g, y] = \langle s \rangle_{\omega} f(\sigma \cdot [g, x] ), 
$$
which proves the lemma.
\end{proof}

%

\noindent We finally compute the field of definition of $f$: 

\begin{lem}
The morphism $f \colon \Sh_{K_0}(\G_0, X_0) \ra \Sh_{K_0}(\G_0, \omega X_0)$ is defined over the abelian  extension $E(f) / E$ given by the norm subgroup $E^\times (\omega_f^{-1}(\Z_{\G_0}(\Af) \cap K)) \subset \widehat{E}^\times$. 
\end{lem}

\begin{proof}
Using the reciprocity map   
\begin{equation}
\rec_E \colon \mathbf{A}_{E}^\times/ E^\times \twoheadrightarrow  \Gal(E^{\ab} / E), 
\end{equation}
we calculate the field of definition of $f$ by calculating the corresponding norm subgroup of $\mathbf{A}_{E}^\times / E^\times$ of finite index as follows: using that $s_f \in \Z_{\G_0}(\Af)$ and \eqref{eq:cocycle}), the corresponding cocycle $u(\sigma)$ is trivial on $\sigma \in \Gal(E^{\ab} / E)$ if and only if $s_f \in K_0$, i.e., $f$ is defined over the abelian extension $E(f)$ whose norm subgroup is precisely $E^\times \omega_f^{-1}(\Z_{\G_0}(\Af) \cap K) \subset \mathbf{A}_{E}^\times$. 
\end{proof}

\subsection{Relation between the two Hecke polynomials}

We first relate the Hecke polynomials $H_{\tau}(z)$ to $H_{\omega, \tau}(z)$ for the Shimura data $(\G_0, X_0 )$ and $(\G_0, \omega X_0)$, respectively. We then use that the conjecture is already known for $(\G_0, \omega X_0)$ to deduce the conjecture for $(\G_0, X_0)$. 

\paragraph{Comparing $H_{\tau}(z)$ and $H_{\omega, \tau}(z)$.}

Let $r_{X_0} \colon \Ghat_0 \ra \GL_n(\C)$ be the irreducible representation of $\Ghat_0$ 
whose highest weight is 
the co-character of the domain $X_0$ and let $r_{\omega X_0} \colon \Ghat_0 \ra \GL_n(\C)$ be the irreducible representation whose highest weight is the cocharacter associated to $\omega X_0$. 

\begin{lem}
There exists a function $s \colon \Ghat_0 \ra \C^\times$ such that $r_{\omega X_0} = s r_{X_0}$. Moreover, 
$s$ is invariant under the Weyl group $\Omega(\Tbf)$ and its preimage under the Satake transform 
is exactly $\mathbf{1}_{\omega(\varpi)K}$. 
\end{lem}

\begin{proof}
Consider the image of $\mathbf{1}_{\omega(\varpi) K}$ under the Satake transform which is $\mathbf{1}_{\omega(\varpi) T_c}$. First, we have 
a map 
$$
\cH(G_\tau, K_\tau) \hra \Hom(\Pi_{\ur}(G_\tau), \C), \qquad 
$$
If $\pi$ is an unramified representation of $T_\tau$ then $\pi$ yields a homomorphism $\pi \colon T_\tau / T_c \ra \C^\times$ and hence, a function $f_\pi \in \Hom(X_*(S_\tau), \C^\times)$ given by 
$$
f_{\pi}(\mu) = \pi(\mu(\varpi)), \qquad \mu \in X_*(S_\tau). 
$$
\end{proof}


\begin{lem}\label{lem:heckecomparison}
The Hecke polynomials are related as follows: 
\begin{equation}\label{eq:heckerel}
H_{\tau} \left (\mathbf{1}_{\omega(\varpi) K_{0}} z \right ) = 
\mathbf{1}_{\omega(\varpi)^{n} K_0} H_{\omega, \tau}(z). 
\end{equation}
\end{lem}


\begin{proof}
By definition, 
$$
H_\tau(z) = \det \left ( z - q^{d/2} r_{X_0}(g \rtimes \Fr_\tau)^{n_\tau} \right )
\qquad 
\text{and} 
\qquad
H_{\omega, \tau}(z) = \det \left ( z - q^{d/2} r_{\omega X_0} (g \rtimes \Fr_\tau)^{n_\tau}\right ). 
$$
Here, the coefficients are viewed as $\ad(\Ghat_0)$-invariant polynomials on the set $\Ghat_0 \rtimes \Fr_\tau$. Viewed as a polynomial with coefficients in $\C[X_*(T_\tau)]$ and observe that if 
$\chi_\omega \colon \widehat{T}_\tau \ra \Gm$ is the character corresponding to the co-character $\mu_\omega \colon \Gm \ra T_\tau$ associated to $\omega$ then $r_{\omega X_0}|_{\widehat{T}_\tau} = \chi_{\omega} \cdot r_{X_0}|_{\widehat{T}_\tau}$. Now, using the Satake isomorphism for $T_\tau$, we obtain isomorphisms  
$$
\cH_{\C}(T_\tau, T_c) \isom \C[\widehat{S}_\tau] \isom \C[X_*(S_\tau)]. 
$$
A co-character $\mu \in X_*(S_\tau)$ corresponds to the function 
$\mathbf{1}_{\mu(\varpi) T_c} \in \cH(T_\tau, T_c)$. Applying that to $\mu = \mu_{\omega}$ and using that $\mathbf{1}_{\mu(\varpi) T_c} \in \cH(T_\tau, T_c)$ and is exactly the image of 
$\mathbf{1}_{\mu(\varpi) K_0}$ under the Satake isomorphism $\cH(G_{0, \tau}, K_0) \isom \cH(T_\tau, T_c)^{\Omega(T_\tau)}$, we obtain 
\begin{eqnarray*}
H_{\omega, \tau}(\mathbf{1}_{\mu_\omega(\varpi) K_0} \cdot z) &=& 
\det \left ( \mathbf{1}_{\omega(\varpi) T_c} z  - r_{\omega X_0}(g \rtimes \Fr_\tau)^{n_\tau} \right ) = \\
&=& \det \left ( \mathbf{1}_{\omega(\varpi) T_c} z  - 
\mathbf{1}_{\omega(\varpi) T_c} \cdot r_{X_0}(g \rtimes \Fr_\tau)^{n_\tau} \right ) = \mathbf{1}_{\omega(\varpi)^n K_0} \cdot H_{\tau}(z),  
\end{eqnarray*}
which proves the lemma. 
\end{proof}

\subsection{Proof of the conjecture for $(\G_0, X_0)$}
Assume Conjecture~\ref{thm:congrel} holds for $(\G_0, \omega X_0)$, i.e., that $H_{\omega, \tau}(\Fr_\tau) = 0$ acting on the intersection cohomology group (as in the conjecture), where, as before, $\Fr_\tau$ denotes the geometric Frobenius. We first want to show that $H_{\tau}(\Fr_\tau) = 0$, i.e., to deduce the conjecture for the datum $(\G_0, X_0)$. Recall that $\Fr_\tau = \Art_E (\iota_\tau(\varpi))$ where $\iota_\tau \colon E_\tau \ra \mathbf{A}_E$ is the natural inclusion\footnote{We follow the usual normalization where arithmetic Frobenii correspond to 
inverses of uniformizers under the Artin map.}. We thus apply Lemma~\ref{lem:cocyc} for $\sigma = \Fr_\tau$ and $s = \iota_\tau(\varpi)$ to get 
$$
\Fr_\tau = \mathbf{1}_{\omega(\varpi) K_0} \circ f \circ \Fr_{\tau} \circ f^{-1}= f \circ \mathbf{1}_{\omega(\varpi) K_0} \circ \Fr_\tau \circ f^{-1}.  
$$
Here, we have used that the operators $\mathbf{1}_{\varpi K_{0}}$ and $\langle \varpi \rangle$ coincide on the level of the corresponding Shimura varieties and that $f$ and $\mathbf{1}_{\varpi K_{0}}$ commute.  We can thus 
write 
\begin{eqnarray*}
H_{\tau}(\Fr_\tau) &=& H_{\tau}(f \circ \mathbf{1}_{\omega(\varpi) K_0} \circ \Fr_\tau \circ f^{-1}) = f \circ H_{\tau}( \mathbf{1}_{\omega(\varpi) K_0} \circ \Fr_\tau) \circ f^{-1} = \\ 
&=& f \circ \mathbf{1}_{\varpi^n K_0} \circ H_{\omega, \tau}(  \Fr_\tau ) \circ f^{-1} = 0, 
\end{eqnarray*}
where we have used Lemma~\ref{lem:heckecomparison} and the assumption that $H_{\omega, \tau}(  \Fr_\tau ) = 0$.

\section{Distribution Relations and Proof of Theorem~\ref{thm:horizdistrel}}\label{sec:distrel}

\subsection{Distribution relations on $\Z[\Inv_\tau]$}

\paragraph{The action of the local Hecke algebra on $\Z[\Hyp_\tau]$ and $\Z[\Inv_\tau]$.}
The local Hecke algebra $\cH_\tau = \cH(G_\tau, K_\tau)$ acts on 
$\Z[\Hyp_\tau] = \Z[G_\tau / K_\tau]$ as follows: under the fixed 
identification $\Hyp_\tau \isom G_\tau / K_\tau$, the Hecke operator 
$\mathbf{1}_{K_\tau t K_\tau}$ acts on a hyperspecial point $[g K_\tau]$ via 
$[g K_\tau] \mapsto \sum_{\alpha} [ gg_\alpha K_\tau ]$, 
where $\ds K_\tau g K_\tau = \bigsqcup_{\alpha} g_\alpha K_\tau$ is the double coset decomposition. It is easy to check that this action is well-defined and in addition, it descends to an action of $\cH_\tau$ on $\Z[H_\tau \backslash \Hyp_\tau]  = \Z[\Inv_\tau]$. The latter action can be computed completely explicitly as summarized in the following lemma that is an easy consequence of the adjacency relations in the Bruhat--Tits trees from Figure~\ref{fig:buildings}:  

\begin{lem}\label{lem:local-hecke-inv}
We have 
$$
t_{1, 0} \cdot (a, b) = 
\begin{cases}
(a-1, b) + (q- 1) (a, b) + q^4 (a+1, b) & \text{if } a > 0, \\
(q^3-q)q(1,b) + q^2(0, b+1) + (q-1)(0, b) + (0, b-1) & \text{if } a = 0,\ b > 0, \\ 
(q^3-q)q(1,0) + q(q+1)(0, 1) & \text{if } a = b = 0. 
\end{cases}
$$
and 
$$
t_{0,1} \cdot (a, b) = 
\begin{cases}
q(q+1) (a, 1) & \text{if } b = 0,\\
(a, b-1) + (q-1) (a, b) + q^2 & \text{if } b > 0. 
\end{cases}
$$
\end{lem} 

\begin{rem}
There is a more algebraic proof of the lemma which uses explicit pair of lattices with invariants $(a, b) \in \Inv_\tau$ and applies the Hecke operators $t_{1,0}$ and $t_{0,1}$ directly to those. Fix any $s \in \cO_{E_\tau}$ such that $s + \overline{s} = 0$. 
The proof of Lemma~\ref{lem:transitive} shows that the Witt basis $\{ e_+, e_0, se_+ + e_0 + e_-\}$ 
determines a special apartment. Consider the lattice 
$$
L_V =\langle \varpi^a e_+, e_0, \varpi^{-a} (se_+ + e_0 + e_-)\rangle,  
$$
as well as the lattice $ \langle \varpi^{-b} e_+, e_0, \varpi^b(se_+ + e_0 + e_-)\rangle$. The latter can also be written as $L_W \oplus \cO_{E_\tau} e_0$ where $L_W = \langle \varpi^{-b} e_+, \varpi^b(se_+ + e_-)\rangle$. We easily compute that the pair $(L_V, L_W)$ satisfies $\inv_\tau (L_V, L_W) = (a, b)$. To apply $t_{1, 0}$ and $t_{0, 1}$, we decompose the double-cosets $K_V (\varpi, 1, \varpi^{-1})K_V$ and $K_W (\varpi, \varpi^{-1})K_W$ into disjoint union of right-cosets. This approach is useful when generalizing the distribution relation to arbitrary odd $n$ instead of $n = 3$ where it is difficult to draw the graphs. It is also used in the case when the allowable prime is split see~\cite{boumasmoud-brooks-jetchev:split}. 
\end{rem}

\paragraph{Computing the distribution relations on $\Z[\Inv_\tau]$.}

\begin{prop}\label{prop:distrelinv}
We have $H_\tau(1) \cdot (0, 0) \in q(q+1) \Z[\Inv_\tau]$. 
\end{prop}

\begin{proof}
The proof is a consequence of Lemma~\ref{lem:local-hecke-inv} combined with Theorem~\ref{thm:hecke}. More precisely, using the explicit form of the Hecke polynomial together with the action of the Hecke operators $t_{1, 0}$ and $t_{0, 1}$, we compute:  

\begin{eqnarray*}
H_\tau(1) \cdot (0, 0) &=& -(q - 1)q(q + 1)^2(q^6 - q^5 + 2q^4 - q^2 + q - 1)(q^6 - q^2 + 1)\cdot(0, 1) + \\
&+&  (q - 1)(q + 1)q^6(q^2 + q + 1)(q^6 - q^2 + 1) \cdot (1, 2) + \\
&+& (q - 1) q (q + 1)^2 (q^{14} - q^{13} + 2q^{12} - q^{11} - q^{10} + 3q^9 - 4q^8 + 2q^7 + 2q^6 - 4q^5 + \\ 
&+& 4q^4 - q^3 - 2q^2 + 2q - 1) \cdot (0, 0) - \\
&-& (q - 1) (q + 1)^2 q^{11} \cdot (2, 1) - \\ 
&-& (q - 1) (q + 1) q^3 (q^{13} + q^{12} + q^{10} - q^9 - q^8 + q^6 + 2q^4 - q^3 - 3q^2 + q + 1) (1,1) + \\ 
&+& (q - 1)(q + 1)^2 q^8 (q^5 - q^4 + q^3 - q + 1) \cdot (2, 0) + \\
&+& (q - 1) (q + 1) q^2 (q^{13} - q^{12} + 2q^{10} - 2q^9 + 2q^8 + 2q^7 - 4q^6 + q^5 - 2q^3 + 2q^2 + q - 1) \cdot (1, 0) + \\
&+& (q - 1)^2 (q + 1)^3 q^7 (q^2 + 1) \cdot (0, 3) - \\ 
&-& (q - 1)^2 q^3 (q + 1)^4 (q^2 + 1) (q^5 - q^4 + q^3 - q + 1) \cdot (0, 2). 
\end{eqnarray*}
Clearly, the latter is an element of $q(q+1) \Z[\Inv_\tau]$. 
\end{proof}

\subsection{The action of the Hecke algebra and the Galois action}
The Galois group $\Gal(E^{\ab} / E)$ acts on $\Z[\cZ_K(\G, \Hbf)]$ via the Shimura reciprocity law described in Section~\ref{subsec:shimvar}. The Hecke algebra $\cH(\G, K)$ of $K$-bi-invariant locally constant functions on $\G(\Af)$ acts on both $\Z[\cZ_K(\G, \Hbf)]$ and $\Z[\G(\Af) / K]$ and the action on the former is Galois equivariant. Recall that if $g' \in \G(\Af)$ then the function $\mathbf{1}_{K g'K} \in \cH(\G, K)$ acts as follows: if $\ds K g' K = \bigsqcup g_i K$ then 
\begin{equation}\label{eq:hecke-build}
\ds \mathbf{1}_{K g' K} (gK) = \sum g g_i K, \qquad g \in \G(\Af).  
\end{equation}
Similarly, 
\begin{equation}\label{eq:hecke-cycles}
\mathbf{1}_{K g' K} \left ( \cZ_K(g) \right )= \sum \cZ_K(g g_i), \qquad g \in \G(\Af). 
\end{equation}
We will prove Theorem~\ref{thm:horizdistrel} by relating the two actions locally at the fixed allowable inert place $\tau$, i.e., by relating the action of the local Hecke algebra $\cH(G_\tau, K_\tau)$ to the action of the decomposition group at $\tau$.

\subsection{Proof of Theorem~\ref{thm:horizdistrel}}

Let $\xi \in \cZ_K(\G, \Hbf)$ be the cycle from the statement of Theorem~\ref{thm:horizdistrel} (recall that $\int_\tau(\xi) = (0, 0)$). We consider $H_\tau(\Fr_\tau) \cdot \xi$ and look at the local invariants (at $\tau$) of the resulting linear combination of ``adjacent" cycles by using the distribution relation on invariants (Proposition~\ref{prop:distrelinv}). To 
verify that one gets an exact trace (down to $E_\tau$) of an element of $\Z[\cZ_K(\G, \Hbf)]$ whose local conductor at $\tau$ is $\varpi^2$, we will use the fact that the Galois and the Hecke actions commute. 

\paragraph{The action of the decomposition group.} Let $\xi_n \in \cZ_K(\G, \Hbf)$ be a cycle of local conductor $\bc_\tau(\xi) = \tau^n$ for some $n > 0$ (by abuse of notation, $\tau^n$ means the $n$th power of the fixed uniformizer of $E_\tau$). 
Given an integer $m$, $0 \leq m < n$, we define the local trace at of this cycle $\tau$ as  
$$
\Tr_{n, m}(\xi) := \sum_{x \in \cO_{m}^\times / \cO_{n}^\times} \xi^{\Art_{E_\tau}(x)},  
$$
where $\Art_{E_\tau} \colon E_\tau^\times \ra \Gal(E_\tau^{\ab} / E_\tau)$ is the local Artin 
map and we view $\Gal(E_\tau^{\ab} / E_\tau)$ as a subgroup of $\Gal(E^{\ab} / E)$ via the chosen embedding $\iota_\tau \colon \overline{E} \ra \overline{E}_\tau$. The above trace is the exact analogue of the trace $\Tr_{E[\tau^n] / E[\tau^m]}$ where $E[\tau^n]$ denotes the ring class field of conductor 
$\tau^n$ (by abuse of notation, $\tau$ denotes the prime ideal of $F$ corresponding to the place $\tau$). Note that these traces are used in the theory of Heegner points over the anticyclotomic towers (see also \cite[App.]{cornut-vatsal}).

\paragraph{Galois equivariance.}
Write $\ds H_\tau(1) \cdot \xi = \sum_{\xi' \in \cZ_K(\G, \Hbf)} c_{\xi'} \cdot \xi'$. By taking invariants on both sides and using the distribution relations on $\Z[\Inv_\tau]$ (Proposition~\ref{prop:distrelinv}), we obtain exactly the equality of Proposition~\ref{prop:distrelinv}. This allows us to write
\begin{eqnarray*}\label{eq:distrel}
H_\tau(1) \cdot \xi = \sum_{(a, b) \in \Inv_\tau} \sum_{x \in \cO_0^\times / \cO_{c}^\times}  m_{a, b}(x)\xi_{a, b}^{\Art_{E_\tau}(x)}, 
\end{eqnarray*}
where $c = \min(a, 2b)$. The latter is justified by the fact that for any $\sigma \in \Gal(E[c]_\tau / E_\tau)$ there exists $\sigma' \in \Gal(E[c]_\tau / E[1]_\tau)$ such that $\xi^{\sigma} = \xi^{\sigma'}$. This fact is not hard to deduce from the reciprocity laws from Section~\ref{par:concomp} and Section~\ref{par:galorbits-localcond}. It remains to argue that for fixed $(a, b) \in \Inv_\tau$, the multiplicities $m_{a, b}(x)$ are equal as $x$ varies over $\cO_{0}^\times / \cO_{c}^\times$. The latter follows from the fact that the action of the local Hecke action commutes with the action of the decomposition group. This is sufficient to deduce Theorem~\ref{thm:horizdistrel} from Proposition~\ref{prop:distrelinv}: for each $(a, b)$ for which $c > 0$ the inner sum is in the image of $\Tr_{2, 0}$ by the equality of the multiplicities; for each $(a, b)$ for which $c = 0$, $m_{a, b}(1) \xi_{a, b} \in \text{Im} (\Tr_{2, 0})$ since $m_{a, b}(1) \equiv 0 \bmod q(q+1) = \# \cO_{0}^\times / \cO_2^\times$ (again, by Proposition~\ref{prop:distrelinv}).

\section*{Acknowledgements}
I am grateful to Christophe Cornut for the numerous conversions and for sharing his profound knowledge about various aspects of the subject. I am indebted to Hunter Brooks for the careful reading of the draft and for the multiple helpful discussions and feedback. I thank Yiannis Sakellaridis for the numerous discussions and for helping me understand the Satake transform and various other aspects of the computation in a more conceptual manner.   
I thank Reda Boumasmoud, Henri Darmon, Benedict Gross, Ben Howard, Jean-Stefan Koskivirta, Philippe Michel, Paul Nelson, Jan Nekov\'{a}\v{r}, Richard Pink, Ken Ribet, Karl Rubin, Sug-Woo Shin, Chris Skinner, Xinyi Yuan, Shou-Wu Zhang and Wei Zhang for their valuable comments.

\bibliographystyle{amsalpha}
\bibliography{biblio}

\def\cprime{$'$}
\providecommand{\bysame}{\leavevmode\hbox to3em{\hrulefill}\thinspace}
\providecommand{\MR}{\relax\ifhmode\unskip\space\fi MR }
\providecommand{\MRhref}[2]{%
  \href{http://www.ams.org/mathscinet-getitem?mr=#1}{#2}
}
\providecommand{\href}[2]{#2}
\begin{thebibliography}{GGP09}

\bibitem[BBJ15]{boumasmoud-brooks-jetchev}
R.~Boumasmoud, H.~Brooks, and D.~Jetchev, \emph{Vertical distribution relations
  for special cycles on unitary {S}himura varieties}, {\tt
  http://arxiv.org/pdf/1512.00926v1.pdf} (2015).

\bibitem[BBJ16]{boumasmoud-brooks-jetchev:split}
\bysame, \emph{Horizontal distribution relations for special cycles on unitary
  {S}himura varieties: split case}, preprint (2016).

\bibitem[BR94]{blasius-rogawski:zeta}
D.~Blasius and J.~Rogawski, \emph{Zeta functions of {S}himura varieties},
  Motives ({S}eattle, {WA}, 1991), Proc. Sympos. Pure Math., vol.~55, Amer.
  Math. Soc., Providence, RI, 1994, pp.~525--571.

\bibitem[BT72]{bruhat-tits:1}
F.~Bruhat and J.~Tits, \emph{Groupes r\'eductifs sur un corps local}, Inst.
  Hautes \'Etudes Sci. Publ. Math. (1972), no.~41, 5--251.

\bibitem[BT87]{bruhat-tits:4}
\bysame, \emph{Sch\'emas en groupes et immeubles des groupes classiques sur un
  corps local. {II}. {G}roupes unitaires}, Bull. Soc. Math. France \textbf{115}
  (1987), no.~2, 141--195.

\bibitem[CGJ92]{cohen-goresky-ji}
D.~Cohen, M.~Goresky, and L.~Ji, \emph{On the {K}\"unneth formula for
  intersection cohomology}, Trans. Amer. Math. Soc. \textbf{333} (1992), no.~1,
  63--69.

\bibitem[Che80]{cheeger}
J.~Cheeger, \emph{On the {H}odge theory of {R}iemannian pseudomanifolds},
  Geometry of the {L}aplace operator ({P}roc. {S}ympos. {P}ure {M}ath., {U}niv.
  {H}awaii, {H}onolulu, {H}awaii, 1979), Proc. Sympos. Pure Math., XXXVI, Amer.
  Math. Soc., Providence, R.I., 1980, pp.~91--146.

\bibitem[Cor09]{cornut:normes1}
C.~Cornut, \emph{Normes {$p$}-adiques et extensions quadratiques}, Ann. Inst.
  Fourier (Grenoble) \textbf{59} (2009), no.~6, 2223--2254.

\bibitem[Cor10]{cornut:normes2}
\bysame, \emph{On $p$-adic norms and quadratic extensions {II}}, to appear in
  Manuscripta Mathematica (2010).

\bibitem[CV05]{cornut-vatsal}
C.~Cornut and V.~Vatsal, \emph{C{M} points and quaternion algebras}, Doc. Math.
  \textbf{10} (2005), 263--309 (electronic).

\bibitem[Del79]{deligne:shimura}
P.~Deligne, \emph{Vari\'et\'es de {S}himura: interpr\'etation modulaire, et
  techniques de construction de mod\`eles canoniques}, Automorphic forms,
  representations and {$L$}-functions ({P}roc. {S}ympos. {P}ure {M}ath.,
  {O}regon {S}tate {U}niv., {C}orvallis, {O}re., 1977), {P}art 2, Proc. Sympos.
  Pure Math., XXXIII, Amer. Math. Soc., Providence, R.I., 1979, pp.~247--289.

\bibitem[Gar97]{garrett:buildings}
P.~Garrett, \emph{Buildings and classical groups}, Chapman \& Hall, London,
  1997.

\bibitem[GGP09]{gan-gross-prasad}
W.-T. Gan, B.~Gross, and D.~Prasad, \emph{Symplectic local root numbers,
  central critical {L}-values, and restriction problems in the representation
  theory of classical groups}, preprint (2009).

\bibitem[GI63]{goldman-iwahori}
O.~Goldman and N.~Iwahori, \emph{The space of {$\p$}-adic norms}, Acta Math.
  \textbf{109} (1963), 137--177.

\bibitem[Gro84]{gross:heegner_points}
B.~Gross, \emph{Heegner points on {$X\sb 0(N)$}}, Modular forms (Durham, 1983),
  Ellis Horwood Ser. Math. Appl.: Statist. Oper. Res., Horwood, Chichester,
  1984, pp.~87--105.

\bibitem[Gro91]{gross:kolyvagin}
B.\thinspace{}H. Gross, \emph{Kolyvagin's work on modular elliptic curves},
  $L$-functions and arithmetic (Durham, 1989), Cambridge Univ. Press,
  Cambridge, 1991, pp.~235--256.

\bibitem[Gro98]{gross:satake}
B.~Gross, \emph{On the {S}atake isomorphism}, Galois representations in
  arithmetic algebraic geometry ({D}urham, 1996), London Math. Soc. Lecture
  Note Ser., vol. 254, Cambridge Univ. Press, Cambridge, 1998, pp.~223--237.

\bibitem[Gro04]{gross:heegner-representation}
\bysame, \emph{Heegner points and representation theory}, Heegner points and
  Rankin $L$-series, Math. Sci. Res. Inst. Publ., vol.~49, Cambridge Univ.
  Press, 2004, pp.~37--65.

\bibitem[Gro09]{gross:letter}
\bysame, \emph{A letter to {P}ierre {D}eligne (2009)}.

\bibitem[KL97]{kleiner-leeb}
B.~Kleiner and B.~Leeb, \emph{Rigidity of quasi-isometries for symmetric spaces
  and {E}uclidean buildings}, Inst. Hautes \'Etudes Sci. Publ. Math. (1997),
  no.~86, 115--197 (1998).

\bibitem[Kol90]{kolyvagin:euler_systems}
V.~A. Kolyvagin, \emph{Euler systems}, The Grothendieck Festschrift, Vol.\ II,
  Birkh\"auser Boston, Boston, MA, 1990, pp.~435--483.

\bibitem[Kos13]{koskivirta:thesis}
J.-S. Koskivirta, \emph{Relation de congruence pour les groupes unitaires
  $\text{GU}(n-1,1)$}, Ph.D. thesis, Universit\'e de Strasbourg (2013).

\bibitem[Kos14]{koskivirta:journal}
\bysame, \emph{Congruence relations for {S}himura varieties associated with
  {$GU(n-1,1)$}}, Canad. J. Math. \textbf{66} (2014), no.~6, 1305--1326.

\bibitem[Lan94]{lang:ant}
S.~Lang, \emph{Algebraic number theory}, second ed., Springer-Verlag, New York,
  1994.

\bibitem[Loo88]{looijenga}
E.~Looijenga, \emph{{$L^2$}-cohomology of locally symmetric varieties},
  Compositio Math. \textbf{67} (1988), no.~1, 3--20.

\bibitem[Mil05]{milne:shimura}
J.~S. Milne, \emph{Introduction to {S}himura varieties}, Harmonic analysis, the
  trace formula, and {S}himura varieties, Clay Math. Proc., vol.~4, Amer. Math.
  Soc., Providence, RI, 2005, pp.~265--378.

\bibitem[Par00]{parreau:immeuble}
A.~Parreau, \emph{Immeubles affines: construction par les normes et \'etude des
  isom\'etries}, Crystallographic groups and their generalizations ({K}ortrijk,
  1999), Contemp. Math., vol. 262, Amer. Math. Soc., Providence, RI, 2000,
  pp.~263--302.

\bibitem[RTZ13]{rapoport-terstiege-zhang}
M.~Rapoport, U.~Terstiege, and W.~Zhang, \emph{On the arithmetic fundamental
  lemma in the minuscule case}, Compos. Math. \textbf{149} (2013), no.~10,
  1631--1666.

\bibitem[Rub00]{rubin:book}
K.~Rubin, \emph{{E}uler {S}ystems}, Princeton University Press, Spring 2000,
  {A}nnals of {M}athematics {S}tudies {\bf 147}.

\bibitem[Sat63]{satake:spherical}
I.~Satake, \emph{Theory of spherical functions on reductive algebraic groups
  over {${\mathfrak p}$}-adic fields}, Inst. Hautes \'Etudes Sci. Publ. Math.
  (1963), no.~18, 5--69.

\bibitem[SS90]{saper-stern}
L.~Saper and M.~Stern, \emph{{$L_2$}-cohomology of arithmetic varieties}, Ann.
  of Math. (2) \textbf{132} (1990), no.~1, 1--69.

\bibitem[Wed00]{wedhorn:congruence}
T.~Wedhorn, \emph{Congruence relations on some {S}himura varieties}, J. Reine
  Angew. Math. \textbf{524} (2000), 43--71.

\bibitem[Zha12]{zhang:afl}
W.~Zhang, \emph{On arithmetic fundamental lemmas}, Invent. Math. \textbf{188}
  (2012), no.~1, 197--252.

\end{thebibliography}

\end{document}